\documentclass[10pt,a4paper]{amsart}
\usepackage{amssymb}
\usepackage{amscd}
\usepackage[all]{xy}

\newcounter{TmpEnumi}
\newcounter{RsvEnumi}

\numberwithin{equation}{section}

\def\today{\number\day\space\ifcase\month\or   January\or February\or
   March\or April\or May\or June\or   July\or August\or September\or
   October\or November\or December\fi\   \number\year}

\theoremstyle{definition}
\newtheorem{thm}{Theorem}[section]
\newtheorem{lem}[thm]{Lemma}
\newtheorem{prp}[thm]{Proposition}
\newtheorem{dfn}[thm]{Definition}
\newtheorem{cor}[thm]{Corollary}

\newtheorem{rmk}[thm]{Remark}
\newtheorem{ntn}[thm]{Notation}
\newtheorem{exa}[thm]{Example}
\newtheorem{pbm}[thm]{Problem}

\newtheorem{qst}[thm]{Question}

\newcommand{\af}{\alpha}
\newcommand{\bt}{\beta}
\newcommand{\gm}{\gamma}
\newcommand{\dt}{\delta}
\newcommand{\ep}{\varepsilon}
\newcommand{\zt}{\zeta}
\newcommand{\et}{\eta}

\newcommand{\io}{\iota}

\newcommand{\ld}{\lambda}
\newcommand{\sm}{\sigma}
\newcommand{\kp}{\kappa}
\newcommand{\ph}{\varphi}
\newcommand{\ps}{\psi}
\newcommand{\rh}{\rho}
\newcommand{\om}{\omega}
\newcommand{\ta}{\tau}

\newcommand{\Dt}{\Delta}

\newcommand{\Ld}{\Lambda}

\newcommand{\Om}{\Omega}

\newcommand{\Q}{{\mathbb{Q}}}

\newcommand{\C}{{\mathbb{C}}}
\newcommand{\N}{{\mathbb{Z}}_{> 0}}
\newcommand{\Nz}{{\mathbb{Z}}_{\geq 0}}

\newcommand{\inv}{{\mathrm{inv}}}

\newcommand{\id}{{\mathrm{id}}}

\newcommand{\diag}{{\mathrm{diag}}}

\newcommand{\sgn}{{\mathrm{sgn}}}

\newcommand{\card}{{\mathrm{card}}}

\newcommand{\dirlim}{\varinjlim}

\newcommand{\limi}[1]{\lim_{{#1} \to \infty}}

\newcommand{\andeqn}{\,\,\,\,\,\, {\mbox{and}} \,\,\,\,\,\,}

\newcommand{\ts}[1]{{\textstyle{#1}}}
\newcommand{\ds}[1]{{\displaystyle{#1}}}
\newcommand{\ssum}[1]{{\ts{ {\ds{\sum}}_{#1} }}}

\newcommand{\wolog}{without loss of generality}
\newcommand{\Wolog}{Without loss of generality}

\newcommand{\tfae}{the following are equivalent}
\newcommand{\ifo}{if and only if}

\newcommand{\ca}{C*-algebra}

\newcommand{\hm}{homomorphism}

\newcommand{\fd}{finite dimensional}

\newcommand{\ct}{continuous}

\newcommand{\rpn}{representation}

\newcommand{\msp}{measure space}

\newcommand{\sfm}{$\sm$-finite measure space}

\newcommand{\XBM}{(X, {\mathcal{B}}, \mu)}
\newcommand{\YCN}{(Y, {\mathcal{C}}, \nu)}
\newcommand{\cB}{{\mathcal{B}}}
\newcommand{\cC}{{\mathcal{C}}}
\newcommand{\LLp}{L (L^p (X, \mu))}

\newcommand{\OP}[2]{{\mathcal{O}}_{#1}^{#2}}
\newcommand{\MP}[2]{M_{#1}^{#2}}

\newcommand{\ran}{{\operatorname{ran}}}

\newcommand{\E}{\varnothing}

\title[Amenability]{Isomorphism, nonisomorphism, and amenability
  of $L^p$~UHF algebras}

\author{N.~Christopher Phillips}

\date{24~September 2013}

\address{Department of Mathematics, University  of Oregon,
       Eugene OR 97403-1222, USA.}

\email[]{ncp@darkwing.uoregon.edu}

\thanks{This material is based upon work supported by the
  US National Science Foundation under
  Grant DMS-1101742.}

\begin{document}

\begin{abstract}
In a previous paper,
we introduced $L^p$~UHF algebras for $p \in [1, \infty).$
We concentrated on the spatial $L^p$~UHF algebras,
which are classified up to isometric isomorphism
by $p$ and the scaled ordered $K_0$-group.
In this paper,
we concentrate on a larger class, the $L^p$~UHF algebras
of tensor product type constructed using diagonal similarities.
Such an algebra is still simple and has the same K-theory
as the corresponding spatial $L^p$~UHF algebra.
For each choice of $p$ and the K-theory,
we provide uncountably many nonisomorphic such algebras.
We further characterize the spatial algebras among them.
In particular, if $A$ is one of these algebras,
the following are equivalent:
\begin{itemize}
\item
$A$ is isomorphic (not necessarily isometrically)
to a spatial $L^p$~UHF algebra.
\item
$A$ is amenable as a Banach algebra.
\item
$A$ has approximately inner tensor flip.
\end{itemize}
These conditions are also equivalent to a natural numerical condition
defined in terms of the ingredients used to construct the algebra.
\end{abstract}

\maketitle

\indent
In~\cite{PhLp2a},
we introduced analogs of UHF \ca{s}
which act on $L^p$~spaces instead of on Hilbert space.
The purpose of this paper is to prove further results about such algebras.
For each $p \in [1, \infty)$
and choice of K-theory,
we prove the existence of uncountably many such algebras
with the given K-theory and which are not isomorphic,
even when the isomorphisms are not required to be isometric.
We give several conditions which characterize
the ``standard'' examples of~\cite{PhLp2a} within a larger class,
including amenability and approximate innerness of the tensor flip.

In~\cite{PhLp2a},
we were primarily interested in a particular family
of examples,
the spatial $L^p$~UHF algebras,
which were needed for the proof of simplicity
of the analogs of Cuntz algebras acting on $L^p$~spaces.
For $p = 2,$
a spatial $L^p$~UHF algebra is a UHF \ca.
For general $p \in [1, \infty),$
we showed that~$p,$
together with the scaled ordered $K_0$-group
(equivalently, the supernatural number),
is a complete invariant for both isomorphism
and isometric isomorphism of spatial $L^p$~UHF algebras.
We also proved that spatial $L^p$~UHF algebras
are amenable Banach algebras.
For a larger class of $L^p$~UHF algebras,
we proved simplicity and existence of a unique normalized trace.

Here,
we consider $L^p$~UHF algebras~$A$ of tensor product type,
a subclass of the class for which we proved simplicity in~\cite{PhLp2a}.
Tensor product type
means that there are
a sequence $d = (d (1), \, d (2), \, \ldots )$
in $\{ 2, 3, 4, \ldots \},$
probability spaces $(X_n, {\mathcal{B}}_n, \mu_n),$
and unital \rpn{s} $\rh_n \colon M_{d (n)} \to L (L^p (X_n, \mu_n))$
such that,
with $\XBM$ being the product measure space,
$A$ is the closed subalgebra of the bounded operators
on $L^p (X, \mu) = \bigotimes_{n = 1}^{\infty} L^p (X_n, \mu_n)$
generated by all operators of the form
\[
\rh_{1} (a_1) \otimes \rh_{2} (a_2)
     \otimes \cdots \otimes \rh_{m} (a_{m})
     \otimes 1 \otimes 1 \otimes \cdots
\]
with
\[
m \in \Nz, \,\,\,\,\,\,
a_1 \in M_{d (1)}, \,\,\,\,\,\, a_2 \in M_{d (2)},
  \,\,\,\,\,\, \cdots, \,\,\,\,\,\, a_{m} \in M_{d (m)}.
\]
For our strongest results,
we further require that the representations $\rh_n$
be direct sums of finitely or countably many representations
similar to the identity representation.
That is,
identifying $M_{d (m)}$ with the algebra
of bounded operators on a probability space with $d (m)$ points,
they are direct sums of maps $a \mapsto s a s^{-1}$
for invertible elements $s \in M_{d (m)}.$
We mostly further require that the matrices $s$ all be diagonal.

The K-theory of~$A$ is determined by the formal product
of the numbers $d (n),$
more precisely,
by the formal prime factorization in which the exponent
of a prime~$p$
is the sum of its exponents in the prime factorizations
of the numbers $d (n).$
To avoid the intricacies of K-theory,
we work in terms of this invariant,
called a ``supernatural number''
(Definition~\ref{D_2Y15_SNat} below),
instead.

Our convention is that
isomorphisms of Banach algebras are required to be \ct{}
and have \ct{} inverses,
but they are not required to be isometric.
To specify the more restrictive version,
we refer to an isometric isomorphism,
or say that two Banach algebras are isometrically isomorphic.

For the rest of the introduction,
fix $p \in [1, \infty)$ and a supernatural number~$N.$
For an $L^p$~UHF algebra $A \subset \LLp$
of tensor product type constructed using diagonal similarities
and with the given supernatural number~$N,$
we give (Theorem~\ref{T_3708_Third})
a number of equivalent conditions for isomorphism
with the spatial $L^p$~UHF algebra with supernatural number~$N.$
One of these is approximate innerness of the tensor flip
on the $L^p$~operator tensor product:
there is a bounded net $(v_{\ld})_{\ld \in \Ld}$
in the $L^p$~operator tensor product $A \otimes_p A$
such that $(v_{\ld}^{-1})_{\ld \in \Ld}$ is also bounded
and such that
$\lim_{\ld \in \Ld} v_{\ld} (a \otimes b) v_{\ld}^{-1} = b \otimes a$
for all $a, b \in A.$
(The C*-algebra version of this condition has been studied
in~\cite{EfR}.
It is rare.)

Another condition is numerical:
with $\rh_n$ as in the description above,
it is that $\sum_{n = 1}^{\infty} ( \| \rh_n \| - 1 ) < \infty.$
Moreover, $\| \rh_n \|$ is easily computed in terms of the similarities
used in the construction of~$A.$

Two further equivalent conditions are amenability
and symmetric amenability as a Banach algebra.
See Proposition~\ref{P_2Y25_ChrAm}
and the preceding discussion for more on these conditions.
Recall that a \ca{} is amenable \ifo{} it is nuclear.
(See Corollary~2 of~\cite{Cns} and Theorem~3.1 of~\cite{Haa}.)
In Theorem~2 of~\cite{Oz},
it is shown that a unital nuclear \ca{} is symmetrically amenable
\ifo{} every nonzero quotient of $A$ has a tracial state.

A fifth condition is that $A$ be similar to a spatial $L^p$~UHF algebra,
that is, that there be an invertible element $v \in \LLp$
such that $v A v^{-1}$
is a spatial $L^p$~UHF subalgebra of $\LLp.$
When $p = 2,$
our results imply that if $A$ is amenable
then it is in fact similar to a \ca.
In this case,
we can omit the requirement that the similarities
in the construction of our algebra be diagonal.
A question open for some time
(see Problem~30 in the ``Open Problems'' chapter of~\cite{Rnd})
asks whether an amenable closed subalgebra of the bounded operators
on a Hilbert space is similar to a \ca.
Little seems to be known.
Two main positive results are
for subalgebras of $K (H)$
(see the last paragraph of~\cite{Gfd})
and for commutative subalgebras of finite factors
(see~\cite{Ch}).
A claimed result for singly generated operator algebras~\cite{FFM1}
has been retracted~\cite{FFM2}.
If $A \subset L (H)$ is a closed amenable subalgebra,
then $A$ is already a \ca{}
if $A$ is generated by elements which are normal in $L (H)$
(\cite{CL})
or if $A$ is unital and $1$-amenable
(Theorem 7.4.18(2) of~\cite{BlLM}).
The answer is negative in the inseparable case~\cite{FrOz}.
The class of examples for which we give a positive answer,
the $L^2$~UHF algebras
of tensor product type constructed using similarities,
is small
(although it includes uncountably many isomorphism types
for every supernatural number~$N$),
but is quite different from any other classes of examples
for which the question was previously known
to have a positive answer.
The fact that the theorem holds for any $p \in [1, \infty)$
when the similarities are diagonal
suggests that there might be an interesting
$L^p$~analog of this question.

We have partial results in the same direction
for general $L^p$~UHF algebras of tensor product type.
We have no counterexamples to show that the results described above
fail in this broader class.
However, we also do not know how to deduce
anything from amenability
of a general $L^p$~UHF algebras of tensor product type,
or even from the stronger condition of symmetric amenability.
We do have a condition
related to approximate innerness of the tensor flip,
and a numerical condition involving norms of \rpn{s},
which for algebras in this class
imply isomorphism to the corresponding spatial algebra.

We further prove (Theorem~\ref{T_3326_Uctbl}) that
for every $p \in [1, \infty)$ and supernatural number~$N,$
there are uncountably many pairwise nonisomorphic
$L^p$~UHF algebras of tensor product type
constructed using diagonal similarities.
In particular, there are uncountably many nonisomorphic
such algebras which are not amenable.

This paper is organized as follows.
In Section~\ref{Sec_TPT},
we recall some material from~\cite{PhLp2a},
in particular,
the construction of $L^p$~UHF algebras of tensor product type.
In Section~\ref{Sec_Sptl},
we discuss amenability, symmetric amenability,
approximate innerness of the tensor flip,
and approximate innerness of the related tensor half flip.
For $L^p$~UHF algebras of tensor product type,
we prove the implications we can between these conditions,
several related conditions,
and isomorphism to the spatial $L^p$~UHF algebra.

In Section~\ref{Sec_ClEx},
we specialize to $L^p$~UHF algebras of tensor product type
constructed using systems of similarities.
The new feature is that the norms of the \rpn{s} involved are
more computable.
We prove a number of lemmas which will be used
in the remaining two sections.
Although we do not formally state it,
at the end we discuss a strengthening of
the main result of Section~\ref{Sec_Sptl}.

In Section~\ref{Sec_Amen} we further specialize to diagonal similarities.
One can then get good information on the norms of matrix units.
We prove the equivalence of a number of conditions
(some of them described above)
to isomorphism with the corresponding spatial $L^p$~UHF algebra.
Section~\ref{Sec_ManyNI} gives the construction of uncountably
many nonisomorphic algebras in this class.

We are grateful to
Volker Runde
for useful email discussions and for providing references,
and to Narutaka Ozawa for suggesting that we consider amenability
of our algebras.
Some of this work was done during a visit to
Tokyo University during December 2012.
We are grateful to that institution for its hospitality.

\section{$L^p$~UHF algebras of tensor product type}\label{Sec_TPT}

\indent
In this section,
we recall the construction
of $L^p$~UHF algebras of (infinite) tensor product type
(Example~3.8 of~\cite{PhLp2a}),
and
several other results and definitions of~\cite{PhLp2a}.
We start by recalling supernatural numbers,
which are the elementary form
of the isomorphism invariant for spatial $L^p$~UHF algebras.

\begin{dfn}[Definition~3.3 of~\cite{PhLp2a}]\label{D_2Y15_SNat}
Let $P$ be the set of prime numbers.
A {\emph{supernatural number}} is a function
$N \colon P \to \N \cup \{ \infty \}$
such that $\sum_{t \in P} N (t) = \infty.$

Let $d = (d (1), \, d (2), \, \ldots )$
be a sequence in $\{ 2, 3, 4, \ldots \}.$
We define
\[
r_d (n) = d (1) d (2) \cdots d (n)
\]
for $n \in \Nz.$
(Thus, $r_d (0) = 1.$)
We then define the {\emph{supernatural number associated with~$d$}}
to be the function
$N_d \colon P \to \Nz \cup \{ \infty \}$
given by
\[
N_d (t) = \sup \big( \big\{ k \in \Nz \colon
        {\mbox{there is $n \in \Nz$ such that $t^k$ divides $r_d (n)$}}
              \big\} \big).
\]
\end{dfn}

\begin{dfn}[Definition~3.4 of~\cite{PhLp2a}]\label{D_2Y14_pUHFTypeN-2}
Let $\XBM$ be a \sfm,
let $p \in [1, \infty),$
and let $A \subset \LLp$ be a unital subalgebra.
Let $N$ be a supernatural number.
We say that $A$ is an
{\emph{$L^p$~UHF algebra of type~$N$}}
if there exist a sequence $d$ as in Definition~\ref{D_2Y15_SNat}
with $N_d = N,$
unital subalgebras $D_0 \subset D_1 \subset \cdots \subset A,$
and algebraic isomorphisms $\sm_n \colon M_{r_d (n)} \to D_n,$
such that $A = {\overline{\bigcup_{n = 0}^{\infty} D_n}}.$
\end{dfn}

\begin{ntn}\label{N_3730_Mp}
For $d \in \N$ and $p \in [1, \infty],$
we let $l^p_d = l^p \big( \{1, 2, \ldots, d \} \big),$
using {\emph{normalized}}
counting measure on $\{1, 2, \ldots, d \},$
that is, the total mass is~$1.$
We further let $\MP{d}{p} = L \big( l_d^p \big)$
with the usual operator norm,
and we algebraically identify $\MP{d}{p}$ with the algebra $M_d$ of
$d \times d$ complex matrices in the standard way.
\end{ntn}

One gets the same normed algebra~$\MP{d}{p}$
if one uses the usual counting measure.
(We make a generalization explicit in Lemma~\ref{L_3903_ChM} below.)

Many articles on Banach spaces use $L_p (X, \mu)$
rather than $L^p (X, \mu),$
and use $l_p^d$ for what we call~$l_d^p.$
Our convention is chosen
to avoid conflict with the standard notation for Leavitt algebras,
which appear briefly here
and play a major role in the related papers
\cite{PhLp1}, \cite{PhLp2a}, and~\cite{PhLp3}.

We now recall the construction of
$L^p$~UHF algebras of tensor product type
(Example~3.8 of~\cite{PhLp2a}).
See~\cite{PhLp2a}
for further details and for justification of the statements made here,
in particular,
for the tensor product decomposition
of $L^p$ of a product of measure spaces.

\begin{exa}[Example~3.8 of~\cite{PhLp2a}]\label{E_2Y14_LpUHF}
Let $p \in [1, \infty).$
We take ${\mathbb{N}} = \N.$
For each $n \in {\mathbb{N}},$
let $(X_n, {\mathcal{B}}_n, \mu_n)$ be a probability space,
let $d (n) \in \{ 2, 3, \ldots \},$
and let $\rh_n \colon M_{d (n)} \to L (L^p (X_n, \mu_n))$
be a representation
(unital, by our conventions).

For every subset $S \subset {\mathbb{N}},$
let $(X_S, {\mathcal{B}}_S, \mu_S)$ be the product measure space
of the spaces $(X_n, {\mathcal{B}}_n, \mu_n)$ for $n \in S.$
We let $1_S$ denote the identity operator on $L^p (X_S, \mu_S).$
% Thus $X_{\varnothing}$ is a one point space with counting
% measure.
For $S \subset {\mathbb{N}}$ and $n \in \Nz$ we take
\[
S_{\leq n} = S \cap \{ 1, 2, \ldots, n \}
\andeqn
S_{> n} = S \cap \{ n + 1, \, n + 2, \ldots \}.
\]
We make the identification
\begin{equation}\label{Eq_2Y14_2Fact}
L^p (X_S, \mu_S)
  = L^p (X_{S_{\leq n}}, \, \mu_{S_{\leq n}})
   \otimes_p L^p (X_{S_{> n}}, \, \mu_{S_{> n}}).
\end{equation}

Suppose now that $S$ is finite.
Set $l (S) = \card (S),$
and
write
\[
S = \{ m_{S, 1}, \, m_{S, 2}, \, \ldots, m_{S, l (S)} \}
\]
with $m_{S, 1} < m_{S, 2} < \cdots < m_{S, l (S)}.$
We make the identification
\[
L^p (X_S, \mu_S)
 = L^p (X_{m_{S, 1}}, \mu_{m_{S, 1}})
   \otimes_p L^p (X_{m_{S, 2}}, \mu_{m_{S, 2}})
   \otimes_p \cdots
   \otimes_p L^p (X_{m_{S, l (S)}}, \mu_{m_{S, l (S)}}).
\]
Set
\[
d (S) = \prod_{j = 1}^{l (S)} d (m_{S, j})
\andeqn
M_S = \bigotimes_{j = 1}^{l (S)} M_{d (m_{S, j})}
    \subset L^p (X_S, \mu_S).
\]
We take $d (\E) = 1$ and $M_{\E} = \C.$
Then $M_S \cong M_{d (S)}.$
Further let $\rh_S \colon M_S \to L (L^p (X_S, \mu_S))$
be the unique representation such that
for
\[
a_1 \in M_{d (m_{S, 1})}, \, \, a_2 \in M_{d (m_{S, 2})},
  \, \, \cdots, \, \, a_{l (S)} \in M_{d (m_{S, l (S)})},
\]
we have
\[
\rh_S (a_1 \otimes a_2 \otimes \cdots \otimes a_{l (S)})
  = \rh_{d (m_{S, 1})} (a_1) \otimes \rh_{d (m_{S, 2})} (a_2)
     \otimes \cdots \otimes \rh_{d (m_{S, l (S)})} (a_{l (S)}).
\]

For finite sets $S \subset T \subset {\mathbb{N}},$
there is an obvious \hm{} $\ph_{T, S} \colon M_S \to M_T,$
obtained by filling in a tensor factor of~$1$ for every element
of $T \setminus S.$
We then define
$\rh_{T, S} = \rh_T \circ \ph_{T, S} \colon
   M_S \to L (L^p (X_T, \mu_T)).$
When $S$ is finite but $T$ is not,
we define a \rpn{} of $M_S$
on $L^p (X_T, \mu_T)$ as follows.
Choose some $n \geq \sup (S),$
and, following~(\ref{Eq_2Y14_2Fact})
and
% Theorem~\ref{T-LpTP}(\ref{T-LpTP-3a}),
Theorem~2.16(3) of~\cite{PhLp1},
for $a \in M_S$ set
\[
\rh_{T, S} (a) = \rh_{T_{\leq n}, S} (a) \otimes 1_{T_{> n}}
  \in L (L^p (X_T, \mu_T)).
\]

We equip $M_S$ with the norm
$\| a \| = \| \rh_S (a) \|$
for all $a \in M_S.$
Then the maps $\ph_{T, S},$
for $S \subset T \subset {\mathbb{N}}$ finite,
and $\rh_{T, S},$
for $S \subset T \subset {\mathbb{N}}$ with $S$ finite,
are all isometric.
For $S \subset {\mathbb{N}}$ finite,
we now define $A_S \subset \LLp$
by $A_S = \rh_{{\mathbb{N}}, S} (M_S),$
and for $S \subset {\mathbb{N}}$ infinite,
we set
$A_S = {\overline{\bigcup_{n = 0}^{\infty} A_{S_{\leq n}} }}.$
In a similar way, we define $A_{T, S} \subset L (L^p (X_T, \mu_T)),$
for arbitrary $S, T$ with $S \subset T \subset {\mathbb{N}}.$

The algebra $A = A_{{\mathbb{N}}}$
is an $L^p$~UHF algebra of type~$N_d$
in the sense of Definition~\ref{D_2Y14_pUHFTypeN-2}.
When the ingredients used in its construction need to be specified,
we set $d = (d (1), d (2), \ldots)$
and $\rh = (\rh_1, \rh_2, \ldots),$
and write $A (d, \rh).$
We also take $\sm_n \colon M_{r_d (n)} \to A$
to be $\sm_n = \rh_{{\mathbb{N}}, {\mathbb{N}}_{\leq n} }.$
\end{exa}

The basic representations of $\MP{d}{p}$ are the
spatial \rpn{s}.
For the purposes of this paper,
we use the following condition for a \rpn{} to be spatial.

\begin{prp}\label{P_3910_SpMd}
Let $p \in [1, \infty),$
let $d \in \N,$
let $\XBM$ be a \sfm,
and let $\rh \colon \MP{d}{p} \to \LLp$
be a unital \hm.
Let $(e_{j, k})_{j, k = 1, 2, \ldots, d}$
be the standard system of matrix units for~$\MP{d}{p}.$
Then $\rh$ is a spatial \rpn{}
\ifo{} there is a measurable partition
$X = \coprod_{j = 1}^d X_j$
such that for $j, k = 1, 2, \ldots, d$ the operator
$\rh (e_{j, k})$ is a spatial partial isometry,
in the sense of Definition~6.4 of~\cite{PhLp1},
with domain support~$X_k$
and range support~$X_j.$
\end{prp}

\begin{proof}
For $p \neq 2,$
this statement easily follows by combining several parts
of Theorem~7.2 of~\cite{PhLp1}.
For $p = 2$
(and in fact for any $p \in [1, \infty)$),
it is also easily proved directly.
\end{proof}

\begin{cor}\label{C_3910_TensSp}
Let $p \in [1, \infty),$
let $k, l \in \N,$
let $\XBM$ and $\YCN$ be \sfm{s},
and let $\rh \colon \MP{k}{p} \to \LLp$
and $\sm \colon \MP{l}{p} \to L (L^p (Y, \nu))$
be spatial \rpn{s}.
Identify $L^p (X, \mu) \otimes_p L^p (Y, \nu)$ with
$L^p (X \times Y, \, \mu \times \nu)$
as in Theorem~2.16 of~\cite{PhLp1}
and $\MP{k}{p} \otimes_p \MP{l}{p}$ with $\MP{k l}{p}$
as in Corollary~1.13 of~\cite{PhLp2a}.
Then
$\rh \otimes \sm \colon \MP{k l}{p}
   \to L \big( L^p (X \times Y, \, \mu \times \nu) \big)$
is a spatial \rpn.
\end{cor}

\begin{proof}
By Lemma~6.20 of~\cite{PhLp1},
the tensor product of spatial partial isometries is again
a spatial partial isometry,
with domain and range supports given as the products
of the domain and range supports of the tensor factors.
The result is then immediate from Proposition~\ref{P_3910_SpMd}.
\end{proof}

For $p \in [1, \infty) \setminus \{ 2 \},$
a \rpn{} of $M_d$ on an $L^p$~space
is spatial \ifo{} it is isometric
(see Theorem 7.2(3) of~\cite{PhLp1}),
but this is not true for $p = 2.$

We briefly recall
that a \rpn{} of~$M_d$
has enough isometries
(Definition 2.7(1) of~\cite{PhLp2a})
if there is an irreducibly acting subgroup of its invertible
group $\inv (M_d)$
whose images are isometries,
that it locally has enough isometries
(Definition 2.7(2) of~\cite{PhLp2a})
if it is a direct sum
(in the sense of Definition~\ref{D_3406_pSumDfn} below)
of \rpn{s}
with enough isometries
(presumably using different subgroups for the different summands),
and that it dominates the spatial \rpn{}
(Definition 2.7(3) of~\cite{PhLp2a})
if in addition one of the summands is spatial.

\begin{dfn}[Definition~3.9 of~\cite{PhLp2a}]\label{D_2Y15_ITP}
Let $\XBM$ be a \sfm,
let $p \in [1, \infty),$
and let $A \subset \LLp$ be a unital subalgebra.
\begin{enumerate}
\item\label{D_2Y15_ITP_Basic}
We say that $A$ is an
{\emph{$L^p$~UHF algebra of tensor product type}}
if there exist $d$ and $\rh = (\rh_1, \rh_2, \ldots)$
as in Example~\ref{E_2Y14_LpUHF}
such that $A$ is isometrically isomorphic to $A (d, \rh).$
\item\label{D_2Y15_ITP_Spatial}
We say that $A$ is a {\emph{spatial $L^p$~UHF algebra}}
if in addition it is possible to choose
each \rpn{} $\rh_n$ to be spatial in the sense
of Definition~7.1 of~\cite{PhLp1}.
\item\label{D_2Y15_ITP_LEI}
We say that $A$ {\emph{locally has enough isometries}}
if it is possible to choose $d$ and $\rh$
as in~(\ref{D_2Y15_ITP_Basic})
such that, in addition,
$\rh_n$ locally has enough isometries,
in the sense of
Definition 2.7(2) of~\cite{PhLp2a},
for all $n \in \N.$
\item\label{D_2Y15_DomSpatial}
If $A$ locally has enough isometries,
we further say that $A$
{\emph{dominates the spatial representation}}
if it is possible to choose $d$ and $\rh$
as in~(\ref{D_2Y15_ITP_LEI})
such that, in addition,
for all $n \in \N$ the representation
$\rh_n$ dominates the spatial representation
in the sense of
Definition 2.7(3) of~\cite{PhLp2a}.
\end{enumerate}
\end{dfn}

In particular,
if we take the infinite tensor product of
the algebras $\MP{d (n)}{p},$
represented as in Notation~\ref{N_3730_Mp},
the resulting $L^p$~UHF algebra of tensor product type is spatial.

\begin{thm}\label{T_3729_EU}
Let $p \in [1, \infty)$
and let $N$ be a supernatural number.
Then there exists a spatial $L^p$~UHF algebra~$A$ of type~$N.$
It is unique up to isometric isomorphism.
It is of tensor product type,
in fact isometrically isomorphic
to $A (d, \rh)$ as in Example~\ref{E_2Y14_LpUHF}
for any sequence $d$ with $N_d = N$
and any sequence $\rh$ consisting of spatial \rpn{s}.
Moreover, for any \sfm{} $(Z, {\mathcal{D}}, \ld)$
and any closed unital subalgebra
$D \subset L (L^p (Z, \ld)),$
any two choices of spatial direct system of type~$N$
(as in Definition~3.5 of~\cite{PhLp2a}),
with any unital isometric representations
of the direct limits on $L^p$~spaces,
give isometrically isomorphic $L^p$~tensor products
$A \otimes_p D.$
\end{thm}

\begin{proof}
For all but the last two sentences, see Theorem 3.10 of~\cite{PhLp2a}.
For the second last sentence,
the only additional fact needed is that
$A (d, \rh)$ is a spatial $L^p$~UHF algebra.
This follows from Definition~3.5 of~\cite{PhLp2a}
and Corollary~\ref{C_3910_TensSp}.

We prove the last statement.
Let $\XBM$ and $\YCN$ be \sfm{s}.
Let $A \subset \LLp$ and $B \subset L (L^p (Y, \nu))$
be spatial $L^p$~UHF algebras of type~$N.$
Let $A_0 \subset A_1 \subset \cdots \subset A$
and $B_0 \subset B_1 \subset \cdots \subset B$
be subalgebras such that,
with $\ph_{n, m} \colon A_m \to A_n$
and $\ps_{n, m} \colon B_m \to B_n$
being the inclusion maps for $m \leq n,$
the systems
$\big( (A_n)_{n \in \Nz}, \, ( \ph_{n, m} )_{m \leq n} \big)$
and $\big( (B_n)_{n \in \Nz}, \, ( \ps_{n, m} )_{m \leq n} \big)$
are spatial direct systems of type~$N,$
and such that $A = {\overline{\bigcup_{n = 0}^{\infty} A_n}}$
and $B = {\overline{\bigcup_{n = 0}^{\infty} B_n}}.$
Then
\[
A \otimes_p D
 = {\overline{\bigcup_{n = 0}^{\infty} A_n \otimes_p D}}
 \cong \dirlim A_n \otimes_p D
\andeqn
B \otimes_p D
 = {\overline{\bigcup_{n = 0}^{\infty} B_n \otimes_p D}}
 \cong \dirlim B_n \otimes_p D,
\]
with the isomorphisms being isometric.
Theorem~3.7 of~\cite{PhLp2a}
shows that the direct limits
are isometrically isomorphic.
\end{proof}

\section{Spatial $L^p$~UHF algebras}\label{Sec_Sptl}

\indent
In this section,
we consider $L^p$~UHF algebras
of tensor product type
for a fixed value of~$p$ and a fixed supernatural number.
We give some conditions which imply that
such an algebra is isomorphic to the corresponding spatial algebra.
We also prove that the spatial $L^p$~UHF algebras
are symmetrically amenable,
hence amenable in the sense of Banach algebras,
and that they have approximately inner tensor flip.
The machinery we develop
will be used for the more special classes
considered in later sections.

We will need the projective tensor product of Banach spaces.
Many of its properties are given in Section B.2.2 of
Appendix~B in~\cite{Rnd}.
There are few proofs, but the omitted proofs are easy.

\begin{ntn}\label{N_2Y25_ProjTP}
Let $E$ and $F$ be Banach spaces.
We denote by $E {\widehat{\otimes}} F$
the projective tensor product of Definition B.2.10 of~\cite{Rnd},
and similarly for more than two factors.
When the norm on this space must be made explicit to avoid confusion,
we denote it by $\| \cdot \|_{\pi}.$
(It is given on an element $\mu$ in the
algebraic tensor product $E \otimes_{\mathrm{alg}} F$
by the infimum of all sums
$\sum_{k = 1}^n \| \xi_k \| \cdot \| \et_k \|$
for which $\xi_1, \xi_2, \ldots, \xi_n \in E$
and $\et_1, \et_2, \ldots, \et_n \in F$
satisfy $\sum_{k = 1}^n \xi_k \otimes \et_k = \mu.$)
If $E_1,$ $E_2,$ $F_1,$ and $F_2$ are Banach spaces,
$a_1 \in L \big( E_1, F_1),$ and $a_2 \in L (E_2, F_2 \big),$
we write $a {\widehat{\otimes}} b$
for the tensor product map
in
$L \big( E_1 {\widehat{\otimes}} E_2, \,
           F_1 {\widehat{\otimes}} F_2 \big).$
\end{ntn}

The projective tensor product is the maximal tensor product,
in the sense that there is a contractive linear map
from $E {\widehat{\otimes}} F$
to the completion of $E \otimes_{\mathrm{alg}} F$ 
in any other cross norm.
(See Exercise B.2.9 of~\cite{Rnd}.)
The projective tensor product of Banach algebras $A$ and $B$
is a Banach algebra (Exercise B.2.14 of~\cite{Rnd}),
in which, by the definition of a cross norm,
$\big\| a {\widehat{\otimes}} b \big\| = \| a \| \cdot \| b \|$
for $a \in A$ and $b \in B.$

The main purpose of the following lemma
is to establish notation for the maps in the statement.

\begin{lem}\label{L_2Y25_PjTMaps}
Let $A$ be a Banach algebra.
Then there are unique contractive linear maps
$\Dt_A \colon A {\widehat{\otimes}} A \to A$
and
$\Dt_A^{\mathrm{op}} \colon A {\widehat{\otimes}} A \to A$
such that $\Dt_A (a \otimes b) = a b$ and
$\Dt_A^{\mathrm{op}} (a \otimes b) = b a$ for all $a, b \in A.$
\end{lem}

\begin{proof}
This is immediate from the standard properties
of the projective tensor product.
\end{proof}

\begin{lem}\label{L_3804_DtFnc}
Let $A$ and $B$ be Banach algebras,
and let $\ph \colon A \to B$ be a \ct{} \hm.
Then
\[
\Dt_B \circ \big( \ph {\widehat{\otimes}} \ph \big)
  = \ph \circ \Dt_A
\andeqn
\Dt_B^{\mathrm{op}} \circ \big( \ph {\widehat{\otimes}} \ph \big)
  = \ph \circ \Dt_A^{\mathrm{op}}.
\]
\end{lem}

\begin{proof}
Use the relation $\ph (x y) = \ph (x) \ph (y)$ for $x, y \in A.$
\end{proof}

The definition of an amenable Banach algebra is given in
Definition 2.1.9 of~\cite{Rnd};
see Theorem 2.2.4 of~\cite{Rnd}
for two standard equivalent conditions.
We will also use symmetric amenability
(introduced in~\cite{Jh0}).
The following characterizations for unital Banach algebras
are convenient here.

\begin{prp}\label{P_2Y25_ChrAm}
Let $A$ be a unital Banach algebra.
Then:
\begin{enumerate}
\item\label{P_2Y25_ChrAm_Amen}
$A$ is amenable \ifo{} there is a
bounded net $(m_{\ld})_{\ld \in \Ld}$
in $A {\widehat{\otimes}} A$
such that
%
% \begin{equation}\label{Eq_2Y25_ChrAm_Amen}
\[
\lim_{\ld \in \Ld}
    \big( (a \otimes 1) m_{\ld} - m_{\ld} (1 \otimes a) \big)
  = 0
\]
% \end{equation}
%
for all $a \in A$
and such that $\lim_{\ld \in \Ld} \Dt_A (m_{\ld}) = 1.$
\item\label{P_2Y25_ChrAm_Sym}
$A$ is symmetrically amenable \ifo{} there is a bounded net
$(m_{\ld})_{\ld \in \Ld}$
in $A {\widehat{\otimes}} A$
such that
%
% \begin{equation}\label{Eq_2Y25_ChrAm_Sym}
\[
\lim_{\ld \in \Ld}
    \big( (a \otimes b) m_{\ld} - m_{\ld} (b \otimes a) \big)
  = 0
\]
% \end{equation}
%
for all $a, b \in A$
and such that
$\lim_{\ld \in \Ld} \Dt_A (m_{\ld})
 = \lim_{\ld \in \Ld} \Dt_A^{\mathrm{op}} (m_{\ld}) = 1.$
\end{enumerate}
\end{prp}

\begin{proof}
Part~(\ref{P_2Y25_ChrAm_Amen}) is Theorem 2.2.4 of~\cite{Rnd}.
We have taken advantage of the fact that $A$ is unital
to rewrite the definition of the module structures
on $A {\widehat{\otimes}} A$ used there in terms of the multiplication
in $A {\widehat{\otimes}} A$
and to simplify the condition involving~$\Dt_A.$

With the same modifications,
Part~(\ref{P_2Y25_ChrAm_Sym}) follows from
Proposition~2.2 of~\cite{Jh0}.
(In~\cite{Jh0},
the commutation relation in~(\ref{P_2Y25_ChrAm_Sym})
is stated separately for $a \otimes 1$ and $1 \otimes b,$
but the combination of those is clearly equivalent to our condition.)
\end{proof}

We will need the diagonal in $M_d$ from
Example 2.2.3 of~\cite{Rnd}
(used in Example 2.3.16 of~\cite{Rnd}),
and called ${\mathbf{m}}$ there.

\begin{lem}\label{L_2Y25_DiagMn}
Let $d \in \N,$
let $(e_{j, k})_{j, k = 1, 2, \ldots, d}$
be the standard system of matrix units for~$M_d,$
and let $y_d \in M_d \otimes M_d$ be given by
$y_d = \frac{1}{d} \sum_{r, s = 1}^d e_{r, s} \otimes e_{s, r}.$
Then:
\begin{enumerate}
\item\label{L_2Y25_DiagMn_Sq1}
$y_d^2 = d^{-2} \cdot 1.$
\item\label{L_2Y25_DiagMn_Conj}
$y_d (a \otimes b) y_d^{-1} = b \otimes a$ for all $a, b \in M_d.$
\item\label{L_2Y25_DiagMn_Dt}
$\Dt_{M_d} (y_d) = \Dt_{M_d}^{\mathrm{op}} (y_d) = 1.$
\item\label{L_2Y25_DiagMn_FromGp}
For any finite subgroup $G \subset \inv (M_d)$
whose natural action on $\C^d$ is irreducible,
we have
\[
y_d = \frac{1}{\card (G)} \sum_{g \in G} g \otimes g^{-1}.
\]
\item\label{L_2Y25_DiagMn_Uniq}
If $x \in M_d$ satisfies
$x (a \otimes b) = (b \otimes a) x$ for all $a, b \in M_d$
and $\Dt_{M_d} (x) = 1,$
then $x = y_d.$
\item\label{L_2Y25_DiagMn_Hom}
Let $\ph \colon M_d \to M_d$ be an automorphism.
Then $( \ph \otimes \ph) (y_d) = y_d.$
\item\label{L_2Y25_DiagMn_Norm}
For every $p \in [1, \infty),$
we have $\| y_d \|_{\pi} = 1$
in $\MP{d}{p} {\widehat{\otimes}} \MP{d}{p}.$
\item\label{L_2Y25_DiagMn_pNm}
For every $p \in [1, \infty),$
we have $\| y_d \|_p = d^{- 1}$
in $\MP{d}{p} \otimes_p \MP{d}{p}.$
\item\label{L_2Y25_DiagMn_d1d2}
Let $d_1, d_2 \in \N.$
Let
\[
\ph \colon (M_{d_1} \otimes M_{d_2}) \otimes (M_{d_1} \otimes M_{d_2})
    \to (M_{d_1} \otimes M_{d_1}) \otimes (M_{d_2} \otimes M_{d_2})
\]
be the \hm{} determined by
\[
\ph (a_1 \otimes a_2 \otimes b_1 \otimes b_2)
  = a_1 \otimes b_1 \otimes a_2 \otimes b_2
\]
for $a_1, b_1 \in M_{d_1}$ and $a_2, b_2 \in M_{d_2}.$
Then $\ph ( y_{d_1 d_2} ) = y_{d_1} \otimes y_{d_2}.$
\end{enumerate}
\end{lem}

\begin{proof}
Parts (\ref{L_2Y25_DiagMn_Sq1})
and~(\ref{L_2Y25_DiagMn_Dt})
are computations.
Part~(\ref{L_2Y25_DiagMn_Sq1})
shows that $y_d$ is invertible.
Part~(\ref{L_2Y25_DiagMn_Conj}) is then also a computation,
best done by taking both $a$ and $b$ to be standard matrix units.
For~(\ref{L_2Y25_DiagMn_Uniq}),
let $x$ be as there.
Then $y_d^{-1} x$ commutes with every element of $M_d \otimes M_d$
by~(\ref{L_2Y25_DiagMn_Conj}).
Since this algebra has trivial center,
it follows that there is $\ld \in \C$ such that $x = \ld y_d.$
Then $\ld = \ld \Dt_{M_d} (y_d) = \Dt_{M_d} (x) = 1.$
To prove~(\ref{L_2Y25_DiagMn_Hom}),
check that $( \ph \otimes \ph) (y_d)$ satisfies the conditions on~$x$
in~(\ref{L_2Y25_DiagMn_Uniq}).
(For the second condition,
one will need Lemma~\ref{L_3804_DtFnc}.)

Now let $G$ be as in part~(\ref{L_2Y25_DiagMn_FromGp}).
Set $x = \sum_{g \in G} g \otimes g^{-1}.$
Then, using~(\ref{L_2Y25_DiagMn_Conj}) at the second step
and the fact that $G$ is a group at the third step,
we have
\begin{equation}\label{Eq_2Z18xyd}
y_d x = \sum_{g \in G} y_d (g \otimes g^{-1})
      = \sum_{g \in G} (g^{-1} \otimes g) y_d
      = x y_d.
\end{equation}
We regard $G \times G$ as a subgroup of $\inv (M_d \otimes M_d)$
via $(h, k) \mapsto h \otimes k$ for $h, k \in G.$
Let $h, k \in G.$
Then, using the change of variables $g \mapsto k g^{-1} h^{-1}$
at the third step and~(\ref{Eq_2Z18xyd}) at the last step,
we have
\begin{align*}
y_d x (h \otimes k)
& = \sum_{g \in G} y_d (g h \otimes g^{-1} k)
\\
& = \sum_{g \in G} (g^{-1} k \otimes g h) y_d
  = \sum_{g \in G} (h g \otimes k g^{-1}) y_d
  = (h \otimes k) x y_d
  = (h \otimes k) y_d x.
\end{align*}
The natural action of $G \times G$ on $\C^d \otimes \C^d$
is irreducible by
Corollary~2.9 of~\cite{PhLp2a}.
Therefore $y_d x$ is a scalar.
So $x = y_d (y_d x)$ is a scalar multiple of~$y_d.$
Computing $\Dt (x) = \card (G) \cdot 1,$
we get $x = \card (G) y_d.$

We now prove part~(\ref{L_2Y25_DiagMn_Norm}).
Since $\Dt_{M_d} (y_d) = 1$ and $\| \Dt_{M_d} \| \leq 1,$
it is clear that $\| y_d \|_{\pi} \geq 1.$
For the reverse inequality,
let $G \subset \inv (M_d)$ be the group of signed permutation matrices.
Using
Lemma~2.11 of~\cite{PhLp2a}
and part~(\ref{L_2Y25_DiagMn_FromGp}),
we get
\[
\| y_d \|_{\pi}
  = \Bigg\| \frac{1}{\card (G)}
            \sum_{g \in G} g \otimes g^{-1} \Bigg\|_{\pi}
  \leq \frac{1}{\card (G)} \sum_{g \in G} \| g \| \cdot \| g^{-1} \|
  = 1.
\]
This proves~(\ref{L_2Y25_DiagMn_Norm}).
For~(\ref{L_2Y25_DiagMn_pNm}),
we identify $\MP{d}{p} \otimes_p \MP{d}{p} = \MP{d^2}{p}$
following
Corollary~1.13 of~\cite{PhLp2a}.
Then $d \cdot y_d$ becomes a permutation matrix,
so that $\| d \cdot y_d \| = 1.$

Part~(\ref{L_2Y25_DiagMn_d1d2}) follows easily by verifying
the conditions in part~(\ref{L_2Y25_DiagMn_Uniq}).
\end{proof}

\begin{dfn}\label{D_2Y25_AIFlip}
Let $\XBM$ be a \sfm,
let $p \in [1, \infty),$
and let $A \subset \LLp$ be a unital subalgebra.
\begin{enumerate}
\item\label{D_2Y25_AIFlip_HF}
We say that $A$ has {\emph{approximately inner $L^p$-tensor half flip}}
(with constant~$M$)
if there exists a net
$(v_{\ld})_{\ld \in \Ld}$ in $\inv (A \otimes_p A)$
such that $\| v_{\ld} \| \leq M$
and $\big\| v_{\ld}^{-1} \big\| \leq M$
for all $\ld \in \Ld,$
and such that
$\lim_{\ld \in \Ld} v_{\ld} (a \otimes 1) v_{\ld}^{-1} = 1 \otimes a$
for all $a \in A.$
\item\label{D_2Y25_AIFlip_Full}
We say that $A$ has
{\emph{approximately inner $L^p$-tensor flip}}
(with constant~$M$)
if there exists a net $(v_{\ld})_{\ld \in \Ld}$
in $\inv (A \otimes_p A)$
such that $\| v_{\ld} \| \leq M$
and $\big\| v_{\ld}^{-1} \big\| \leq M$
for all $\ld \in \Ld,$
and such that
$\lim_{\ld \in \Ld} v_{\ld} (a \otimes b) v_{\ld}^{-1} = b \otimes a$
for all $a, b \in A.$
\end{enumerate}
\end{dfn}

We caution that, except for \ca{s} in the case $p = 2,$
the conditions in Definition~\ref{D_2Y25_AIFlip}
presumably depend on how the algebra $A$ is represented
on an $L^p$-space,
and are not intrinsic to~$A$.
It may well turn out that if they hold for some algebra~$A,$
and if $\ph \colon A \to B$ is an isomorphism
such that $\ph$ and $\ph^{-1}$ are completely bounded
in a suitable $L^p$~operator sense,
then they hold for~$B.$
(Also see Proposition~\ref{P_3730_FlipIso} below.)
Pursuing this idea is beyond the scope of this paper.
As an easily accessible substitute for completely bounded isomorphism,
we use the condition that there be an isomorphism
$\ps \colon A \otimes_p A \to B \otimes_p A$
such that $\ps (a_1 \otimes a_2) = \ph (a_1) \otimes a_2$
for all $a_1, a_2 \in A.$

The conditions in Definition~\ref{D_2Y25_AIFlip} are very strong.
See~\cite{EfR}
for the severe restrictions that having an approximately inner
C*~tensor flip places on a \ca;
the machinery used there
has been greatly extended since.

We are interested these conditions
because of the following consequence.
The hypotheses are strong because we do not yet have a proper general theory
of tensor products of algebras on $L^p$~spaces.

\begin{thm}\label{T_2Y25_AIPHF}
Let $\XBM$ and $\YCN$ be \sfm{s},
let $p \in [1, \infty),$
and let $A \subset \LLp$ and $B \subset L (L^p (Y, \nu) )$
be closed unital subalgebras
such that $A$ has approximately inner $L^p$-tensor half flip.
Let $\ph \colon A \to B$ be a unital \hm,
and suppose that there is a \ct{} \hm{}
$\ps \colon A \otimes_p A \to B \otimes_p A$
such that $\ps (a_1 \otimes a_2) = \ph (a_1) \otimes a_2$
for all $a_1, a_2 \in A.$
Then $\ph$ is injective and bounded below,
the range of $\ph$ is closed,
and $\ph$ is a Banach algebra isomorphism from $A$ to its range.
\end{thm}

\begin{proof}
It suffices to find $C > 0$ such that
$\| \ph (a) \| \geq C \| a \|$ for all $a \in A.$
By hypothesis,
there are $M \in [1, \infty)$
and a net $(v_{\ld})_{\ld \in \Ld}$ in $\inv (A \otimes_p A)$
satisfying the conditions
in Definition~\ref{D_2Y25_AIFlip}(\ref{D_2Y25_AIFlip_HF}).
Set $C = M^{-2} \| \ps \|^{-2}.$
For all $a \in A,$
we have
\[
\lim_{\ld \in \Ld}
   \ps (v_{\ld}) (\ph (a) \otimes 1_A) \ps \big( v_{\ld}^{-1} \big)
  = \ps \left( \lim_{\ld \in \Ld} v_{\ld} (a \otimes 1_A) v_{\ld}^{-1}
                  \right)
  = \ps (1_A \otimes a)
  = 1_B \otimes a.
\]
Therefore
\[
\| a \|
 = \| 1_B \otimes a \|
 \leq (\| \ps \| M) \cdot \| \ph (a) \otimes 1 \| \cdot (\| \ps \| M)
 = M^2 \| \ps \|^2 \| \ph (a) \|.
\]
That is,
$\| \ph (a) \| \geq C \| a \|.$
\end{proof}

We give two permanence properties
for approximately inner $L^p$-tensor (half) flip.

\begin{prp}\label{P_3730_FlipTP}
Let $\XBM$ and $\YCN$ be \sfm{s},
let $p \in [1, \infty),$
and let $A \subset \LLp$ and $B \subset L (L^p (Y, \nu) )$
be closed unital subalgebras.
Suppose that $A$ and $B$
have approximately inner $L^p$-tensor (half) flip.
Then the same is true of $A \otimes_p B.$
\end{prp}

\begin{proof}
Let $(v_{\ld})_{\ld \in \Ld_1}$
and $(w_{\ld})_{\ld \in \Ld_2}$
be nets as in the appropriate part of Definition~\ref{D_2Y25_AIFlip}
for $A$ and~$B.$
Using Theorem~2.6(5) of~\cite{PhLp1},
one checks that the net
$(v_{\ld_1} \otimes w_{\ld_2})_{(\ld_1, \ld_2) \in \Ld_1 \times \Ld_2}$
satisfies the condition
in the appropriate part of Definition~\ref{D_2Y25_AIFlip}
for $A \otimes B.$
\end{proof}

\begin{prp}\label{P_3730_FlipIso}
Let $\XBM$ and $\YCN$ be \sfm{s},
let $p \in [1, \infty),$
and let $A \subset \LLp$ and $B \subset L (L^p (Y, \nu) )$
be closed unital subalgebras.
Suppose that
there is a \ct{} \hm{} $\ph \colon A \to B$
with dense range
such that the algebraic tensor product of two copies of~$\ph$
extends to a \ct{} \hm{}
$\ph \otimes_p \ph \colon A \otimes_p A \to B \otimes_p B.$
If $A$ has approximately inner $L^p$-tensor (half) flip,
then so does~$B.$
\end{prp}

\begin{proof}
We give the proof for the tensor flip;
the proof for the tensor half flip is the same.
Let $(v_{\ld})_{\ld \in \Ld}$
be a net as in
Definition~\ref{D_2Y25_AIFlip}(\ref{D_2Y25_AIFlip_Full}).
For $\ld \in \Ld,$
set $w_{\ld} = (\ph \otimes_p \ph) (v_{\ld}).$
Then $(w_{\ld})_{\ld \in \Ld}$
and $( w_{\ld}^{-1} )_{\ld \in \Ld}$
are bounded nets in $B \otimes_p B.$
We have $w_{\ld} (a \otimes b) w_{\ld}^{-1} \to b \otimes a$
for all $a, b \in \ph (A).$
Density of $\ph (A)$
and boundedness of $(w_{\ld})_{\ld \in \Ld}$
and $( w_{\ld}^{-1} )_{\ld \in \Ld}$
allows us to use an $\frac{\ep}{3}$~argument
to conclude that
$w_{\ld} (a \otimes b) w_{\ld}^{-1} \to b \otimes a$
for all $a, b \in B.$
\end{proof}

\begin{prp}\label{P_2Y26_GetAIF}
Let $d$ be a sequence in $\{ 2, 3, 4, \ldots \}.$
Let $r_d$ be as in Definition~\ref{D_2Y15_SNat}.
Let $A$ be a unital Banach algebra,
let $D_0 \subset D_1 \subset \cdots \subset A$ be an increasing sequence
of unital subalgebras
such that $A = {\overline{\bigcup_{n = 0}^{\infty} D_n}},$
and for $n \in \N$
let $\sm_n \colon M_{r_d (n)} \to A$
be a unital \hm{} whose range is~$D_n.$
\begin{enumerate}
\item\label{P_2Y26_GetAIF_Amen}
Suppose that,
in $A {\widehat{\otimes}} A,$
we have
$\sup_{n \in \N}
   \big\| \big( \sm_n {\widehat{\otimes}} \sm_n \big)
          (y_{r_d (n)}) \big\|_{\pi}
 < \infty.$
Then $A$ is symmetrically amenable.
\item\label{P_2Y26_GetAIF_AIF}
Let $p \in [1, \infty),$
let $\XBM$ be a \sfm,
and suppose that $A \subset \LLp$
and that,
in $A \otimes_p A,$
we have
$\sup_{n \in \N}
   r_d (n) \big\| (\sm_n \otimes_p \sm_n) (y_{r_d (n)}) \big\| < \infty.$
Then $A$ has approximately inner $L^p$-tensor flip.
\end{enumerate}
\end{prp}

\begin{proof}
We prove~(\ref{P_2Y26_GetAIF_Amen}),
by verifying the conditions in
Proposition~\ref{P_2Y25_ChrAm}(\ref{P_2Y25_ChrAm_Sym}).
Set $z_n = \big( \sm_n {\widehat{\otimes}} \sm_n \big) (y_{r_d (n)})$
for $n \in \N.$
The hypotheses imply that $(z_n)_{n \in \N}$ is bounded
in $A {\widehat{\otimes}} A.$
Lemma~\ref{L_2Y25_DiagMn}(\ref{L_2Y25_DiagMn_Dt})
and Lemma~\ref{L_3804_DtFnc}
imply that $\Dt_A (z_n) = \Dt_A^{\mathrm{op}} (z_n) = 1.$

It follows from Lemma~\ref{L_2Y25_DiagMn}(\ref{L_2Y25_DiagMn_Conj})
that $z_n (a \otimes b) = (b \otimes a) z_n$
for all $a, b \in D_n.$
Therefore
\begin{equation}\label{Eq_2Y26_Lim}
\limi{n} \big[ z_n (a \otimes b) - (b \otimes a) z_n \big] = 0
\end{equation}
for all
$a, b \in \bigcup_{n = 0}^{\infty} D_n.$
Since $(z_n)_{n \in \N}$ is bounded
and $\bigcup_{n = 0}^{\infty} D_n$
is dense in~$A,$
% we get
% \[
% \limi{n} \big( z_n (a \otimes b) - (b \otimes a) z_n \big) = 0
% \]
% for all $a, b \in A.$
% This relation completes the proof of the claim.
it follows that~(\ref{Eq_2Y26_Lim})
holds for all $a, b \in A.$
This completes the proof.

For~(\ref{P_2Y26_GetAIF_AIF}),
we verify the condition in
Definition~\ref{D_2Y25_AIFlip}(\ref{D_2Y25_AIFlip_Full}).
Set
\[
v_n = r_d (n) (\sm_n \otimes_p \sm_n) (y_{r_d (n)})
\]
for $n \in \N.$
By hypothesis, $(v_n)_{n \in \N}$ is bounded
in $A \otimes_p A.$
Lemma~\ref{L_2Y25_DiagMn}(\ref{L_2Y25_DiagMn_Sq1})
implies that $v_n^{-1} = v_n$ for all $n \in \N,$
so $( v_n^{-1} )_{n \in \N}$ is also bounded.
It follows from Lemma~\ref{L_2Y25_DiagMn}(\ref{L_2Y25_DiagMn_Conj})
that $v_n (a \otimes b) v_n^{-1} = b \otimes a$
for all $a, b \in D_n.$
Arguing as at the end of the proof of~(\ref{P_2Y26_GetAIF_Amen}),
we get $\limi{n} v_n (a \otimes b) v_n^{-1} = b \otimes a$
for all $a, b \in A.$
\end{proof}

The construction in Example~\ref{E_2Y14_LpUHF}
requires probability measures,
so that infinite products make sense.
This causes problem when passing to subspaces.
We state for reference the solution we usually use.

\begin{lem}\label{L_3903_ChM}
Let $\XBM$ be a \msp,
let $p \in [1, \infty),$
and let $\af \in (0, \infty).$
Then $L^p (X, \mu)$ and $L^p (X, \af \mu)$
are equal as vector spaces,
but with
$\| \xi \|_{L^p (X, \af \mu)} = \af^{1 / p} \| \xi \|_{L^p (X, \mu)}$
for $\xi \in L^p (X, \mu).$
Moreover, the identity map from $\LLp$
to $L (L^p (X, \af \mu))$ is an isometric isomorphism.
\end{lem}

\begin{proof}
All statements in the lemma are immediate.
\end{proof}

\begin{dfn}\label{D_3907_SubSys}
Let the notation be as in Example~\ref{E_2Y14_LpUHF}.
A {\emph{subsystem}} of $(d, \rh)$
consists of subspaces $Z_n \subset X_n$
such that $\mu_n (Z_n) > 0$
and $L^p (Z_n, \mu_n)$
is an invariant subspace for $\rh_n$ for all $n \in \N.$

For $n \in \N,$
set $\ld_n = \mu_n (Z_n)^{-1} \mu_n |_{Z_n}$
and let $\io_n \colon L (L^p (Z_n, \mu_n)) \to L (L^p (Z_n, \ld_n))$
be the isometric isomorphism of Lemma~\ref{L_3903_ChM}.
The {\emph{$L^p$~UHF algebra of the subsystem $(Z_n)_{n \in \N}$}}
is the one constructed as in Example~\ref{E_2Y14_LpUHF}
using the same sequence~$d$
and the sequence $\gm$ of \rpn{s} given by
$\gm_n (x) = \io_n \big( \rh_n (x) |_{L^p (Z_n, \mu_n)} \big)$
for $n \in \N$ and $x \in M_{d (n)}.$
\end{dfn}

\begin{lem}\label{L_3907_MapToSub}
Let the notation be as in
Definition~\ref{D_3907_SubSys}
and Example~\ref{E_2Y14_LpUHF}.
Set $A = A (d, \rh)$ and $B = A (d, \gm).$
Let $\sm_n \colon M_{r_d (n)} \to A$
be as at the end of Example~\ref{E_2Y14_LpUHF},
and let $\ta_n \colon M_{r_d (n)} \to B$
be the analogous map for $(d, \gm).$
Then there is a unique contractive unital \hm{}
$\kp \colon A \to B$ such that $\kp \circ \sm_n = \ta_n$
for all $n \in \N.$
Moreover,
$\kp$ has the following properties:
\begin{enumerate}
\item\label{3907_MapToSub_DR}
$\kp$ has dense range.
\item\label{3907_MapToSub_TP}
Whenever $\YCN$ is a \sfm{}
and $D \subset L (L^p (Y, {\mathcal{C}}, \nu) )$
is a closed unital subalgebra,
there is a contractive \hm{}
$\kp_D \colon A \otimes_p D \to B \otimes_p D$
such that $\kp_D (a \otimes x) = \kp (a) \otimes x$
for all $a \in A$ and $x \in D.$
\end{enumerate}
\end{lem}

To get the map~$\kp,$
we really only need $\| \ta_n (x) \| \leq \| \sm_n (x) \|$
for all $n \in \Nz$ and all $x \in M_{r_d (n)}.$
We need the subsystem condition to get~(\ref{3907_MapToSub_TP}).

\begin{proof}[Proof of Lemma~\ref{L_3907_MapToSub}]
We construct the \hm{} $\kp_D$
in~(\ref{3907_MapToSub_TP})
for arbitrary~$D.$
We then get $\kp$ by taking $D = \C.$

We continue to follow the notation
of Definition~\ref{D_3907_SubSys}
and Example~\ref{E_2Y14_LpUHF}.
In particular,
for $S \subset {\mathbb{N}}$ finite,
we let $X_S$ and its measure $\mu_S$
be as in Example~\ref{E_2Y14_LpUHF}.
We abbreviate $A_{ {\mathbb{N}}_{\leq n}}$
to $A_n,$
and set
$\af_n = \rh_{ {\mathbb{N}}_{\leq n}, {\mathbb{N}}_{\leq n + 1}}
  \colon A_n \to A_{n + 1},$
so that $\af_n$ is isometric and $A$ is isometrically
isomorphic to $\dirlim_n A_n.$
Analogously,
define $Z_S = \prod_{k \in S} Z_k,$
let $\ld_S$ be the corresponding product measure,
set $B_n = B_{ {\mathbb{N}}_{\leq n}},$
and set
$\bt_n = \gm_{ {\mathbb{N}}_{\leq n}, {\mathbb{N}}_{\leq n + 1}},$
giving $B = \dirlim_n B_n.$
Set $c_S = \prod_{k \in S} \mu_k (Z_k)^{-1},$
so that $\ld_S \times \nu = c_S \mu_S \times \nu.$
Let
\[
\io_S \colon L ( L^p (Z_S \times Y, \, \mu_S \times \nu) )
   \to L ( L^p (Z_S \times Y, \, \ld_S \times \nu) )
\]
be the identification of Lemma~\ref{L_3903_ChM},
and define
$\kp_{D, n} \colon A_{n} \otimes_p D
               \to B_{n} \otimes_p D$
by
\[
\kp_{D, n} (a)
 = \io_{{\mathbb{N}}_{\leq n}}
    \big( a |_{L^p (Z_{{\mathbb{N}}_{\leq n}} \times Y, \,
       \mu_{{\mathbb{N}}_{\leq n}} \times \nu)} \big)
\]
for $a \in A_n \otimes_p D.$

One checks immediately that $\kp_{D, n}$ is contractive
for all $n \in \Nz.$
Since $A_{n}$ is \fd,
$\kp_{D, n}$ is also bijective.
Moreover,
the following diagram
(not including the dotted arrow) now clearly commutes:
\[
\xymatrix{
A_{0} \otimes_p D \ar[rr]^{\af_{0} \otimes_p \id_D} \ar[d]_{\kp_{D, 0}}
  & & A_{1} \otimes_p D \ar[rr]^{\af_{1} \otimes_p \id_D}
             \ar[d]_{\kp_{D, 1}}
  & & A_{2} \otimes_p D \ar[r] % ^{\af_{2} \otimes_p \id_D}
             \ar[d]_{\kp_{D, 2}}
  & \cdots \ar[r]
  & A \otimes_p D \ar@{-->}[d]^{\kp_D} \\
B_0 \otimes_p D \ar[rr]_{\bt_{0} \otimes_p \id_D}
  & & B_1 \otimes_p D \ar[rr]_{\bt_{1} \otimes_p \id_D}
  & & B_2 \otimes_p D \ar[r] % _{\bt_{2} \otimes_p \id_D}
  & \cdots \ar[r] & B \otimes_p D.
}
\]
For $n \in \N,$
there is a \hm{}
$A_{n} \otimes_p D \to A \otimes_p D$
given by
$a \otimes x \mapsto a \otimes 1_{{\mathbb{N}}_{\leq n} } \otimes x$
for $a \in A_n$ and $x \in D.$
Except for the fact that the domain
is now taken to be normed,
this \hm{} is just $\sm_n \otimes \id_D.$
It is isometric by Theorem~2.16(5) of~\cite{PhLp1}.
Thus we can identify $A \otimes_p D$
with $\dirlim_n A_{n} \otimes_p D.$
Similarly we can identify $B \otimes_p D$
with $\dirlim_n B_{n} \otimes_p D,$
using the maps which
become, after forgetting the norms, $\ta_n \otimes \id_D.$
Commutativity of the part of diagram with solid arrows,
contractivity of $\kp_{D, n},$
and the fact that $\kp_{D, n}$ becomes the identity
when the norms are forgotten,
therefore imply the existence of a contractive
\hm{} $\kp_D \colon A \otimes_p D \to B \otimes_p D$
such that
$\kp_D \circ ( \sm_n \otimes \id_D ) = \ta_n \otimes \id_D$
for all $n \in \N.$
Its range is dense because it contains
$\bigcup_{n = 0}^{\infty} B_n \otimes D.$
\end{proof}

\begin{cor}\label{L_2Z17New}
Let $p \in [1, \infty),$
and let $A$ be an $L^p$~UHF algebra
of tensor product type which dominates the spatial representation
in the sense of Definition~\ref{D_2Y15_ITP}(\ref{D_2Y15_DomSpatial}).
Let $B$ be the spatial $L^p$~UHF algebra whose supernatural number~$N$
is the same as that of~$A,$
as in Theorem~\ref{T_3729_EU}.
Then there exists a contractive unital \hm{} $\kp \colon A \to B$
with the following properties:
\begin{enumerate}
\item\label{2Z17New_DR}
$\kp$ has dense range.
\item\label{2Z17New_Inv}
Let $\sm_n \colon M_{r_d (n)} \to A$
be as at the end of Example~\ref{E_2Y14_LpUHF},
and identify $M_{r_d (n)}$ with $\MP{r_d (n)}{p}.$
Then $\kp \circ \sm_n$ is isometric for all $n \in \Nz.$
\item\label{2Z17New_TP}
Whenever $\YCN$ is a \sfm{}
and $D \subset L (L^p (Y, {\mathcal{C}}, \nu) )$
is a closed unital subalgebra,
there is a \ct{} \hm{}
$\kp_D \colon A \otimes_p D \to B \otimes_p D$
such that $\kp_D (a \otimes x) = \kp (a) \otimes x$
for all $a \in A$ and $x \in D.$
\end{enumerate}
\end{cor}

\begin{proof}
Adopt the notation of Example~\ref{E_2Y14_LpUHF}.
We apply Lemma~\ref{L_3907_MapToSub},
for $n \in \N$
taking $Z_n \subset X_n$
to be a subspace,
whose existence is guaranteed by
Definition~\ref{D_2Y15_ITP}(\ref{D_2Y15_DomSpatial}),
such that $L^p (Z_n, \mu_n)$
is invariant under $\rh_n$
and $a \mapsto \rh_n (a) |_{L^p (Z_n, \mu_n)}$ is spatial.
Everything except~(\ref{2Z17New_Inv}) is immediate.
In part~(\ref{2Z17New_Inv}),
we have to identify the algebra $B \otimes_p D$
which appears here with the algebra $B \otimes_p D$
in Lemma~\ref{L_3907_MapToSub}.
We first observe that $\kp \circ \sm_n$ is spatial
for all $n \in \Nz.$
Thus, the conclusion is correct
for some unital isometric representation of $B$ on an $L^p$~space.
By the last part of Theorem~\ref{T_3729_EU},
the tensor product $B \otimes_p D$
is the same
for any unital isometric representation of $B$ on an $L^p$~space.
\end{proof}

\begin{lem}\label{L_2Y26_SumProd}
Let $(\af_n)_{n \in \N}$ be a sequence in $[1, \infty).$
Then $\sum_{n = 1}^{\infty} ( \af_n - 1 ) < \infty$
\ifo{}
$\prod_{n = 1}^{\infty} \af_n < \infty.$
\end{lem}

\begin{proof}
If
$\sup_{n \in \N} \af_n = \infty,$
then both statements fail.
Otherwise, set $\bt_n = \af_n - 1$ for $n \in \N,$
and set $M = \sup_{n \in \N} \bt_n.$
Then for all $n \in \N$ we have $0 \leq \bt_n \leq M.$
Using concavity of $\bt \mapsto \log (1 + \bt)$
on the interval $[0, M]$
for the first inequality,
we get
\[
M^{-1} \log (M + 1) \bt_n \leq \log (1 + \bt_n) \leq \bt_n.
\]
Therefore $\sum_{n = 1}^{\infty} \bt_n < \infty$
\ifo{} $\sum_{n = 1}^{\infty} \log ( 1 + \bt_n ) < \infty,$
which clearly happens \ifo{}
$\prod_{n = 1}^{\infty} (1 + \bt_n) < \infty.$
\end{proof}

\begin{thm}\label{T_2Z29_First}
Let $p \in [1, \infty).$
Let $A$ be an $L^p$~UHF algebra of tensor product type
(Definition~\ref{D_2Y15_ITP}(\ref{D_2Y15_ITP_Basic})),
and let $d,$ $\rh = (\rh_1, \rh_2, \ldots),$
and $\sm_n \colon M_{r_d (n)} \to A$
be as there.
Assume that $A$ dominates the spatial representation
(Definition~\ref{D_2Y15_ITP}(\ref{D_2Y15_DomSpatial}));
in particular, that $A$ locally has enough isometries
(Definition~\ref{D_2Y15_ITP}(\ref{D_2Y15_ITP_LEI})).
Let $B$ be the spatial $L^p$~UHF algebra whose supernatural number
is the same as that of~$A,$
as in Theorem~\ref{T_3729_EU}.
Then \tfae:
\begin{enumerate}
\item\label{T_2Z29_First_IsoSp}
There exists an isomorphism $\ph \colon A \to B$
such that the algebraic tensor product of two copies of~$\ph$
extends to an isomorphism
$\ph \otimes_p \ph \colon A \otimes_p A \to B \otimes_p B.$
\item\label{T_2Z29_First_ClRange}
$A \otimes_p A$
has approximately inner tensor half flip.
\item\label{T_2Z29_First_AIF}
$A$ has approximately inner $L^p$-tensor flip.
\item\label{T_2Z29_First_Half}
$A$ has approximately inner $L^p$-tensor half flip.
\item\label{T_2Z29_First_SmBdd}
There is a uniform bound on the norms of the \hm{s}
\[
\sm_n \otimes_p \sm_n \colon \MP{r_d (n)^2}{p} \to A \otimes_p A.
\]
\item\label{T_2Z29_First_SumFin}
The \hm{s}
$\rh_n \otimes_p \rh_n \colon \MP{d (n)^2}{p} \to A \otimes_p A$
satisfy
\[
\sum_{n = 1}^{\infty} ( \| \rh_n \otimes_p \rh_n \| - 1 ) < \infty.
\]
\setcounter{TmpEnumi}{\value{enumi}}
\end{enumerate}
Moreover, these conditions imply the following:
\begin{enumerate}
\setcounter{enumi}{\value{TmpEnumi}}
\item\label{T_2Z29_First_Simp}
$A$ is simple.
\item\label{T_2Z29_First_RhBdd}
$\sum_{n = 1}^{\infty} ( \| \rh_n \| - 1 ) < \infty.$
\item\label{T_2Z29_First_SymAmen}
$A$ is symmetrically amenable.
\item\label{T_2Z29_First_Amen}
$A$ is amenable.
\setcounter{RsvEnumi}{\value{enumi}}
\end{enumerate}
\end{thm}

\begin{proof}
It is clear that (\ref{T_2Z29_First_AIF})
implies~(\ref{T_2Z29_First_Half}).
That~(\ref{T_2Z29_First_Half}) implies~(\ref{T_2Z29_First_ClRange})
is 
Proposition~\ref{P_3730_FlipTP}.

We show that~(\ref{T_2Z29_First_ClRange})
implies both (\ref{T_2Z29_First_IsoSp}) and~(\ref{T_2Z29_First_SmBdd}).

Let $\kp \colon A \to B$ be as in Corollary~\ref{L_2Z17New}.
Corollary \ref{L_2Z17New}(\ref{2Z17New_TP}) provides \ct{} \hm{s}
\[
\gm_1 \colon A \otimes_p A \to B \otimes_p A
\andeqn
\gm_2 \colon B \otimes_p A \to B \otimes_p B
\]
such that $\gm_1 (a \otimes b) = \kp (a) \otimes b$
for all $a, b \in A$
and $\gm_2 (b \otimes a) = b \otimes \kp (a)$
for all $a \in A$ and $b \in B.$
Then
${\overline{\kp}}
 = \gm_2 \circ \gm_1 \colon A \otimes_p A \to B \otimes_p B$
is a \ct{} \hm{} such that
${\overline{\kp}} (a \otimes b) = \kp (a) \otimes \kp (b)$
for all $a, b \in A.$

We next claim that whenever $\YCN$ is a \sfm{}
and $D \subset L (L^p (Y, {\mathcal{C}}, \nu) )$
is a closed unital subalgebra,
there is a \ct{} \hm{}
$\bt \colon A \otimes_p A \otimes_p D \to B \otimes_p B \otimes_p D$
such that $\bt (a \otimes b \otimes x) = \kp (a) \otimes \kp (b) \otimes x$
for all $a, b \in A$ and $x \in D.$
To prove this,
apply Corollary \ref{L_2Z17New}(\ref{2Z17New_TP})
with $A \otimes_p D$ in place of~$D,$
and apply Corollary \ref{L_2Z17New}(\ref{2Z17New_TP}),
after changing the order of the tensor factors,
with $B \otimes_p D$ in place of~$D,$
to get \ct{} \hm{s}
\[
\bt_1 \colon A \otimes_p A \otimes_p D \to B \otimes_p A \otimes_p D
\andeqn
\bt_2 \colon B \otimes_p A \otimes_p D \to B \otimes_p B \otimes_p D,
\]
such that $\bt_1 (a \otimes b \otimes x) = \kp (a) \otimes b \otimes  x$
for all $a, b \in A$ and $x \in D,$
and such that $\bt_2 (b \otimes a \otimes x) = b \otimes \kp (a) \otimes x$
for all $a \in A,$ and $b \in B,$ and $x \in D.$
Then set $\bt = \bt_2 \circ \bt_1.$
This proves the claim.

Apply the claim with $D = A \otimes_p A,$
obtaining
\[
\bt \colon A \otimes_p A \otimes_p A \otimes_p A
 \to B \otimes_p B \otimes_p A \otimes_p A
\]
such that $\bt (a \otimes x) = {\overline{\kp}} (a) \otimes x$
for all $a, x \in A \otimes_p A.$
Theorem~\ref{T_2Y25_AIPHF}
provides $C > 0$
such that $\| {\overline{\kp}} (a) \| \geq C \| a \|$
for all $a \in A \otimes_p A.$
The map ${\overline{\kp}}$ clearly has dense range.
Therefore ${\overline{\kp}}$ is an isomorphism.
Moreover,
for $a \in A$ we have
\[
\| \kp (a) \|
 = \| \kp (a) \otimes 1_B \|
 = \| {\overline{\kp}} (a \otimes 1_A) \|
 \geq C \| a \otimes 1_A \|
 = C \| a \|.
\]
Since $\kp$ has dense range,
it follows that $\kp$ is also an isomorphism.
Part~(\ref{T_2Z29_First_IsoSp}) is proved.

To get~(\ref{T_2Z29_First_SmBdd}),
observe that the choice of $\kp$
and Corollary \ref{L_2Z17New}(\ref{2Z17New_Inv})
imply that $\kp \circ \sm_n$ is isometric for all~$n \in \Nz.$
Corollary~1.13 of~\cite{PhLp2a}
now implies that
\[
{\overline{\kp}} \circ (\sm_n \otimes \sm_n)
 = (\kp \circ \sm_n) \otimes (\kp \circ \sm_n)
\]
is isometric for all~$n \in \Nz.$
Therefore
$\sup_{n \in \Nz} \| \sm_n \otimes_p \sm_n \| \leq C^{-1},$
which is~(\ref{T_2Z29_First_SmBdd}).

Using Corollary~1.13 of~\cite{PhLp2a}
to see that the maps
$\MP{r_d (n)}{p} \otimes_p \MP{r_d (n)}{p} \to B \otimes_p B$
are isometric,
it follows from
Proposition~\ref{P_2Y26_GetAIF}(\ref{P_2Y26_GetAIF_AIF})
and Lemma~\ref{L_2Y25_DiagMn}(\ref{L_2Y25_DiagMn_pNm})
that $B$ has approximately inner tensor flip.
The implication from~(\ref{T_2Z29_First_IsoSp})
to~(\ref{T_2Z29_First_AIF})
is then Proposition~\ref{P_3730_FlipIso}.

Next assume~(\ref{T_2Z29_First_SmBdd});
we prove~(\ref{T_2Z29_First_SumFin}).
For $n \in \N,$
by suitably permuting tensor factors
and using
Lemma~1.11 of~\cite{PhLp2a} and Corollary 1.13 of~\cite{PhLp2a},
we can make the identification
\[
\MP{r_d (n)^2}{p}
 = \MP{d (1)^2}{p} \otimes_p \MP{d (2)^2}{p} \otimes_p
   \cdots \otimes_p \MP{d (n)^2}{p}.
\]
Then
\[
\prod_{k = 1}^{n} \| \rh_k \otimes_p \rh_k \|
 \leq \| \sm_n \otimes_p \sm_n \|
\]
by
Lemma~1.14 of~\cite{PhLp2a}.
Thus
$\prod_{k = 1}^{\infty} \| \rh_k \otimes_p \rh_k \| < \infty.$
So
$\sum_{n = 1}^{\infty} \big( \| \rh_n \otimes_p \rh_n \| - 1 \big) < \infty$
by Lemma~\ref{L_2Y26_SumProd}.

Next,
we show that~(\ref{T_2Z29_First_SumFin})
implies~(\ref{T_2Z29_First_AIF}).
First, $\prod_{k = 1}^{\infty} \| \rh_k \otimes_p \rh_k \| < \infty$
by Lemma~\ref{L_2Y26_SumProd}.
By Lemma~\ref{L_2Y25_DiagMn}(\ref{L_2Y25_DiagMn_d1d2}),
up to a permutation
of tensor factors
(which is isometric by
Lemma~1.11 of~\cite{PhLp2a}),
for $d_1, d_2 \in \N$
we have $y_{d_1} \otimes y_{d_2} = y_{d_1 d_2}.$
Therefore, using the tensor product decompositions
in Example~\ref{E_2Y14_LpUHF}
on $L^2 (X \times X, \, \mu \times \mu),$
we get
\begin{align*}
r_d (n) (\sm_n \otimes_p \sm_n) (y_{r_d (n)})
& = d (1) (\rh_1 \otimes_p \rh_1) (y_{d (1)})
      \otimes d (2) (\rh_2 \otimes_p \rh_2) (y_{d (2)})
    \\
& \hspace*{5em} {\mbox{}}
      \otimes \cdots \otimes d (n) (\rh_n \otimes_p \rh_n) (y_{d (n)})
      \otimes (1_{{\mathbb{N}}_{\leq n}}
      \otimes 1_{{\mathbb{N}}_{\leq n}}).
\end{align*}
So
Theorem~2.16(5) of~\cite{PhLp1}
and Lemma~\ref{L_2Y25_DiagMn}(\ref{L_2Y25_DiagMn_pNm})
imply that
\[
\sup_{n \in \Nz} r_d (n) \| (\sm_n \otimes_p \sm_n) (y_{r_d (n)}) \|
  \leq \prod_{k = 1}^{\infty} \| \rh_k \otimes_p \rh_k \|
  < \infty.
\]
Thus $A$ has approximately inner $L^p$-tensor flip
by Proposition~\ref{P_2Y26_GetAIF}(\ref{P_2Y26_GetAIF_AIF}).

We have now completed the proof of the equivalence of conditions
(\ref{T_2Z29_First_IsoSp}),
(\ref{T_2Z29_First_ClRange}),
(\ref{T_2Z29_First_AIF}),
(\ref{T_2Z29_First_Half}),
(\ref{T_2Z29_First_SmBdd}),
and (\ref{T_2Z29_First_SumFin}).
It remains to show that these conditions imply
(\ref{T_2Z29_First_Simp}),
(\ref{T_2Z29_First_RhBdd}),
(\ref{T_2Z29_First_SymAmen}),
and~(\ref{T_2Z29_First_Amen}).

That (\ref{T_2Z29_First_IsoSp})
implies~(\ref{T_2Z29_First_Simp})
is
Theorem~3.13 of~\cite{PhLp2a}.

Assume~(\ref{T_2Z29_First_SumFin});
we prove~(\ref{T_2Z29_First_RhBdd}).
We have $\| \rh_n \otimes_p \rh_n \| \geq \| \rh_n \|^2$
by Lemma~1.13 of~\cite{PhLp2a}.
Lemma~\ref{L_2Y26_SumProd}
implies that
$\prod_{n = 1}^{\infty} \| \rh_n \otimes_p \rh_n \| < \infty,$
so $\prod_{n = 1}^{\infty} \| \rh_n \|^2 < \infty,$
whence $\prod_{n = 1}^{\infty} \| \rh_n \| < \infty.$
So $\sum_{n = 1}^{\infty} ( \| \rh_n \| - 1 ) < \infty$
by Lemma~\ref{L_2Y26_SumProd}.

Assume~(\ref{T_2Z29_First_IsoSp});
we prove~(\ref{T_2Z29_First_SymAmen}).
Since symmetric amenability only depends on the isomorphism class
of a Banach algebra,
it is enough to prove symmetric amenability when $A = B.$ 
The maps $\sm_n \colon \MP{r_d (n)}{p} \to B$ are then isometric,
so that the maps
$\sm_n {\widehat{\otimes}} \sm_n \colon
   \MP{r_d (n)}{p} {\widehat{\otimes}} \MP{r_d (n)}{p} \to
    A {\widehat{\otimes}} A$
are contractive.
Now apply Proposition~\ref{P_2Y26_GetAIF}(\ref{P_2Y26_GetAIF_Amen})
and Lemma~\ref{L_2Y25_DiagMn}(\ref{L_2Y25_DiagMn_Norm}).

Condition~(\ref{T_2Z29_First_Amen})
follows from~(\ref{T_2Z29_First_SymAmen}).
\end{proof}

\section{$L^p$~UHF algebras constructed from similarities}\label{Sec_ClEx}

\indent
In this section, we describe and derive the properties of
a special class of $L^p$~UHF algebras of tensor product type.
In Example~\ref{E_2Y14_LpUHF},
for $n \in \N$ we will require that
the \rpn{} $\rh_n \colon M_{d (n)} \to L (L^p (X_n, \mu_n))$
be a finite or countable direct sum of \rpn{s}
of the form $x \mapsto s x s^{-1}$ for invertible elements
$s \in \MP{d (n)}{p}.$
The results of Section~\ref{Sec_Amen}
and Section~\ref{Sec_ManyNI}
are for algebras in this class.

One of the main purposes of this section is to set up notation.
The definition of a system of $d$-similarities
is intended only for local use,
to simplify the description of our construction.

\begin{ntn}\label{N_3406_MatU}
For $d \in \N,$
we let $(e_{j, k})_{j, k = 1, 2, \ldots, d}$
be the standard system of matrix units for~$M_d.$
\end{ntn}

\begin{dfn}\label{D_3406_SysSim}
Let $d \in \N.$
A {\emph{system of $d$-similarities}}
is a triple $S = (I, s, f)$ consisting of
a countable (possibly finite) index set~$I,$
a function $s \colon I \to \inv (M_d),$
and a function $f \colon I \to [0, 1]$
such that:
\begin{enumerate}
\item\label{D_3406_SysSim_One}
$1 \in {\mathrm{ran}} (s).$
\item\label{D_3406_SysSim_Inv}
${\overline{{\mathrm{ran}} (s)}} \subset \inv (M_d).$
\item\label{D_3406_SysSim_Cpt}
${\overline{{\mathrm{ran}} (s)}}$ is compact.
\item\label{D_3406_SysSim_Pos}
$f (i) > 0$ for all $i \in I.$
\item\label{D_3406_SysSim_Sum1}
$\sum_{i \in I} f (i) = 1.$
\setcounter{TmpEnumi}{\value{enumi}}
\end{enumerate}
We say that $s$ is {\emph{diagonal}}
if, in addition:
\begin{enumerate}
\setcounter{enumi}{\value{TmpEnumi}}
\item\label{D_3406_SysSim_Diag}
$s (i)$ is a diagonal matrix for all $i \in I.$
\setcounter{TmpEnumi}{\value{enumi}}
\end{enumerate}
The {\emph{basic system of $d$-similarities}}
is the system $S_0 = (I_0, s_0, f_0)$
given by $I_0 = \{ 0 \},$ $s_0 (0) = 1,$ and $f_0 (0) = 1.$
\end{dfn}

We will associate representations to
systems of $d$-similarities,
but we first need some lemmas.

\begin{lem}\label{L_3317_NConj}
Let $E$ be a Banach space,
and let $s \in L (E)$ be invertible.
Define $\af \colon L (E) \to L (E)$ by $\af (a) = s a s^{-1}$
for all $a \in L (E).$
Then $\| \af \| = \| s \| \cdot \| s^{-1} \|.$
\end{lem}

\begin{proof}
It is clear that $\| \af \| \leq \| s \| \cdot \| s^{-1} \|.$
We prove the reverse inequality.
Let $\ep > 0.$
Choose $\dt > 0$ such that
\[
( \| s \| - \dt) ( \| s^{-1} \| - \dt)
   > \| s \| \cdot \| s^{-1} \| - \ep.
\]
Choose $\et \in E$ such that $\| \et \| = 1$
and $\| s \et \| > \| s \| - \dt.$
Choose $\mu \in E$ such that $\| \mu \| = 1$
and $\| s^{-1} \mu \| > \| s^{-1} \| - \dt.$
Use the Hahn-Banach Theorem to find a \ct{}
linear functional $\om$ on~$E$
such that $\| \om \| = 1$ and $\om ( s^{-1} \mu ) = \| s^{-1} \mu \|.$
Define $a \in L (E)$ by $a \xi = \om (\xi) \et$ for all $\xi \in E.$
Then $\| a \| = 1.$
Also,
\[
s a s^{-1} \mu
 = s ( \om (s^{-1} \mu)) \et
 = \| s^{-1} \mu \| \cdot s \et.
\]
Therefore
\[
\| s a s^{-1} \mu \|
 = \| s^{-1} \mu \| \cdot \| s \et \|
 > ( \| s \| - \dt) ( \| s^{-1} \| - \dt)
 > \| s \| \cdot \| s^{-1} \| - \ep.
\]
So
$\| \af (a) \|
  \geq \| s a s^{-1} \| > \| s \| \cdot \| s^{-1} \| - \ep,$
whence $\| \af \| > \| s \| \cdot \| s^{-1} \| - \ep.$
Since $\ep > 0$ is arbitrary,
it follows that $\| \af \| \geq \| s \| \cdot \| s^{-1} \|.$
\end{proof}

\begin{cor}\label{C_3406_SySiBdd}
Let $d \in \N,$
let $(I, s, f)$ be a system of $d$-similarities,
and let $p \in [1, \infty).$
For $i \in I,$
let $\ps_i \colon M_d \to \MP{d}{p}$
be the representation $\ps_i (a) = s (i) a s (i)^{-1}$
for $a \in M_d.$
Then $\sup_{i \in I} \| \ps_i (a) \|_p$ is finite
for all $a \in M_d.$
\end{cor}

\begin{proof}
Set
\[
M = \sup \big( \big\{ \| v \|_p \| v^{-1} \|_p \colon
   v \in {\overline{{\operatorname{ran}} (s)}} \big\} \big).
\]
Then $M < \infty$
because ${\overline{{\operatorname{ran}} (s)}}$
is a compact subset of $\inv (M_d).$
Identify $M_d$ with $\MP{d}{p}.$
Then Lemma~\ref{L_3317_NConj} implies
$\sup_{i \in I} \| \ps_i \| \leq M.$
\end{proof}

We need infinite direct sums of \rpn{s} of algebras
on $L^p$~spaces.
(The finite case is given in Lemma 2.14
of~\cite{PhLp1},
and the infinite case is discussed in
Remark~2.15 of~\cite{PhLp1}
and implicitly used in Example~2.14 of~\cite{PhLp2a}.)
Since we work with \sfm{s},
we must restrict to countable direct sums.
We must require $p \neq \infty$
since the infinite direct sum of $L^{\infty}$~spaces
is usually not $L^{\infty}$ of the disjoint union.

\begin{lem}\label{L_3406_pSum}
Let $A$ be a unital complex algebra,
and let $p \in [1, \infty).$
Let $I$ be a countable index set,
and for $i \in I$ let $(X_i, \cB_i, \mu_i)$
be a \sfm\  and
let $\pi_i \colon A \to L (L^p (X_i, \mu_i))$
be a unital \hm.
Assume that $\sup_{i \in I} \| \pi_i (a) \|$ is finite
for every $a \in A.$
Let $E_0$ be the algebraic
direct sum over $i \in I$ of the spaces $L^p (X_i, \mu_i).$
Equip $E_0$
with the norm
\[
\| (\xi_i )_{i \in I} \|
 = \left( \ssum{i \in I} \| \xi_i \|_p^p \right)^{1 / p},
\]
and let $E$ be the completion of $E_0$ in this norm.
Then $E \cong L^p \left( \coprod_{i \in I} X_i \right),$
using the obvious disjoint union measure,
and there is a unique \rpn\  $\pi \colon A \to L (E)$
such that
\[
\pi (a) \big( (\xi_i )_{i \in I} \big)
        = \big( \pi_i (a) \xi_i \big)_{i \in I}
\]
for $a \in A$ and $(\xi_i )_{i \in I} \in E_0.$
Moreover,
$\| \pi (a) \| = \sup_{i \in I} \| \pi_i (a) \|$
for all $a \in A.$
\end{lem}

\begin{proof}
This is immediate.
\end{proof}

\begin{dfn}\label{D_3406_pSumDfn}
Let the notation and hypotheses be as in Lemma~\ref{L_3406_pSum}.
We call $E$ the {\emph{$L^p$~direct sum}}
of the spaces $L^p (X_i, \mu_i),$
and we call $\pi$
the {\emph{$L^p$~direct sum}} of
the \rpn{s} $\pi_i$ for $i \in I.$
We write $E = \bigoplus_{i \in I} L^p (X_i, \mu_i)$
and $\pi = \bigoplus_{i \in I} \pi_i,$
letting $p$ be understood.
\end{dfn}

\begin{rmk}\label{R_3913_DSElt}
We can form direct sums of bounded families of elements
as well as of representations.
(Just take $A$ to be the polynomial rang $\C [ x ].$)
Then, for example,
if $a_i \in L (L^p (X_i, \mu_i))$ is invertible for all $i \in I,$
and $\sup_{i \in I} \| a_i \|$ and $\sup_{i \in I} \| a_i^{-1} \|$
are both finite,
then $a = \bigoplus_{i \in I} a_i$ is invertible with inverse
$a^{-1} = \bigoplus_{i \in I} a_i^{-1}.$
Moreover, $\| a \| = \sup_{i \in I} \| a_i \|$
and $\| a^{-1} \| = \sup_{i \in I} \| a_i^{-1} \|.$
(Take $A = \C [ x, x^{-1} ].$)
\end{rmk}

We need to know that the sum of orthogonal spatial partial isometries
is again a spatial partial isometry.
This result is related to those of Section~6 of~\cite{PhLp1},
but does not appear there.
We refer to Section~6 of~\cite{PhLp1} for the terminology.
Here we only need the application to direct sums
(Corollary~\ref{C_3924_DSumSPI}),
but the more general statement will be needed elsewhere.

\begin{lem}\label{L_3924_SumSPI}
Let $\XBM$ and $\YCN$ be \sfm{s},
let $p \in [1, \infty],$
and for $j \in \N$
let $s_j \in L \big( L^p (X, \mu), \, L^p (Y, \nu) \big)$
be a spatial partial isometry with domain support $E_j \subset X,$
with range support $F_j \subset Y,$
with spatial realization~$S_j,$
with phase factor~$g_j,$
and with reverse~$t_j.$
Suppose that the sets $E_j$ are disjoint up to sets of measure zero,
and that the sets $F_j$ are disjoint up to sets of measure zero.
Then there exists a unique spatial partial isometry
$s \in L \big( L^p (X, \mu), \, L^p (Y, \nu) \big)$
with domain support $E = \bigcup_{j = 1}^{\infty} E_j$
and with range support $F = \bigcup_{j = 1}^{\infty} F_j,$
and such that $s |_{L^p (E_j, \mu)} = s_j |_{L^p (E_j, \mu)}$
for all $j \in \N.$
Its spatial realization $S$ is given by
\[
S (B) = \bigcup_{j = 1}^{\infty} S_j (B \cap E_j)
\]
for $B \in \cB$ with $B \subset E.$
Its phase factor $g$ is
given almost everywhere by
$g (y) = g_j (y)$ for $y \in F_j.$
Its reverse $t$ is obtained in the same manner as~$s$
using the $t_j$ in place of the~$s_j.$
Moreover, for all $\xi \in L^p (X, \mu),$
we have
$s \xi = \sum_{j = 1}^{\infty} s_j \xi,$
with weak* convergence for $p = \infty$
and norm convergence for $p \in [1, \infty).$
\end{lem}

\begin{proof}
Everything is easy to check directly.
\end{proof}

\begin{cor}\label{C_3924_DSumSPI}
Let $p \in [1, \infty],$
and let $I$ be a countable set.
For $i \in I$ let $(X_i, \cB_i, \mu_i)$
and $(Y_i, \cC_i, \nu_i)$ be \sfm{s},
and let $s_i \in L \big( L^p (X_i, \mu_i), \, L^p (Y_i, \nu_i) \big)$
be a spatial partial isometry with domain support $E_i \subset X_i,$
with range support $F_i \subset Y_i,$
with spatial realization~$S_i,$
with phase factor~$g_i,$
and with reverse~$t_i.$
Then $s = \bigoplus_{i \in I} s_i$
(as in Remark~\ref{R_3913_DSElt})
is a spatial partial isometry
with domain support $\coprod_{i \in I} E_i,$
with range support $\coprod_{i \in I} F_i,$
and with reverse $\bigoplus_{i \in I} t_i.$
The phase factor $g$ is determined by
$g |_{E_i} = g_i$ for $i \in I,$
and the spatial realization $S$ is determined by
$S (B) = S_i (B)$ for $B \in \cB_i$ with $B \subset E_i.$
\end{cor}

\begin{proof}
Apply Lemma~\ref{L_3924_SumSPI}.
The disjointness condition is trivial.
\end{proof}

\begin{lem}\label{L_3807_SpDSum}
Let the notation and hypotheses be as in Lemma~\ref{L_3406_pSum}.
If $\pi_i$ is spatial for every $i \in I,$
then $\bigoplus_{i \in I} \pi_i$ is spatial.
\end{lem}

\begin{proof}
Use Corollary~\ref{C_3924_DSumSPI}
to verify the criterion in Proposition~\ref{P_3910_SpMd}.
\end{proof}

The following definition is
Example~2.15 of~\cite{PhLp2a}.

\begin{dfn}\label{D_3406_SySiRpn}
Let $d \in \N,$
let $S = (I, s, f)$ be a system of $d$-similarities,
and let $p \in [1, \infty).$
We define $X_S = X_{I, s, f} = \{ 1, 2, \ldots, d \} \times I,$
with the $\sm$-algebra $\cB_S$ consisting of all subsets of~$X_S$
and the atomic measure $\mu_S = \mu_{I, s, f}$
determined by $\mu_{I, s, f} ( \{ (j, i) \} ) = d^{-1} f (i)$
for $j = 1, 2, \ldots, d$ and $i \in I.$
For $i \in I,$
we make the obvious identification (Lemma~\ref{L_3903_ChM})
\begin{equation}\label{Eq_3406_Ident}
L \big( L^p ( \{ 1, 2, \ldots, d \} \times \{ i \}, \, \mu_{S}) \big)
  = \MP{d}{p}.
\end{equation}
We then define
\[
\ps_i \colon M_d \to
L \big( L^p ( \{ 1, 2, \ldots, d \} \times \{ i \}, \, \mu_{S}) \big)
\]
by $\ps_i (a) = s (i) a s (i)^{-1}$ for $a \in M_d.$
The {\emph{representation associated with $(p, S)$}}
is then the unital \hm{}
$\ps^{p, S} = \ps^{p, I, s, f} \colon
 M_d \to L \big( L^p ( X_{S}, \mu_{S} ) \big)$
obtained as the $L^p$~direct sum over $i \in I,$
as in Definition~\ref{D_3406_pSumDfn},
of the \rpn{s}~$\ps_i,$
and whose existence is ensured by Corollary~\ref{C_3406_SySiBdd}
and Lemma~\ref{L_3406_pSum}.
Define
\[
\MP{d}{p, S}
  = \MP{d}{p, I, s, f}
  = \ps^{p, S} (M_d)
  \subset L \big( L^p ( X_{S}, \mu_{S} ) \big).
\]

We further define the {\emph{$p$-bound}} of $S$ to be
\[
R_{p, S}
 = R_{p, I, s, f}
 = \sup_{i \in I} \| s (i) \|_p \| s (i)^{-1} \|_p.
\]
\end{dfn}

\begin{lem}\label{L_3406_TensorB}
Let $d \in \N,$
let $S$ be a system of $d$-similarities,
and let $p \in [1, \infty).$
Further let $\YCN$ be a \sfm,
and let $B \subset L (L^p (Y, \nu))$ be a unital closed subalgebra.
Then the assignment $x \otimes b \mapsto \ps^{p, S} (x) \otimes b$
defines a bijective \hm{}
$\ps^{p, S}_B$
from the algebraic tensor product $M_d \otimes_{\operatorname{alg}} B$
to
\[
\MP{d}{p, S} \otimes_p B \subset
   L \big( L^p ( X_{S} \times Y, \, \mu_{I, s, f} \times \nu) \big).
\]
\end{lem}

\begin{proof}
The statement is immediate from the fact that
$M_d$ is \fd.
\end{proof}

\begin{ntn}\label{N_3406_Norm}
Let the notation be as in Lemma~\ref{L_3406_TensorB}.
For $x \in M_d \otimes_{\operatorname{alg}} B,$
we define
$\| x \|_{p, S}
 = \| x \|_{p, I, s, f}
 = \big\| \ps^{p, I, s, f}_B (x) \big\|.$
Taking $B = \C,$
we have in particular defined $\| x \|_{p, S}$ for $x \in M_d.$
\end{ntn}

\begin{ntn}\label{N_3406_Kp}
Let $d \in \N,$
let $S_1$ and $S_2$ be systems of $d$-similarities,
and let $p \in [1, \infty).$
Further let $\YCN$ be a \sfm,
and let $B \subset L (L^p (Y, \nu))$ be a unital closed subalgebra.
Using the notation of Lemma~\ref{L_3406_TensorB},
we define
\[
\kp^{p, S_2, S_1}_B
 \colon \MP{d}{p, S_1} \otimes_p B
  \to \MP{d}{p, S_2} \otimes_p B
\]
by
\[
\kp^{p, S_2, S_1}_B
 = \ps^{p, S_2}_B
    \circ \big( \ps^{p, S_1}_B \big)^{-1}.
\]
\end{ntn}

\begin{lem}\label{L_3405_Basics}
Let $d \in \N,$
and let $p \in [1, \infty).$
Let $S = (I, s, f),$
$S_1 = (I_1, s_1, f_1),$
$S_2 = (I_2, s_2, f_2),$ and $S_3 = (I_3, s_3, f_3)$
be systems of $d$-similarities,
and let $S_0 = (I_0, s_0, f_0)$ be the basic system of $d$-similarities
(with $s (0) = 1$).
Further let $\YCN$ be a \sfm,
and let $B \subset L (L^p (Y, \nu))$ be a unital closed subalgebra.
Then:
\begin{enumerate}
\item\label{L_3405_Basics_NForm}
$\big\| \ps^{p, S}_B (x) \big\|
  = \sup_{i \in I}
    \big\| (s (i) \otimes 1) \ps^{p, S_0}_B (x)
       (s (i)^{-1} \otimes 1) \big\|$
for all $x \in M_d \otimes_{\operatorname{alg}} B.$
\item\label{L_3405_Basics_FrmBase}
$\big\| \kp^{p, S, S_0}_B \big\| = R_{p, S}.$
\item\label{L_3405_Basics_ToBase}
$\big\| \kp^{p, S_0, S}_B \big\| = 1.$
\item\label{L_3405_Basics_Comp}
$\kp^{p, S_3, S_2}_B
 \circ \kp^{p, S_2, S_1}_B
 = \kp^{p, S_3, S_1}_B.$
\item\label{L_3405_Basics_Sub}
If ${\overline{\ran (s_2)}} \subset {\overline{\ran (s_1)}}$ then
$\big\| \kp^{p, S_2, S_1}_B \big\| = 1.$
\item\label{L_3405_Basics_Indep}
If ${\overline{\ran (s_1)}} = {\overline{\ran (s_2)}}$ then
$\kp^{p, S_2, S_1}_B$ is an isometric bijection.
\end{enumerate}
\end{lem}

\begin{proof}
Part~(\ref{L_3405_Basics_Comp}) is immediate.

Identify $M_d \otimes_{\operatorname{alg}} B$
with $\MP{d}{p} \otimes_p B,$
and write $\| \cdot \|_p$ for the corresponding norm.
(The identification map
is just $\ps^{p, S_0}_B.$
Thus, in the following,
we are taking $\ps^{p, S_0}_B$ to be the identity.)

For $i \in I,$
define a unital \hm{}
\[
\ps_{i, B} \colon
  \MP{d}{p} \otimes_p B
  \to L \big( L^p \big( \{ 1, 2, \ldots, d \} \times \{ i \} \times Y,
       \,  \mu_{S} \times \nu \big) \big)
\]
by making the identification~(\ref{Eq_3406_Ident})
and setting $\ps_{i, B} (x) = (s (i) \otimes 1) x (s (i)^{-1} \otimes 1)$
for $x \in \MP{d}{p} \otimes_p B.$
Using Lemma~\ref{L_3317_NConj}
and $\| y \otimes 1 \|_p = \| y \|_p$ for all $y \in \MP{d}{p},$
we get $\| \ps_{i, B} \| = \| s (i) \|_p \| s (i)^{-1} \|_p.$

Parts (\ref{L_3405_Basics_NForm})
and~(\ref{L_3405_Basics_FrmBase})
now both follow from the fact
that $\ps^{p, S}_B (x)$ is the operator
on
\[
L^p \left(
 \coprod_{i \in I} \{ 1, 2, \ldots, d \} \times \{ i \} \times Y, \,\,
 \mu_{S} \times \nu \right)
\]
which acts on
$L^p \big( \{ 1, 2, \ldots, d \} \times \{ i \} \times Y, \,
 \mu_{S} \times \nu \big)$
as $\ps_{i, B} (x),$
and whose norm is therefore $\sup_{i \in I} \| \ps_{i, B} (x) \|_p.$
Moreover,
we also have
\[
\big\| \ps^{p, S}_B (x) \big\|_{p, S}
  = \sup \big( \big\{ 
    \| (v \otimes 1) x
       (v^{-1} \otimes 1) \|_p \colon
      v \in {\overline{\ran (s)}} \big\} \big)
\]
for all $x \in M_d \otimes_{\operatorname{alg}} B.$
The inequality $\big\| \kp^{p, S_2, S_1}_B \big\| \leq 1$
in part~(\ref{L_3405_Basics_Sub})
is now immediate,
as is part~(\ref{L_3405_Basics_Indep}).
The inequality $\big\| \kp^{p, S_2, S_1}_B \big\| \geq 1$
in part~(\ref{L_3405_Basics_Sub})
follows from $\kp^{p, S_2, S_1}_B (1) = 1.$
By Definition \ref{D_3406_SysSim}(\ref{D_3406_SysSim_One}),
part~(\ref{L_3405_Basics_ToBase})
is a special case of part~(\ref{L_3405_Basics_Sub}).
\end{proof}

We get some additional properties for diagonal
systems of similarities.

\begin{lem}\label{L_3406_NormDg}
Let $d \in \N,$
let $S = (I, s, f)$ be a diagonal system of $d$-similarities,
let $S_0$ be the basic system of $d$-similarities
(Definition~\ref{D_3406_SysSim}),
and let $p \in [1, \infty).$
For $i \in I,$
define
$\af_{i, 1}, \af_{i, 2}, \ldots, \af_{i, d} \in \C \setminus \{ 0 \}$
by $s (i) = \diag ( \af_{i, 1}, \af_{i, 2}, \ldots, \af_{i, d} ).$
For $j, k = 1, 2, \ldots, d,$
define
$r_{j, k} = \sup_{i \in I} | \af_{i, j} | \cdot | \af_{i, k} |^{-1}.$
Let $\YCN$ be a \sfm,
and let $B \subset L (L^p (Y, \nu))$ be a unital closed subalgebra.
Then:
\begin{enumerate}
\item\label{L_3406_NormDg_Cj}
$( s (i) \otimes 1) (e_{j, k} \otimes b) ( s (i)^{-1} \otimes 1)
 = \af_{i, j} \af_{i, k}^{-1} e_{j, k} \otimes b$
for $i \in I,$ $b \in B,$ and $j, k = 1, 2, \ldots, d.$
\item\label{L_3406_NormDg_Om1}
$1 \leq r_{j, k} < \infty$ for $j, k = 1, 2, \ldots, d.$
\item\label{L_3406_NormDg_Omjj}
$r_{j, j} = 1$ for $j = 1, 2, \ldots, d.$
\item\label{L_3406_NormDg_Normejk}
$\| e_{j, k} \otimes b \|_{p, S} = r_{j, k} \| b \|$
for $j, k = 1, 2, \ldots, d$ and $b \in B.$
\item\label{L_3406_NormDg_Diag}
$\| e_{j, j} \otimes b \|_{p, S} = \| b \|$
for $j = 1, 2, \ldots, d$ and $b \in B.$
\item\label{L_3406_NormDg_Nrho}
$\big\| \kp^{p, S, S_0 }_B \big\|
 = R_{p, S}
 = \sup \big( \big\{ r_{j, k} \colon
  j, k \in \{ 1, 2, \ldots, d \} \big).$
\item\label{L_3406_NormDg_Sum}
Let $b_1, b_2, \ldots, b_d \in B.$
Then
\[
\big\| e_{1, 1} \otimes b_1 + e_{2, 2} \otimes b_2 + \cdots
    + e_{d, d} \otimes a_d \big\|_{p, S}
 = \max \big( \| b_1 \|, \, \| b_2 \|, \, \ldots, \, \| b_d \| \big).
\]
\end{enumerate}
\end{lem}

\begin{proof}
Part~(\ref{L_3406_NormDg_Cj}) is a calculation.
In part~(\ref{L_3406_NormDg_Om1}),
finiteness of $r_{j, k}$ follows from
compactness of ${\overline{{\mathrm{ran}} (s)}}.$
The inequality $r_{j, k} \geq 1$
follows from
\[
\max \big( | \af_{i, j} | \cdot | \af_{i, k} |^{-1},
     \, | \af_{i, j} |^{-1} \cdot | \af_{i, k} | \big)
   \geq 1.
\]
Part~(\ref{L_3406_NormDg_Omjj}) is clear.
Part~(\ref{L_3406_NormDg_Normejk})
follows from part~(\ref{L_3406_NormDg_Cj}),
Lemma \ref{L_3405_Basics}(\ref{L_3405_Basics_NForm}),
and the fact that $\| e_{j, k} \|_p = 1$ in~$\MP{d}{p}.$
Part~(\ref{L_3406_NormDg_Diag})
follows from part~(\ref{L_3406_NormDg_Normejk})
and part~(\ref{L_3406_NormDg_Omjj}).

For part~(\ref{L_3406_NormDg_Nrho}),
by Lemma \ref{L_3405_Basics}(\ref{L_3405_Basics_FrmBase})
it suffices to prove that
\[
R_{p, S}
 = \sup \big( \big\{ r_{j, k} \colon
  j, k \in \{ 1, 2, \ldots, d \} \big).
\]
This statement follows from the computation
\[
\| s (i) \|_p \| s (i)^{-1} \|_p
 = \left( \sup_{1 \leq j \leq d} | \af_{i, j} | \right)
   \left( \sup_{1 \leq k \leq d} | \af_{i, k} |^{-1} \right)
 = \sup_{1 \leq j, k \leq d} | \af_{i, j} | \cdot | \af_{i, k} |^{-1}
\]
by taking the supremum over $i \in I.$

We prove part~(\ref{L_3406_NormDg_Sum}).
Set
\[
b = e_{1, 1} \otimes b_1 + e_{2, 2} \otimes b_2 + \cdots
    + e_{d, d} \otimes b_d.
\]
Regarded as an element of $\MP{d}{p} \otimes_p B,$
the element $b$ is the $L^p$~direct sum,
over $k \in \{ 1, 2, \ldots, d \},$
of the operators $b_k$ acting on $L^p ( \{ k \} \times Y),$
so
$\| b \|_p = \sup_{k = 1, 2, \ldots, d} \| b_k \|.$
Since $S$ is diagonal,
$s (i) \otimes 1$ commutes with~$b$
for all $i \in I.$
The result now follows from
Lemma \ref{L_3405_Basics}(\ref{L_3405_Basics_NForm}).
\end{proof}

\begin{lem}\label{L_3406_TPSys}
Let $d_1, d_2 \in \N,$
let $S_1 = (I_1, s_1, f_1)$ be a system of $d_1$-similarities,
let $S_2 = (I_2, s_2, f_2)$ be a system of $d_2$-similarities,
and let $p \in [1, \infty).$
Identify $\C^{d_1} \otimes \C^{d_2}$ with $\C^{{d_1} {d_2}}$
via an isomorphism which sends tensor products of standard basis vectors
to standard basis vectors,
and use this isomorphism to identify
$M_{d_1} \otimes M_{d_2} = L (\C^{d_1} \otimes \C^{d_2})$
with~$M_{{d_1} {d_2}} = L (\C^{{d_1} {d_2}}).$
Set $I = I_1 \times I_2,$
and define $s \colon I \to M_{{d_1} {d_2}}$
and $f \colon I \to (0, 1]$ by
\[
s (i_1, i_2) = s_1 (i_1) \otimes s_2 (i_2)
\andeqn
f (i_1, i_2) = f_1 (i_1) f_2 (i_2)
\]
for $i_1 \in I_2$ and $i_2 \in I_2.$
Then $S = (I, s, f)$ is a system of $d_1 d_2$-similarities
and the same identification as already used becomes an
isometric isomorphism
$\MP{d_1}{p, S_1} \otimes_p \MP{d_2}{p, S_2} \to \MP{d_1 d_2}{p, S}.$
We have $R_{p, S} = R_{p, S_1} R_{p, S_2}.$
Moreover, if $S_1$ and $S_2$ are diagonal,
then so is~$S.$
\end{lem}

\begin{proof}
The proof is straightforward, and is omitted,
except that to prove that $R_{p, S} = R_{p, S_1} R_{p, S_2},$
we need to know that
$\| v_1 \otimes v_2 \| = \| v_1 \| \cdot \| v_2 \|$
for $v_1 \in L (L^p (X_{S_1}, \mu_{S_1} ))$
and $v_2 \in L (L^p (X_{S_2}, \mu_{S_2} )).$
For this, we use Theorem 2.16(5) of~\cite{PhLp1}.
\end{proof}

\begin{dfn}\label{D_3406_SfUHF}
Let $p \in [1, \infty),$
let $d = (d (1), \, d (2), \, \ldots )$
be a sequence of integers such that $d (n) \geq 2$ for all $n \in \N,$
and let
\[
I = (I_1, I_2, \ldots ),
\,\,\,\,\,\,
s = (s_1, s_2, \ldots ),
\andeqn
f = (f_1, f_2, \ldots )
\]
be sequences such that $S_n = (I_n, s_n, f_n)$
is a system of $d (n)$-similarities for all $n \in \N.$
We define the
{\emph{$(p, d, I, s, f)$-UHF algebra}}
$D_{p, d, I, s, f}$
and associated subalgebras
by applying the construction of Example~\ref{E_2Y14_LpUHF}
using the \rpn{s}
$\rh_n = \ps^{p, S_n} \colon
    M_{d (n)} \to L ( L^p (X_{S_n}, \mu_{S_n} ) ).$
We let $( X_{I, s, f}, \cB_{I, s, f}, \mu_{I, s, f} )$
be the product measure space $\XBM$
from Example~\ref{E_2Y14_LpUHF},
that is,
\[
( X_{I, s, f}, \cB_{I, s, f}, \mu_{I, s, f} )
 = \prod_{n = 1}^{\infty} ( X_{S_n}, \cB_{S_n}, \mu_{S_n} ).
\]
For $m, n \in \Nz$ with $m \leq n,$
we define
\[
D_{p, d, I, s, f}^{(m, n)}
 = \MP{d (m + 1)}{p, S_{m + 1}} \otimes_p
     \MP{d (m + 2)}{p, S_{m + 2}} \otimes_p
     \cdots \otimes_p
     \MP{d (n)}{p, S_n}.
\]
(This algebra is called
$A_{ {\mathbb{N}}_{\leq n}, {\mathbb{N}}_{> m} }$
in Example~\ref{E_2Y14_LpUHF}.)
We then define $D_{p, d, I, s, f}^{(m, \infty)}$
to be the direct limit
\[
D_{p, d, I, s, f}^{(m, \infty)} = \dirlim_n D_{p, d, I, s, f}^{(m, n)}
\]
using, for $n_1, n_2 \in \Nz$ with $n_2 \geq n_1 \geq m,$ the maps
$\ps_{n_2, n_1}^{(m)} \colon
   D_{p, d, I, s, f}^{(m, n_1)} \to D_{p, d, I, s, f}^{(m, n_2)}$
given by $\ps_{n_2, n_1}^{(m)} (x) = x \otimes 1$
for $x \in D_{p, d, I, s, f}^{(m, n_1)}.$
(This algebra is called
$A_{ {\mathbb{N}}, {\mathbb{N}}_{> m} }$
in Example~\ref{E_2Y14_LpUHF}.)
We regard it as
a subalgebra of $L \big( L^p ( X_{I, s, f}, \mu_{I, s, f} ) \big).$
For $n \geq m$ we let
$\ps_{\infty, n}^{(m)} \colon
   D_{p, d, I, s, f}^{(m, n)} \to D_{p, d, I, s, f}^{(m, \infty)}$
be the map associated with the direct limit.
Finally,
we set $D_{p, d, I, s, f} = D_{p, d, I, s, f}^{(0, \infty)}.$
(This algebra is called
$A (d, \rh)$
in Example~\ref{E_2Y14_LpUHF}.)
\end{dfn}

The algebra $D_{p, d, I, s, f}^{(n, n)}$ is the empty tensor product,
which we take to be~$\C.$

\begin{thm}\label{T_3406_DIsLpUHF}
Let the hypotheses and notation
be as in Definition~\ref{D_3406_SfUHF}.
Then:
\begin{enumerate}
\item\label{T_3406_DIsLpUHF_FinStg}
For $m, n \in \Nz$ with $m \leq n,$
there is a system $T$ of
$d (m + 1) d (m + 2) \cdots d (n)$-similarities
such that
there is an isometric isomorphism
\[
D_{p, d, I, s, f}^{(m, n)}
\cong M_{d (m + 1) d (m + 2) \cdots d (n)}^{p, T}
\]
which sends tensor products of standard matrix units to
standard matrix units.
Moreover, $R_{p, T} = \prod_{l = m + 1}^n R_{p, S_l}$
and, if $S_l$ is diagonal
for $l = m + 1, \, m + 2, \, \ldots, \, n,$
then $T$ is diagonal.
\item\label{T_3406_DIsLpUHF_ITPType}
For every $m \in \Nz,$
the algebra
$D_{p, d, I, s, f}^{(m, \infty)}$ is an $L^p$~UHF algebra of
infinite tensor product type
which locally has enough isometries in the sense of
Definition 2.7(2) of~\cite{PhLp2a}.
\item\label{T_3406_DIsLpUHF_Simp}
For every $m \in \Nz,$
the algebra
$D_{p, d, I, s, f}^{(m, \infty)}$
is simple and has a unique continuous normalized trace.
\item\label{T_3406_DIsLpUHF_Decomp}
For every $m \in \Nz,$
using part~(\ref{T_3406_DIsLpUHF_FinStg}) and Definition~\ref{D_3406_MP}
for the definition of the domain,
there is a unique isometric isomorphism
\[
\ph \colon
 D_{p, d, I, s, f}^{(m, n)} \otimes_p D_{p, d, I, s, f}^{(n, \infty)}
 \to D_{p, d, I, s, f}^{(m, \infty)}
\]
such that
$\ph (x \otimes y) = \ps_{\infty, n}^{(m)} (x) (1 \otimes y)$
for $x \in D_{p, d, I, s, f}^{(m, n)}$
and $y \in D_{p, d, I, s, f}^{(n, \infty)}.$
\end{enumerate}
\end{thm}

\begin{proof}
Part~(\ref{T_3406_DIsLpUHF_FinStg})
follows from
Lemma~\ref{L_3406_TPSys} by induction.

We prove parts (\ref{T_3406_DIsLpUHF_ITPType})
and~(\ref{T_3406_DIsLpUHF_Decomp}).
By renumbering, \wolog{} $m = 0.$
In Example~\ref{E_2Y14_LpUHF},
for $n \in \N$
we set $X_n = X_{S_n},$
$\mu_n = \mu_{S_n},$
and $\rh_n = \ps^{p, S_n}.$
It is easy to see that the algebra constructed there
(and represented on $L^p \left( \prod_{n = 1}^{\infty} X_n \right)$
is isometrically isomorphic to $D_{p, d, I, s, f}^{(0, \infty)}.$

For part~(\ref{T_3406_DIsLpUHF_ITPType}),
it remains to check that this algebra locally has enough isometries.
We verify Definition 2.7(2) of~\cite{PhLp2a}
by using the partition
$X_n = \coprod_{i \in I_n} Y_i$
with $Y_i = \{ 1, 2, \ldots, d (n) \} \times \{ i \}$ for $i \in I_n.$
Take $G_0$ to be the group of signed permutation matrices
(Definition~2.9 of~\cite{PhLp2a}),
and,
for $i \in I_n,$ in the application of
Definition 2.7(1) of~\cite{PhLp2a}
to $\rh_n (\cdot) |_{L^p (Y_v, \mu_n)}$
we take the finite subgroup $G$ to be $G = s_n (i)^{-1} G_0 s_n (i).$
Using Lemma 2.10 of~\cite{PhLp2a},
we easily see that $G$ acts irreducibly
and we easily get $\| \rh_n (g) |_{L^p (Y_i, \mu_{S_n})} \| = 1$
for all $g \in G.$

For part~(\ref{T_3406_DIsLpUHF_Decomp}),
set $X = \prod_{m = 1}^{n} X_{S_n},$ with product measure~$\mu,$
and $Y = \prod_{m = n + 1}^{\infty} X_{S_n},$ with product measure~$\nu.$
Applying part~(\ref{T_3406_DIsLpUHF_ITPType}) with $n$ in place of~$m,$
we obtain $D_{p, d, I, s, f}^{(n, \infty)}$ as a closed unital
subalgebra of $L (L^p (Y, \nu)).$
Set
$J = \prod_{m = 1}^n I_m,$
and define $g \colon J \to (0, 1]$
and $t \colon J \to M_{d (1) d (2) \cdots d (n)}$
by
\[
g (i_1, i_2, \ldots \otimes i_n)
  = f_1 (i_1) f_2 (i_2) \cdots f_n (i_n)
\]
and
\[
t (i_1, i_2, \ldots \otimes i_n)
  = s_1 (i_1) \otimes s_2 (i_2) \otimes \cdots \otimes s_n (i_n)
\]
for $i_k \in I_k$ for $k = 1, 2, \ldots, n.$
Then it is easy to see
(compare with part~(\ref{T_3406_DIsLpUHF_FinStg}))
that
we we can identify $D_{p, d, I, s, f}^{(m, n)}$
with $M_{d (m + 1) d (m + 2) \cdots d (n)}^{p, J, t, g},$
and that
$D_{p, d, I, s, f}^{(0, \infty)}
 \subset L \left( L^p \left( \prod_{n = 1}^{\infty} X_n \right) \right)$
is the closed linear span of all
$x \otimes y$ with
\[
x \in D_{p, d, I, s, f}^{(0, n)}
  \subset \LLp
\andeqn
y \in D_{p, d, I, s, f}^{(n, \infty)} \subset L (L^p (Y, \nu)).
\]
Part~(\ref{T_3406_DIsLpUHF_Decomp}) is now immediate.

Part~(\ref{T_3406_DIsLpUHF_Simp})
follows from part~(\ref{T_3406_DIsLpUHF_ITPType})
and from Corollary 3.12 and Theorem 3.13 of~\cite{PhLp2a}.
\end{proof}

\begin{prp}\label{P_3906_NewIso}
Let $p \in [1, \infty),$
let $d = (d (1), \, d (2), \, \ldots )$
be a sequence of integers such that $d (n) \geq 2$ for all $n \in \N,$
and let
\[
I = (I_1, I_2, \ldots ),
\,\,\,\,\,\,
s = (s_1, s_2, \ldots ),
\,\,\,\,\,\,
t = (t_1, t_2, \ldots ),
\andeqn
f = (f_1, f_2, \ldots )
\]
be sequences such that $S_n = (I_n, s_n, f_n)$
and $T_n = (I_n, t_n, f_n)$
are systems of $d (n)$-similarities for all $n \in \N.$
Assume one of the following:
\begin{enumerate}
\item\label{P_3906_NewIso_Scalar}
For every $n \in \N$ and every $i \in I_n,$
there is $\gm_{n} (i) \in \C \setminus \{ 0 \}$
such that $t_n (i) = \gm_{n} (i) s_n (i).$
\item\label{P_3906_NewIso_Isometry}
For every $n \in \N$ and every $i \in I_n,$
there is $v_n (i) \in \inv (M_{d (n)})$
such that $\| v_n (i) \|_p = \| v_n (i)^{-1} \|_p = 1$
and $t_n (i) = v_{n} (i) s_n (i).$
\item\label{P_3906_NewIso_Conj}
For every $n \in \N$ there is $w_n \in \inv (M_{d (n)})$
such that $\| w_n \|_p = \| w_n^{-1} \|_p = 1$
and for all $i \in I_n$ we have $t_n (i) = w_n s_n (i) w_n^{-1}.$
\end{enumerate}
Then $R_{p, T_n} = R_{p, S_n}$ for all $n \in \N,$
and there is an isometric isomorphism
$D_{p, d, I, t, f} \cong D_{p, d, I, s, f}.$
\end{prp}

In part~(\ref{P_3906_NewIso_Conj}),
we could take $t_n (i) = s_n (i) w_n^{-1}.$
We would then no longer need $w_n$ to be isometric,
merely invertible,
and the proof would be a bit simpler.
But this operation rarely gives
$1 \in \ran (t_n).$

\begin{proof}[Proof of Proposition~\ref{P_3906_NewIso}]
We use the notation
of Definition~\ref{D_3406_SfUHF} and Example~\ref{E_2Y14_LpUHF},
with the following modifications.
For $n \in \Nz,$
set $X_n = X_{I_n, s_n, f_n},$
which is the same space as $X_{I_n, t_n, f_n},$
and set $\mu_n = \mu_{I_n, s_n, f_n},$
which is equal to $\mu_{I_n, t_n, f_n}.$
Set $Y_n = \prod_{k = 1}^n X_n,$
and call the product measure~$\nu_n.$
Set $Z_n = \prod_{k = n + 1}^{\infty} X_n,$
and call the product measure~$\ld_n.$
Thus $Z_0 = X_{I, s, f},$
and we abbreviate this space to~$X$
and call the measure on it~$\mu.$
We set
\[
A_{n, s} = D_{p, d, I, s, f}^{(0, n)} \otimes 1_{L^p (Z_n, \ld_n)}
\andeqn
A_{n, t} = D_{p, d, I, t, f}^{(0, n)} \otimes 1_{L^p (Z_n, \ld_n)},
\]
both of which are subsets of $\LLp.$
We set $A_s = D_{p, d, I, s, f}$ and $A_t = D_{p, d, I, t, f}.$
Thus
\[
A_s = {\overline{\bigcup_{n = 0}^{\infty} A_{n, s} }}
\andeqn
A_t = {\overline{\bigcup_{n = 0}^{\infty} A_{n, t} }}.
\]

We prove the case~(\ref{P_3906_NewIso_Scalar}).
For $d \in \N,$ $s \in \inv (M_d),$ and $\gm \in \C \setminus \{ 0 \},$
we have $(\gm s) a (\gm s)^{-1} = s a s^{-1}$
for all $a \in M_d,$
and also
$\| \gm s \|_p \cdot \| (\gm s)^{-1} \|_p = \| s \|_p \cdot \| s^{-1} \|_p.$
It follows that $\ps^{p, T_n} = \ps^{p, S_n}$ for all $n \in \N.$
Therefore $A_s$ and $A_t$ are actually equal as subsets
of $\LLp.$
It is also immediate that $R_{p, T_n} = R_{p, S_n}$ for all $n \in \N.$

We next prove the case~(\ref{P_3906_NewIso_Isometry}).
(We will refer to this argument
in the proof of the case~(\ref{P_3906_NewIso_Conj}) as well.)

Let $n \in \N.$
For all $i \in I_n,$
we clearly have $\| v_n (i) s_n (i) \|_p = \| s_n (i) \|_p$
and $\| (v_n (i) s_n (i))^{-1} \|_p = \| s_n (i)^{-1} \|_p.$
So $R_{p, T_n} = R_{p, S_n}.$
For $i \in I_n,$
we interpret $v_n (i)$ as an element of
$L \big( L^p ( \{ 1, 2, \ldots, d (n) \}
         \times \{ i \}, \, \mu_{n}) \big).$
By Lemma~\ref{L_3903_ChM},
we still have $\| v_n (i) \|_p = \| v_n (i)^{-1} \|_p = 1.$
Following Remark~\ref{R_3913_DSElt},
set $v_n = \bigoplus_{i \in I} v_n (i),$
so $\| v_n \| = \| v_n^{-1} \| = 1.$

For $n \in \Nz,$
define
$z_n \in \LLp$
with respect to the decomposition
\[
L^p (X, \mu)
 = L^p (X_1, \mu_1) \otimes  L^p (X_2, \mu_2) \otimes \cdots
       \otimes  L^p (X_n, \mu_n) \otimes  L^p (Z_n, \ld_n)
\]
by
\[
z_n = v_1 \otimes v_2 \otimes \cdots
       \otimes v_n \otimes 1_{L^p (Z_n, \ld_n)}.
\]
Using Theorem 2.16(5) of~\cite{PhLp1},
we get $\| z_n \| = 1.$
An analogous tensor product decomposition gives
$\| z_n^{-1} \| = 1.$
The formula $a \mapsto z_n a z_n^{-1}$
defines a bijection $\ph_n \colon A_{n, s} \to A_{n, t},$
which is isometric because $\| z_n \| = \| z_n^{-1} \| = 1.$
Clearly $\ph_{n + 1} |_{A_{n, s}} = \ph_n.$
Therefore there exists an isometric \hm{}
$\ph \colon A_s \to A_t$
such that $\ph |_{A_{n, s}} = \ph_n$ for all $n \in \Nz.$
Clearly $\ph$ has dense range.
Therefore $\ph$ is surjective.
This completes the proof of the case~(\ref{P_3906_NewIso_Isometry}).

We now prove the case~(\ref{P_3906_NewIso_Conj}).
Let $n \in \N.$
We clearly have
\[
\| w_n s_n (i) w_n^{-1} \|_p = \| s_n (i) \|_p
\andeqn
\| (w_n s_n (i) w_n^{-1})^{-1} \|_p = \| s_n (i)^{-1} \|_p.
\]
So $R_{p, T_n} = R_{p, S_n}.$
Let $\sm_{n, s} \colon M_{r_d (n)} \to A_{n, s}$
and $\sm_{n, t} \colon M_{r_d (n)} \to A_{n, t}$
be the maps analogous to $\sm_n$ in Example~\ref{E_2Y14_LpUHF},
except that the codomains are taken to be $A_{n, s}$ instead of~$A_s$
and $A_{n, t}$ instead of~$A_t.$
In the proof of the case~(\ref{P_3906_NewIso_Isometry}),
take $v_n (i) = w_n$ for $n \in \N$ and $i \in I_n,$
and then for $n \in \N$ let $v_n$ be as there
and for $n \in \Nz$ let $z_n$ be as there.
Thus $z_n$ is a bijective isometry.
Further set
$y_n = w_1 \otimes w_2 \otimes \cdots \otimes w_n \in M_{r_d (n)}.$
Since $\sm_{n, s}$ and $\sm_{n, t}$
are bijections,
there is a unique bijection
$\bt_n \colon A_{n, s} \to A_{n, t}$
such that
$\bt_n (\sm_{n, s} (x)) = \sm_{n, t} (y_n x y_n^{-1})$
for all $x \in M_{r_d (n)}.$
Since $\sm_{n + 1, \, s} (x \otimes 1) = \sm_{n, s} (x)$
and $\sm_{n + 1, \, t} (x \otimes 1) = \sm_{n, t} (x)$
for all $x \in M_{r_d (n)},$
we get $\bt_{n + 1} |_{A_{n, s}} = \bt_n$ for all $n \in \Nz.$

We now show that $\bt_n$ is isometric for all $n \in \Nz.$
Doing so finishes the proof,
in the same way as at the end
of the proof of the case~(\ref{P_3906_NewIso_Isometry}).
Fix $n \in \Nz.$
Let $x \in \MP{r_d (n)}{p},$ and interpret $y_n$ as an element
of $\MP{r_d (n)}{p}.$
For
\[
i = (i_1, i_2, \ldots, i_n)
  \in I_1 \times I_2 \times \cdots \times I_n,
\]
make the abbreviations
\[
s (i) = s_1 (i_1) \otimes s_2 (i_2) \otimes \cdots \otimes s_n (i_n)
\andeqn
t (i) = t_1 (i_1) \otimes t_2 (i_2) \otimes \cdots \otimes t_n (i_n).
\]
Then $t (i) = y_n s (i) y_n^{-1}.$
By Lemma~\ref{L_3406_TPSys}
and Lemma \ref{L_3405_Basics}(\ref{L_3405_Basics_NForm}),
\[
\| \sm_{n, s} (x) \|
 = \sup \big( \big\{ \| s (i) x s (i)^{-1} \|_p \colon
    i \in I_1 \times I_2 \times \cdots \times I_n  \big\} \big),
\]
and similarly with $t$ in place of~$s.$
Theorem 2.16(5) of \cite{PhLp1},
applied to both $y_n$ and $y_n^{-1},$
shows that $y_n$ is isometric.
For $i \in I_1 \times I_2 \times \cdots \times I_n,$
we use this fact at the second step
and $t (i) = y_n s (i) y_n^{-1}$ at the first step
to get
\[
\big\| t (i) y_n x y_n^{-1} t (i)^{-1} \big\|_p
 = \big\| y_n s (i) x s (i)^{-1} y_n^{-1} \big\|_p
 = \| s (i) x s (i)^{-1} \|_p.
\]
Taking the supremum over
$i \in I_1 \times I_2 \times \cdots \times I_n,$
we get
$\| \sm_{n, t} (y_n x y_n^{-1}) \| = \| \sm_{n, s} (x) \|,$
as desired.
\end{proof}

Since
the Banach algebras
$\MP{d}{p, I, s, f}$ and $\MP{d}{p, I, s, f} \otimes_p B$
and only depend on ${\overline{\ran (s)}},$
we can make the following definition,
based on Example~2.15 of~\cite{PhLp2a}.

\begin{dfn}\label{D_3406_MP}
Let $d \in \N,$
let $K \subset \inv (M_d)$ be a compact set with $1 \in K,$
and let $p \in [1, \infty).$
Choose any system $S = (I, s, f)$ of $d$-similarities
such that ${\overline{\ran (s)}} = K.$
(It is obvious that there is such a system of $d$-similarities.)
We then define $\MP{d}{p, K} = \MP{d}{p, S}$ as a Banach algebra.
For any \sfm{} $\YCN$
and any unital closed subalgebra $B \subset L (L^p (Y, \nu)),$
we define $\MP{d}{p, K} \otimes_p B = \MP{d}{p, S} \otimes_p B$
as a Banach algebra.
We define
\[
\| \cdot \|_{p, K} = \| \cdot \|_{p, S},
\,\,\,\,\,\,
\ps^{p, K}_B = \ps^{p, S}_B,
\andeqn
R_{p, K} = R_{p, S}.
\]
If $K = \{ 1 \},$
taking $S$ to be the basic system $S_0$ of $d$-similarities
(Definition~\ref{D_3406_SysSim}),
we get $\MP{d}{p, \{ 1 \} } = \MP{d}{p}$
and $\ps_B^{p, \{ 1 \} } = \ps_B^{p, S_0 }.$
If $K_1, K_2 \subset \inv (M_d)$ are compact sets
which contain~$1,$
then we choose systems
$S_1 = (I_1, s_1, f_1)$ and $S_2 = (I_2, s_2, f_2)$ of $d$-similarities
such that ${\overline{\ran (s_1)}} = K_1$
and ${\overline{\ran (s_2)}} = K_2,$
and define
\[
\kp_B^{p, K_2, K_1}
 = \kp^{p, S_2, S_1}_B
  \colon \MP{d}{p, K_1} \otimes_p B \to \MP{d}{p, K_2} \otimes_p B.
\]
We say that $K$ is {\emph{diagonal}} if $K$
is contained in the diagonal matrices in~$M_d.$
\end{dfn}

\begin{dfn}\label{D_3406_KUHF}
Let $p \in [1, \infty),$
let $d = (d (1), \, d (2), \, \ldots )$
be a sequence of integers such that $d (n) \geq 2$ for all $n \in \N,$
and let $K = (K_1, \, K_2, \, \ldots )$
be a sequence of compact subsets $K_n \subset \inv (M_{d (n)} )$
with $1 \in K_n$
for $n \in \N.$
We define the
{\emph{$(p, d, K)$-UHF algebra}}
$D_{p, d, K}$
and associated subalgebras as follows.
Choose any sequences $I = (I_1, I_2, \ldots ),$ $s = (s_1, s_2, \ldots ),$
and $f = (f_1, f_2, \ldots )$
such that
for all $n \in \N,$ the triple $S_n = (I_n, s_n, f_n)$
is a system of $d (n)$-similarities
with ${\overline{ \ran (s_n)}} = K_n.$
Then, following Definition~\ref{D_3406_SfUHF}, define
$D_{p, d, K}^{(m, n)} = D_{p, d, I, s, f}^{(m, n)}$
for $m, n \in \Nz \cup \{ \infty \}$ with $m \leq n$ and $m \neq \infty,$
define
$\ps_{\infty, n}^{(m)} \colon
   D_{p, d, K}^{(m, n)} \to D_{p, d, K}^{(m, \infty)}$
as in Definition~\ref{D_3406_SfUHF},
and define $D_{p, d, K} = D_{p, d, I, s, f}^{(0, \infty)}.$
\end{dfn}

One can prove that for algebras of the form $D_{p, d, K}$
as in Definition~\ref{D_3406_KUHF},
the conditions
(\ref{T_2Z29_First_IsoSp}),
(\ref{T_2Z29_First_ClRange}),
(\ref{T_2Z29_First_AIF}),
(\ref{T_2Z29_First_Half}),
(\ref{T_2Z29_First_SmBdd}),
and (\ref{T_2Z29_First_SumFin})
in Theorem~\ref{T_2Z29_First}
are also equivalent to the following:
\begin{enumerate}
\setcounter{enumi}{\value{RsvEnumi}}
\item\label{T_3708_Second_SmBddOne}
There is a uniform bound on the norms of the maps
$\sm_n \colon \MP{r_d (n)}{p} \to A.$
\item\label{T_3708_Second_RhBdd}
$\sum_{n = 1}^{\infty} ( \| \rh_n \| - 1 ) < \infty.$
\end{enumerate}

The new feature
(which will be made explicit in a more restrictive context
in Theorem~\ref{T_3708_Third} below)
is that there are a spatial \rpn{} $\rh_n^{(0)}$ of $M_{d_n}$
and an invertible operator~$w$
with $\| w \| \cdot \| w^{-1} \| = \| \rh_n \|$
such that $\rh_n (x) = w \rh_n^{(0)} (x) w^{-1}$
for all $x \in M_{d (n)}.$
Since $\rh_n^{(0)} \otimes \rh_n^{(0)}$ is again spatial
(Lemma~1.12 of~\cite{PhLp2a}) and
\[
(\rh_n \otimes_p \rh_n) (x)
  = (w \otimes w) \big( \rh_n^{(0)} \otimes \rh_n^{(0)} \big) (x)
         (w \otimes w)^{-1}
\]
for $x \in M_{d (n)} \otimes_p M_{d (n)},$
we can use Lemma~\ref{L_3317_NConj}
to get $\| \rh_n \otimes_p \rh_n \| = \| \rh_n \|^2,$
rather than merely $\| \rh_n \otimes_p \rh_n \| \geq \| \rh_n \|^2.$
Similarly,
we get $\| \sm_n \otimes_p \sm_n \| = \| \sm_n \|^2.$

\section{Amenability of $L^p$~UHF algebras constructed from
 diagonal similarities}\label{Sec_Amen}

\indent
Let $A \subset \LLp$ be an $L^p$~UHF algebra of tensor product type
constructed using a system of diagonal similarities.
The main result of this section
(Theorem~\ref{T_3708_Third})
is that a number of conditions,
of which the most interesting is probably amenability,
are equivalent to $A$ being isomorphic to the
spatial $L^p$~UHF algebra~$B$ with the same supernatural number.
If there is an isomorphism,
we can in fact realize $B$ as a subalgebra of $\LLp,$
in such a way that the isomorphism can be taken to be
given by conjugation by an invertible element in $\LLp.$

When $p = 2,$
the conditions can be relaxed:
we do not need to assume that the similarities are diagonal.

The key technical ideas are information about the form
of an approximate diagonal,
and an estimate on its norm based on information about
the norms of off diagonal matrix units after conjugating
by an invertible element.
We only know sufficiently good estimates for
conjugation by diagonal matrices,
which is why we restrict to diagonal similarities.

\begin{lem}\label{L_3330_Exact}
Let $M \in [1, \infty).$
Let $B$ be a unital Banach algebra which has
an approximate diagonal with norm at most~$M.$
Let $G \subset {\operatorname{inv}} (B)$ be a finite subgroup.
Then for every $\ep > 0$ there exists
$z \in B {\widehat{\otimes}} B$
such that
$\| z \|_{\pi} < M + \ep,$
$\Dt_B (z) = 1,$
and $(g \otimes 1) z = z (1 \otimes g)$
for all $g \in G.$
\end{lem}

The proof is easily modified to show that we can also require that,
for all $a$ in a given finite set $F \subset A,$
we have
$\| (a \otimes 1) z - z (1 \otimes a) \| < \ep.$
We don't need this refinement here.

\begin{proof}[Proof of Lemma~\ref{L_3330_Exact}]
Using the definition of $\Dt_B$
(see Lemma~\ref{L_2Y25_PjTMaps}),
for $a, b \in B$ and $x \in B {\widehat{\otimes}} B,$
we get
\begin{equation}\label{Eq_3330_DtMod}
\Dt_B \big( ( a \otimes 1) x (1 \otimes b) \big) = a \Dt_B (x) b.
\end{equation}

Set $r = \sup_{g \in G} \| g \|.$
Set
\[
\dt = \min \left( \frac{1}{2}, \, \frac{\ep}{4 M}, \,
          \frac{\ep}{4 r}, \, \frac{\ep}{8 r^2} \right).
\]
By hypothesis, there exists $z_0 \in B {\widehat{\otimes}} B$
such that
$\| z_0 \|_{\pi} \leq M,$
$\| \Dt_B (z_0) - 1 \| < \dt,$
and $\big\| (g \otimes 1) z_0 - z_0 (1 \otimes g) \big\|_{\pi} < \dt$
for all $g \in G.$

Since $\dt \leq \frac{1}{2},$
the element $\Dt_B (z_0)$ is invertible in~$B$
and
\[
\big\| \Dt_B (z_0)^{-1} - 1 \big\|
  < \frac{\dt}{1 - \dt}
  \leq 2 \dt.
\]
Set $z_1 = (\Dt_B (z_0)^{-1} \otimes 1) z_0.$
Then
$\Dt_B (z_1) = 1$ by~(\ref{Eq_3330_DtMod}).
Further,
\[
\| z_1 - z_0 \|
  \leq \| \Dt_B (z_0)^{-1} \otimes 1 \| \cdot \| z_0 \|
  < 2 \dt M.
\]
Since $2 \dt M \leq \frac{\ep}{2},$
we get $\| z_1 \| < M + \frac{\ep}{2}.$
Also, for $g \in G$ we have
\begin{align}\label{Eq_3330_CommEst}
\big\| (g \otimes 1) z_1 - z_1 (1 \otimes g) \big\|_{\pi}
& \leq \big\| (g \otimes 1) z_0 - z_0 (1 \otimes g) \big\|_{\pi}
           + 2 \| g \| \cdot \| z_1 - z_0 \|
         \\
& < \dt + 2 r \left( \frac{\ep}{8 r^2} \right)
  \leq \frac{\ep}{2 r}.
         \notag
\end{align}

Now define
\[
z = \frac{1}{\card (G)}
     \sum_{h \in G} (h \otimes 1) z_0 (1 \otimes h)^{-1}.
\]
{}From~(\ref{Eq_3330_CommEst}),
for $g \in G$ we get
\[
\big\| (g \otimes 1) z_1 (1 \otimes g)^{-1} - z_1 \big\|_{\pi}
  \leq \big\| (g \otimes 1) z_1 - z_1 (1 \otimes g) \big\|_{\pi}
          \cdot \| (1 \otimes g)^{-1} \|_{\pi}
  < \left( \frac{\ep}{2 r} \right) r
  \leq \frac{\ep}{2}.
\]
It follows that $\| z - z_1 \| < \frac{\ep}{2},$
whence $\| z \| < M + \ep.$

For $g \in G$ we also get
\begin{align*}
(g \otimes 1) z
& = \frac{1}{\card (G)}
     \sum_{h \in G} (g h \otimes 1) z_0 (1 \otimes h)^{-1}
        \\
& = \frac{1}{\card (G)}
     \sum_{h \in G} (h \otimes 1) z_0 (1 \otimes g^{-1} h)^{-1}
  = z (1 \otimes g).
\end{align*}

Finally, using $\Dt_B (z_1) = 1$ and~(\ref{Eq_3330_DtMod}),
we have
\[
\Dt_B (z)
 = \frac{1}{\card (G)} \sum_{h \in G} h \Dt_B (z_1) h^{-1}
 = 1.
\]
This completes the proof.
\end{proof}

\begin{lem}\label{L_3330_TPerm}
Let $A_1$ and $A_2$ be unital Banach algebras,
let $d_1, d_2 \in \N,$
equip $M_{d_1}$ and $M_{d_2}$ with any algebra norms,
and equip $B_1 = M_{d_1} \otimes A_1$
and $B_2 = M_{d_2} \otimes A_2$ with any algebra tensor norm.
Then $B_1$ and $B_2$ are complete,
and there is a unique algebra bijection
\[
\ph \colon M_{d_1} \otimes_{\mathrm{alg}} M_{d_2} \otimes_{\mathrm{alg}}
            \big( A_1 {\widehat{\otimes}} A_2 \big)
       \to B_1 {\widehat{\otimes}} B_2
\]
such that
\[
\ph \big( x_1 \otimes x_2
      \otimes \big( a_1 {\widehat{\otimes}} a_2 \big) \big)
  = (x_1 \otimes a_1) {\widehat{\otimes}} (x_2 \otimes a_2)
\]
for all
$x_1 \in M_{d_1},$ $x_2 \in M_{d_2},$ $a_1 \in A_1,$
and $a_2 \in A_2.$
\end{lem}

\begin{ntn}\label{N_3330_FPrm}
We will often implicitly use the isomorphism~$\ph$
of Lemma~\ref{L_3330_TPerm}
to write particular elements of $B_1 {\widehat{\otimes}} B_2$
as there in the form
$x_1 \otimes x_2 \otimes a$ with
$x_1 \in M_{d_1},$ $x_2 \in M_{d_2},$ and
$a \in A {\widehat{\otimes}} A,$
or
(using Notation~\ref{N_3406_MatU})
a general element of $B_1 {\widehat{\otimes}} B_2$ as
\[
\sum_{j, k = 1}^{d_1} \sum_{l, m = 1}^{d_2}
  e_{j, k} \otimes e_{l, m} \otimes a_{j, k, l, m}
\]
with $a_{j, k, l, m} \in A_1 {\widehat{\otimes}} A_2$
for $j, k = 1, 2, \ldots, d_1$
and $l, m = 1, 2, \ldots, d_2.$
\end{ntn}

\begin{proof}[Proof of Lemma~\ref{L_3330_TPerm}]
Completeness of $B_1$ and $B_2$
follows from \fd{ity} of $M_{d_1}$ and $M_{d_2}.$
Also by \fd{ity} of $M_{d_1},$
any two algebra tensor norms on $B_1$ are equivalent.
Thus, we may take $B_1 = M_{d_1} {\widehat{\otimes}} A_1.$
Similarly, we may take $B_2 = M_{d_2} {\widehat{\otimes}} A_2,$
and make the identification of algebras
\[
M_{d_1} \otimes_{\mathrm{alg}} M_{d_2} \otimes_{\mathrm{alg}}
            \big( A_1 {\widehat{\otimes}} A_2 \big)
     = M_{d_1} {\widehat{\otimes}} M_{d_2} {\widehat{\otimes}} 
            A_1 {\widehat{\otimes}} A_2.
\]
Now $\ph$ is just the permutation of projective tensor factors isomorphism
\[
\ph \colon
  M_{d_1} {\widehat{\otimes}} M_{d_2} {\widehat{\otimes}} 
            A_1 {\widehat{\otimes}} A_2
 \to M_{d_1} {\widehat{\otimes}} A_1 {\widehat{\otimes}} 
            M_{d_2} {\widehat{\otimes}} A_2.
\]
This completes the proof.
\end{proof}

The following lemma was suggested by the last paragraph
in Section~7.5
of~\cite{BlLM}.

\begin{lem}\label{L_3330_Diag}
Let $A$ be a unital Banach algebra,
let $d \in \N,$
equip $M_d$ with any algebra norm,
and equip $B = M_d \otimes A$ with any algebra tensor norm.
Suppose $z \in B {\widehat{\otimes}} B$
satisfies $\Dt_B (z) = 1$
and
\begin{equation}\label{Eq_3331_x1z}
(x \otimes 1_A \otimes 1_B) z = z (1_B \otimes x \otimes 1_A)
\end{equation}
for all $x \in M_d.$
Then, rearranging tensor factors as in Notation~\ref{N_3330_FPrm},
there exist elements $z_{j, k} \in A {\widehat{\otimes}} A$
for $j, k = 1, 2, \ldots, d$
such that
\[
z = \sum_{j, k, l = 1}^d e_{j, k} \otimes e_{l, j} \otimes z_{l, k}
\andeqn
\sum_{j = 1}^d \Dt_A (z_{j, j}) = 1.
\]
\end{lem}

\begin{proof}
Use Lemma~\ref{L_3330_TPerm}
and Notation~\ref{N_3330_FPrm}
to write
\[
z =
\sum_{j, k, l, m = 1}^d e_{j, k} \otimes e_{l, m} \otimes a_{j, k, l, m}
\]
with $a_{j, k, l, m} \in A {\widehat{\otimes}} A$
for $j, k, l, m = 1, 2, \ldots, d.$
Let $r, s \in \{ 1, 2, \ldots, d \}.$
In~(\ref{Eq_3331_x1z})
put $x = e_{s, r},$
getting
\[
\sum_{k, l, m = 1}^d
        e_{s, k} \otimes e_{l, m} \otimes a_{r, k, l, m}
 - \sum_{j, k, l = 1}^d
        e_{j, k} \otimes e_{l, r} \otimes a_{j, k, l, s}
= 0.
\]
Let $p \in \{ 1, 2, \ldots, d \}$ and multiply on the right
by $e_{p, p} \otimes e_{r, r} \otimes 1_{A {\widehat{\otimes}} A},$
getting
\begin{equation}\label{Eq_3331_zprs}
\sum_{l = 1}^d e_{s, p} \otimes e_{l, r} \otimes a_{r, p, l, r}
 - \sum_{j, l = 1}^d
        e_{j, p} \otimes e_{l, r} \otimes a_{j, p, l, s}
= 0.
\end{equation}

Let $q \in \{ 1, 2, \ldots, d \}.$
First, assume $r \neq s$
and multiply~(\ref{Eq_3331_zprs}) on the left by
$e_{r, r} \otimes e_{q, q} \otimes 1_{A {\widehat{\otimes}} A}.$
The first sum is annihilated,
and we get
\[
- e_{r, p} \otimes e_{q, r} \otimes a_{r, p, q, s}
= 0.
\]
Thus (by injectivity in Lemma~\ref{L_3330_TPerm}),
we have $a_{r, p, q, s} = 0$ whenever
$p, q, r, s \in \{ 1, 2, \ldots, d \}$ satisfy $r \neq s.$
For arbitrary $r, s \in \{ 1, 2, \ldots, d \},$
multiply~(\ref{Eq_3331_zprs}) on the left by
$e_{s, s} \otimes e_{q, q} \otimes 1_{A {\widehat{\otimes}} A}.$
This gives
\[
e_{s, p} \otimes e_{q, r} \otimes a_{r, p, q, r}
  - e_{s, p} \otimes e_{q, r} \otimes a_{s, p, q, s} = 0.
\]
Therefore $a_{r, p, q, r} = a_{s, p, q, s}$
for all $p, q, r, s \in \{ 1, 2, \ldots, d \}.$
With $z_{l, k} = a_{1, k, l, 1}$
for $k, l = 1, 2, \ldots, d,$
we therefore get
\[
z = \sum_{j, k, l = 1}^d e_{j, k} \otimes e_{l, j} \otimes z_{l, k}.
\]
Apply $\Dt_B$ to this equation,
using $e_{j, k} e_{l, j} = 0$ when $k \neq l$
and $e_{j, k} e_{k, j} = e_{j, j},$
to get
\[
1 = \Dt_B (z)
  = \sum_{j, k = 1}^d e_{j, j} \otimes \Dt_A ( z_{k, k} ).
\]
So $\sum_{k = 1}^d \Dt_A (z_{k, k}) = 1.$
\end{proof}

\begin{lem}\label{L_3330_Ineq}
Let $d \in \N,$
let $\af_1, \af_2, \ldots, \af_d \in \C \setminus \{ 0 \},$
and set
\[
\bt = \min \big( | \af_1 |, | \af_2 |, \ldots, | \af_d | \big)
\andeqn
\gm = \max \big( | \af_1 |, | \af_2 |, \ldots, | \af_d | \big).
\]
Let $E$ be a normed vector space,
and let $\xi_1, \xi_2, \ldots, \xi_d \in E$
satisfy $\left\| \sum_{j = 1}^d \xi_j \right\| = 1.$
Then there exist
\[
\zt_1, \zt_2, \ldots, \zt_d \in S^1
   = \{ \zt \in \C \colon | \zt | = 1 \}
\andeqn
j_0 \in \{ 1, 2, \ldots, d \}
\]
such that
\begin{equation}\label{Eq_3331_Star}
\left\| \sum_{k = 1}^d
  (\zt_{j_0} \af_{j_0}) (\zt_{k} \af_{k})^{-1} \xi_k \right\|
 \geq \bt^{- 1/2} \gm^{1/2}
\end{equation}
or
\begin{equation}\label{Eq_3331_StSt}
\left\| \sum_{k = 1}^d
  (\zt_{j_0} \af_{j_0})^{-1} (\zt_{k} \af_{k}) \xi_k \right\|
 \geq \bt^{- 1/2} \gm^{1/2}.
\end{equation}
\end{lem}

\begin{proof}
We claim that it suffices to consider the case
$\af_j > 0$ for $j = 1, 2, \ldots, d.$
To see this, assume that the result has been proved
for $| \af_1 |, | \af_2 |, \ldots, | \af_d |,$
yielding $\zt_1^{(0)}, \zt_2^{(0)}, \ldots, \zt_d^{(0)} \in S^1.$
Then the result for $\af_1, \af_2, \ldots, \af_d$
follows by taking $\zt_j = {\overline{\sgn (\af_j)}} \zt_j^{(0)}$
for $j = 1, 2, \ldots, d.$
We therefore assume that $\af_j > 0$ for $j = 1, 2, \ldots, d.$

Both (\ref{Eq_3331_Star}) and~(\ref{Eq_3331_StSt})
are unchanged if we choose any $\rh > 0$
and replace $\af_j$ by $\rh \af_j$ for $j = 1, 2, \ldots, d.$
We may therefore assume that $\bt = 1.$
Reordering the $\af_j$ and the $\xi_j,$
we may assume that
$1 = \af_1 \leq \af_2 \leq \cdots \leq \af_d = \gm.$

By the Hahn-Banach Theorem,
there is a linear functional $\om \colon E \to \C$
such that
\[
\| \om \| = 1
\andeqn
\sum_{j = 1}^d \om (\xi_j) = 1.
\]
Set $\sm_j = {\overline{\sgn (\om (\xi_j))}}$
and $\ld_j = | \om (\xi_j) |$
for $j = 1, 2, \ldots, d.$
Then $\sum_{j = 1}^d \ld_j \geq 1.$

Applying $\om$ at the first step in both the following calculations,
and using $\af_1 = 1$ in the first and $\af_d = \gm$ in the second,
we get
\begin{align}\label{Eq_3406_Bottom}
\left\| \sum_{k = 1}^d
  (\sm_{1} \af_{1})^{-1} (\sm_{k} \af_{k}) \xi_k \right\|
& \geq \left| \sum_{k = 1}^d
      (\sm_{1} \af_{1})^{-1} (\sm_{k} \af_{k}) \om (\xi_k) \right|
\\
& = | \sm_{1}^{-1} | \left| \sum_{k = 1}^d
      \af_{k} \sm_k \om (\xi_k) \right|
  = \sum_{k = 1}^d \af_{k} \ld_k      \notag
\end{align}
and
\begin{align}\label{Eq_3406_Top}
\frac{1}{\gm} \left\| \sum_{k = 1}^d
  \big( {\overline{\sm_{d}}} \af_{d} \big)
     \big( {\overline{\sm_{k}}} \af_{k} \big)^{-1} \xi_k \right\|
& \geq \frac{1}{\gm} \left| \sum_{k = 1}^d
       \big( {\overline{\sm_{d}}} \af_{d} \big)
     \big( {\overline{\sm_{k}}} \af_{k} \big)^{-1} \om (\xi_k) \right|
\\
& = \gm^{-1} | {\overline{\sm_{d}}} | \af_d \left| \sum_{k = 1}^d
      \af_{k}^{-1} \sm_k \om (\xi_k) \right|
  = \sum_{k = 1}^d \af_{k}^{-1} \ld_k.   \notag
\end{align}
Now, using the inequality $\af + \af^{-1} \geq 2$ for $\af > 0$
at the third step,
we get
\begin{align*}
\left( \sum_{k = 1}^d \af_{k} \ld_k \right)
  \left( \sum_{k = 1}^d \af_{k}^{-1} \ld_k \right)
& = \sum_{k = 1}^d \af_{k} \ld_k \cdot \af_{k}^{-1} \ld_k
    + \sum_{k = 1}^d \sum_{j \neq k}
          \af_{k} \ld_k \cdot \af_{j}^{-1} \ld_j
\\
& = \sum_{k = 1}^d \ld_k^2
       + \sum_{k = 1}^d \sum_{j = 1}^{k - 1}
               \big( \af_{k} \af_{j}^{-1} + \af_{k}^{-1} \af_{j}  \big)
               \ld_j \ld_k
\\
& \geq \sum_{k = 1}^d \ld_k^2
       + \sum_{k = 1}^d \sum_{j = 1}^{k - 1} 2 \ld_j \ld_k
  = \left( \sum_{k = 1}^d \ld_k \right)^2
  \geq 1.
\end{align*}
Combining this result with
(\ref{Eq_3406_Bottom}) and~(\ref{Eq_3406_Top}),
we get
\[
%
% \begin{equation}\label{Eq_3331_Or}
\left\| \sum_{k = 1}^d
  (\sm_{1} \af_{1})^{-1} (\sm_{k} \af_{k}) \xi_k \right\|
  \geq \gm^{1/2}
\,\,\,\,\,\,
{\text{or}}
\,\,\,\,\,\,
\frac{1}{\gm} \left\| \sum_{k = 1}^d
  \big( {\overline{\sm_{d}}} \af_{d} \big)
     \big( {\overline{\sm_{k}}} \af_{k} \big)^{-1} \xi_k \right\|
  \geq \gm^{-1/2}.
% \end{equation}
%
\]
In the first case,
we get the conclusion of the lemma by choosing $j_0 = 1$
and $\zt_j = \sm_j$ for $j = 1, 2, \ldots, d.$
In the second case,
we get the conclusion by choosing $j_0 = d$
and $\zt_j = {\overline{\sm_j}}$ for $j = 1, 2, \ldots, d.$
\end{proof}

\begin{lem}\label{L_3330_LEst}
Let $d \in \N,$
let $S = (I, s, f)$ be a diagonal system of $d$-similarities,
and let $p \in [1, \infty).$
Let $\YCN$ be a \sfm,
and let $A \subset L (L^p (Y, \nu))$ be a unital closed subalgebra.
Set $B = \MP{d}{p, S} \otimes_p A.$
Suppose $z \in B {\widehat{\otimes}} B$
satisfies $\Dt_B (z) = 1$
and
\[
(x \otimes 1_A \otimes 1_B) z = z (1_B \otimes x \otimes 1_A)
\]
for all $x \in \MP{d}{p, S}.$
Then $\| z \|_{S, \pi} \geq R_{p, S}^{1/2}.$
\end{lem}

\begin{proof}
Set $B_0 = \MP{d}{p} \otimes_p A.$
Let $i \in I$; we show that
$\| z \|_{S, \pi} \geq \| s (i) \|_p^{1/2} \| s (i)^{-1} \|_p^{1/2}.$
Let $S_0$ be the basic system of $d$-similarities
(Definition~\ref{D_3406_SysSim}).

Adopt Notation~\ref{N_3330_FPrm},
and use Lemma~\ref{L_3330_Diag} to find $z_{j, k} \in A {\widehat{\otimes}} A$
for $j, k = 1, 2, \ldots, d$
such that
\[
z = \sum_{j, k, l = 1}^d e_{j, k} \otimes e_{l, j} \otimes z_{l, k}
\andeqn
\sum_{j = 1}^d \Dt_A (z_{j, j}) = 1.
\]
There are $\af_1, \af_2, \ldots, \af_d \in \C \setminus \{ 0 \}$
such that $s (i) = \diag ( \af_1, \af_2, \ldots, \af_d ).$
Apply Lemma~\ref{L_3330_Ineq}
with $E = A,$
with $\af_1, \af_2, \ldots, \af_d$ as above,
and with $\xi_j = \Dt ( z_{j, j} )$
for $j = 1, 2, \ldots, d.$
Then $\bt = \| s (i)^{-1} \|_p^{-1}$ and $\gm = \| s (i) \|_p.$
We obtain $\zt_1, \zt_2, \ldots, \zt_d \in S^1$
and $j_0 \in \{ 1, 2, \ldots, d \}$
such that
\begin{equation}\label{Eq_3406_Star}
\left\| \sum_{k = 1}^d
  (\zt_{j_0} \af_{j_0}) (\zt_{k} \af_{k})^{-1} \Dt ( z_{k, k} ) \right\|
 \geq \| s (i) \|_p^{1/2} \| s (i)^{-1} \|_p^{1/2}
\end{equation}
or
\begin{equation}\label{Eq_3406_StSt}
\left\| \sum_{k = 1}^d
  (\zt_{j_0} \af_{j_0})^{-1} (\zt_{k} \af_{k}) \Dt ( z_{k, k} ) \right\|
 \geq \| s (i) \|_p^{1/2} \| s (i)^{-1} \|_p^{1/2}.
\end{equation}
Define $u = \diag ( \zt_1, \zt_2, \ldots, \zt_d )$
and $w = u s (i).$
Then $\| u \|_p = \| u^{-1} \|_p = 1,$
so $\| u \otimes 1 \|_p = \| u^{-1} \otimes 1 \|_p = 1.$

Since $B_0 = \MP{d}{p} \otimes_p A$ and $B = \MP{d}{p, S} \otimes_p A$
are just $M_d \otimes_{\operatorname{alg}} A$ as algebras
(by Lemma~\ref{L_3406_TensorB}),
the formula $b \mapsto (s (i) \otimes 1) b (s (i)^{-1} \otimes 1)$
defines a \hm{}
$\ph_0 \colon B \to B_0.$
It follows from Lemma~\ref{L_3405_Basics}(\ref{L_3405_Basics_NForm})
that $\| \ph_0 (b) \|_{p, S_0} \leq \| b \|_{p, S}$
for all $b \in B.$
Therefore also the formula
\[
\ph (b) = (w \otimes 1) b (w^{-1} \otimes 1)
        = (u \otimes 1) \ph_0 (b) (u^{-1} \otimes 1)
\]
defines a contractive \hm{}
$\ph \colon B \to B_0.$
Also, $b \mapsto b$ defines a contractive \hm{} from $B$ to~$B_0,$
namely the map $\kp^{p, S_0, S}_B$
of Notation~\ref{N_3406_Kp}
and Lemma~\ref{L_3405_Basics}(\ref{L_3405_Basics_ToBase}).
The projective tensor product of contractive linear maps is
contractive,
and $\Dt_{B_0} \colon B_0 {\widehat{\otimes}} B_0 \to B_0$
is contractive.
So there are contractive linear maps
$\Dt_1, \Dt_2 \colon B {\widehat{\otimes}} B \to B_0$
such that for all $b_1, b_2 \in B$ we have
the following formulas
(in each case, the first one justifies contractivity
and the second one is in terms of what happens
in $M_d \otimes_{\operatorname{alg}} B$):
\[
\Dt_1 (b_1 \otimes b_2)
 = \Dt_{B_0} \big( \ph_0 (b_1) \otimes
             \kp^{p, S_0, S}_B (b_2) \big)
 = (w \otimes 1) b_1 (w^{-1} \otimes 1) b_2
\]
and
\[
\Dt_2 (b_1 \otimes b_2)
 = \Dt_{B_0} \big( \kp^{p, S_0, S}_B (b_1) \otimes
             \ph_0 (b_2) \big)
 = b_1 (w \otimes 1) b_2 (w^{-1} \otimes 1).
\]

We evaluate $\Dt_1 (z)$ and $\Dt_2 (z).$
For $j, k, l = 1, 2, \ldots, d,$ we have
\[
w e_{j, k} w^{-1} e_{l, j}
 = \begin{cases}
   0                                               & l \neq k
        \\
 (\zt_{j} \af_{j}) (\zt_{k} \af_{k})^{-1} e_{j, j} & l = k.
\end{cases}
\]
Therefore
\[
\Dt_1 (z)
 = \sum_{j, k, l = 1}^d
    w e_{j, k} w^{-1} e_{l, j} \otimes \Dt_A (z_{l, k})
 = \sum_{j = 1}^d e_{j, j} \otimes
    \sum_{k = 1}^d (\zt_{j} \af_{j})
       (\zt_{k} \af_{k})^{-1} \Dt_A (z_{k, k}).
\]
Similarly, one gets
\[
\Dt_2 (z)
 = \sum_{j = 1}^d e_{j, j} \otimes
    \sum_{k = 1}^d (\zt_{j} \af_{j})^{-1}
       (\zt_{k} \af_{k}) \Dt_A (z_{k, k}).
\]

If~(\ref{Eq_3406_Star}) holds,
then we use~(\ref{Eq_3406_Star})
and $\| e_{j_0, j_0} \otimes 1 \|_p = \| e_{j_0, j_0} \|_p = 1$
to get
\begin{align*}
\| z \|_{S, \pi}
& \geq \| \Dt_1 (z) \|
  \geq \big\| (e_{j_0, j_0} \otimes 1) \Dt_1 (z)
                (e_{j_0, j_0} \otimes 1) \big\|
\\
& = \left\| \sum_{k = 1}^d
  (\zt_{j_0} \af_{j_0}) (\zt_{k} \af_{k})^{-1} \Dt ( z_{k, k} ) \right\|
 \geq \| s (i) \|_p^{1/2} \| s (i)^{-1} \|_p^{1/2}.
\end{align*}
If instead~(\ref{Eq_3406_StSt}) holds,
we similarly get
\[
\| z \|_{S, \pi}
  \geq \| \Dt_2 (z) \|
  \geq \big\| (e_{j_0, j_0} \otimes 1) \Dt_1 (z)
                (e_{j_0, j_0} \otimes 1) \big\|
 \geq \| s (i) \|_p^{1/2} \| s (i)^{-1} \|_p^{1/2}.
\]
This completes the proof.
\end{proof}

\begin{thm}\label{T_3330_ConseqAmen}
Let $p \in [1, \infty),$
let $d = (d (1), \, d (2), \, \ldots )$
be a sequence of integers such that $d (n) \geq 2$ for all $n \in \N,$
and let
\[
I = (I_1, I_2, \ldots ),
\,\,\,\,\,\,
s = (s_1, s_2, \ldots ),
\andeqn
f = (f_1, f_2, \ldots )
\]
be sequences such that $S_n = (I_n, s_n, f_n)$
is a system of $d (n)$-similarities for all $n \in \N.$
Suppose the algebra $D_{p, d, I, s, f}$
of Definition~\ref{D_3406_KUHF}
is amenable.
Then $\prod_{n = 1}^{\infty} R_{p, S_n} < \infty.$
\end{thm}

\begin{proof}
The hypothesis means that there is $M \in [1, \infty)$
such that $D_{p, d, I, s, f}$ has
an approximate diagonal with norm at most~$M.$
Let $n \in \N.$
We show that $\prod_{l = 1}^{n} R_{p, S_l} \leq (M + 1)^2.$
This will prove the theorem.

Set $r = d (1) d (2) \cdots d (n).$
Parts (\ref{T_3406_DIsLpUHF_FinStg})
and~(\ref{T_3406_DIsLpUHF_Decomp}) of Theorem~\ref{T_3406_DIsLpUHF}
provide a diagonal system $T = (J, t, g)$ of $r$-similarities,
a \sfm{} $\YCN,$
and a unital closed subalgebra
$A \subset L (L^p (Y, \nu))$
(namely $A = D_{p, d, I, s, f}^{n, \infty}$),
such that $R_{p, T} = \prod_{l = 1}^n R_{p, S_l}$
and $D_{p, d, I, s, f}$ is isometrically isomorphic to
$\MP{r}{p, T} \otimes_p A.$
Let $G_0 \subset \inv \big( \MP{r}{p, L} \big)$
be the group of signed permutation matrices
(Definition~2.9 of~\cite{PhLp2a}),
and set
\[
G = \{ g \otimes 1_A \colon g \in G_0 \}
  \subset \inv \big( \MP{r}{p, L} \otimes_p A \big).
\]
Apply
Lemma~\ref{L_3330_Exact}
with $B = \MP{r}{p, L} \otimes_p A,$
with $G$ as given,
and with $\ep = 1,$
getting
\[
z \in \big( \MP{r}{p, L} \otimes_p A \big)
   \otimes_{\operatorname{alg}} \big( \MP{r}{p, L} \otimes_p A \big)
\]
such that
$\| z \|_{\pi} < M + 1,$
$\Dt_{\MP{r}{p, L} \otimes_p A} (z) = 1,$
and $(g \otimes 1) z = z (1 \otimes g)$
for all $g \in G.$
The signed permutation matrices span $\MP{r}{p, L}$
by Lemma~2.11 of~\cite{PhLp2a},
so
\[
\big( x \otimes 1_A \otimes 1_{\MP{r}{p, L} \otimes_p A} \big) z
  = z \big( 1_{\MP{r}{p, L} \otimes_p A} \otimes x \otimes 1_A \big)
\]
for all $x \in \MP{r}{p, L}.$
Therefore Lemma~\ref{L_3330_LEst} applies, and we conclude
\[
M + 1
 > \| z \|_{L, \pi}
 \geq R_{p, L}^{1/2}
 = \prod_{l = 1}^n R_{p, S_l}^{1/2}.
\]
So $\prod_{l = 1}^{n} R_{p, S_l} \leq (M + 1)^2,$
as desired.
\end{proof}

\begin{lem}\label{L_3708_Sim}
Let $d \in \N,$
let $S = (I, s, f)$ be a diagonal system of $d$-similarities,
and let $p \in [1, \infty).$
Then, following the notation of Definition~\ref{D_3406_SySiRpn},
there exist
a spatial \rpn{}
$\ta \colon M_d \to L \big( L^p ( X_{S}, \mu_{S} ) \big)$
and
$w \in \inv \big( L \big( L^p ( X_{S}, \mu_{S} ) \big) \big)$
such that
\begin{equation}\label{Eq_3708_wBounds}
\| w \| = R_{p, S},
\,\,\,\,\,\,
\| w - 1 \| = R_{p, S} - 1,
\,\,\,\,\,\,
\| w^{-1} \| = 1,
\,\,\,\,\,\,
\| w^{-1} - 1 \| = 1 - \frac{1}{R_{p, S}},
\end{equation}
and $\ps^{p, S} (x) = w \ta (x) w^{-1}$ for all $x \in M_d.$
\end{lem}

\begin{proof}
Let $i \in I.$
There are $\af_1, \af_2, \ldots, \af_d \in \C \setminus \{ 0 \}$
such that
\[
s (i) = \diag (\af_1, \af_2, \ldots, \af_d) \in M_d^p.
\]
Set $\bt_i = \min \big( | \af_1 |, | \af_2 |, \ldots, | \af_d | \big),$
and define
\[
u_i
 = \diag \big( \sgn (\af_1), \, \sgn (\af_2), \, \ldots, \, \sgn (\af_d)
      \big)
\andeqn
w_i = \bt_i^{- 1}
      \diag \big( | \af_1 |, | \af_2 |, \ldots, | \af_d | \big).
\]
Then $s (i) = \bt_i w_i u_i,$
$u_i$ is isometric,
\begin{equation}\label{Eq_3708_ws}
\| w_i \| = \| s (i) \| \cdot \| s (i)^{-1} \|,
\,\,\,\,\,\,
\| w_i - 1 \| = \| s (i) \| \cdot \| s (i)^{-1} \| - 1,
\end{equation}
\begin{equation}\label{Eq_3901_ws2}
\| w_i^{-1} \| = 1,
\,\,\,\,\,\,
\| w_i^{-1} - 1 \| = 1 - \frac{1}{\| s (i) \| \cdot \| s (i)^{-1} \|},
\end{equation}
$s (i) x s (i)^{-1} = w_i (u_i x u_i^{-1}) w_i^{-1}$
for all $x \in M_d,$
and $x \mapsto u_i x u_i^{-1}$
is a spatial \rpn{} of~$M_d^p.$

Define a \rpn{}
\[
\ta_i \colon
 M_d \to L \big( L^p ( \{ 1, 2, \ldots, d \} \times \{ i \},
         \, \mu_{S}) \big)
\]
by using the identification~(\ref{Eq_3406_Ident}) in
Definition~\ref{D_3406_SySiRpn}
on the \rpn{} $x \mapsto u_i x u_i^{-1}$ for $x \in M_d.$
In the notation of Definition~\ref{D_3406_SySiRpn},
we then have $\ps_i (x) = w_i \ta_i (x) w_i^{-1}$ for all $x \in M_d.$

Let
$\ta \colon M_d \to L \big( L^p ( X_{S}, \mu_{S} ) \big)$
be the $L^p$~direct sum,
as in Definition~\ref{D_3406_pSumDfn},
over $i \in I$ of the \rpn{s}~$\ta_i.$
Then $\ta$ is spatial by Lemma~\ref{L_3807_SpDSum}.
Further take $w$ to be the $L^p$~direct sum of the
operators $w_i$ for $i \in I.$
Since $R_{p, S} = \sup_{i \in I} \| s (i) \| \cdot \| s (i)^{-1} \|,$
it follows from (\ref{Eq_3708_ws}) and~(\ref{Eq_3901_ws2})
that $w$ is in fact in $L \big( L^p ( X_{S}, \mu_{S} ) \big)$
and satisfies~(\ref{Eq_3708_wBounds}).
It is clear that
$\ps^{p, S} (x) = w \ta (x) w^{-1}$ for all $x \in M_d.$
\end{proof}

\begin{lem}\label{L_3805_SimTP}
Adopt the notation of Example~\ref{E_2Y14_LpUHF},
but suppose that, for each $n \in {\mathbb{N}},$
instead of $\rh_n \colon M_{d (n)} \to L (L^p (X_n, \mu_n))$
we have two \rpn{s}
\[
\rh_n^{(1)}, \rh_n^{(2)} \colon M_{d (n)} \to L (L^p (X_n, \mu_n)).
\]
Let $A^{(1)}$ and $A^{(2)}$ be the corresponding
$L^p$~UHF algebras of tensor product type.
Suppose that for every $n \in \N$
there is
$w_n \in \inv \big( L ( L^p ( X_n, \mu_n ) ) \big)$
such that $w_n \rh_n^{(1)} (x) w_n^{-1} = \rh_n^{(2)} (x)$
for all $x \in M_{d (n)}.$
Suppose further that
\[
\sum_{n = 1}^{\infty} \| w_n - 1 \| < \infty
\andeqn
\sum_{n = 1}^{\infty} \big\| w_n^{-1} - 1 \big\| < \infty.
\]
Then there is $y \in \inv \big( \LLp \big)$
such that $y A^{(1)} y^{-1} = A^{(2)}.$
\end{lem}

\begin{proof}
For $j = 1, 2,$
we adapt the notation of Example~\ref{E_2Y14_LpUHF}
by letting $\sm_n^{(j)} \colon M_{r_d (n)} \to A^{(j)}$
be the analog,
derived from the \rpn{s}~$\rh_n^{(j)},$
of the map $\sm_n \colon M_{r_d (n)} \to A$
at the end of Example~\ref{E_2Y14_LpUHF}.
We further set
$A_n^{(j)} = \sm_n^{(j)} ( M_{r_d (n)} ) \subset \LLp,$
so that $A^{(j)} = {\overline{\bigcup_{n = 0}^{\infty} A_n^{(j)} }}.$
It suffices to find $y \in \inv \big( \LLp \big)$
such that $y A_n^{(1)} y^{-1} = A_n^{(2)}$
for all $n \in \Nz.$

We claim that
\[
M_1 = \sup_{n \in \Nz} \prod_{k = 1}^n \| w_n \| < \infty
\andeqn
M_2 = \sup_{n \in \Nz} \prod_{k = 1}^n \big\| w_n^{-1} \big\| < \infty.
\]
The proofs are the same for both, so we prove only the first.
For $n \in \N,$ we observe that
$\| w_n \| \leq 1 + \| w_n - 1 \|,$ 
so
\[
\max ( \| w_n \|, \, 1 ) - 1 \leq \| w_n - 1 \|.
\]
Therefore
\[
\sum_{n = 1}^{\infty} \big[ \max ( \| w_n \|, \, 1 ) - 1 \big]
  \leq \sum_{n = 1}^{\infty} \| w_n - 1 \|
  < \infty.
\]
Using Lemma~\ref{L_2Y26_SumProd}
at the second step,
we therefore get
\[
\sup_{n \in \Nz} \prod_{k = 1}^n \| w_n \|
 \leq \prod_{n = 1}^{\infty} \max ( \| w_n \|, \, 1 )
 < \infty.
\]
The claim is proved.

For $n \in \Nz,$
set
\[
y_n = w_1 \otimes w_2 \otimes \cdots \otimes w_n
       \otimes 1_{{\mathbb{N}}_{> n}}
    \in \LLp.
\]
Then
\begin{align*}
\| y_n - y_{n - 1} \|
& = \big\| w_1 \otimes w_2 \otimes \cdots \otimes w_{n - 1}
                \otimes (w_n - 1)
                \otimes 1_{{\mathbb{N}}_{> n}} \big\|
          \\
& = \| w_n - 1 \| \prod_{k = 1}^{n - 1} \| w_k \|
  \leq M_1 \| w_n - 1 \|.
\end{align*}
Therefore
\[
\sum_{n = 1}^{\infty} \| y_n - y_{n - 1} \|
  \leq M_1 \sum_{n = 1}^{\infty} \| w_n - 1 \|
   < \infty.
\]
It follows that
$y = \limi{n} y_n \in L \big( L^p ( X, \mu ) \big)$
exists.
We also get
\[
\big\| y_n^{-1} - y_{n - 1}^{-1} \big\|
 = \big\| w_n^{-1} - 1 \big\|
       \prod_{k = 1}^{n - 1} \big\| w_k^{-1} \big\|
 \leq M_2 \big\| w_n^{-1} - 1 \big\|,
\]
so
$\sum_{n = 1}^{\infty} \big\| y_n^{-1} - y_{n - 1}^{-1} \big\| < \infty,$
whence
$z =
 \limi{n} y_n^{-1} \in L \big( L^p ( X_{I, s, f}, \mu_{I, s, f} ) \big)$
exists.
Clearly $y z = z y = 1,$
so $y$ is invertible with inverse~$z.$
Since $y_l A_n^{(1)} y_l^{-1} = A_n^{(2)}$
for all $l \in \N$ with $l \geq n,$
it follows that $y A_n^{(1)} y^{-1} = A_n^{(2)}.$
\end{proof}

\begin{thm}\label{T_3708_Third}
Let $p \in [1, \infty),$
let $d = (d (1), \, d (2), \, \ldots )$
be a sequence in $\{ 2, 3, \ldots \},$
and let
\[
I = (I_1, I_2, \ldots ),
\,\,\,\,\,\,
s = (s_1, s_2, \ldots ),
\andeqn
f = (f_1, f_2, \ldots )
\]
be sequences such that $S_n = (I_n, s_n, f_n)$
is a diagonal system of $d (n)$-similarities for all $n \in \N.$
Set
$A = D_{p, d, I, s, f}
   \subset L \big( L^p ( X_{I, s, f}, \mu_{I, s, f} ) \big)$
as in Definition~\ref{D_3406_SfUHF}
with the given choices of $I,$ $s,$ and~$f.$
For $n \in \Nz$,
following the notation of Example~\ref{E_2Y14_LpUHF},
let $\sm_n \colon M_{r_d (n)} \to A$
be $\sm_n = \rh_{{\mathbb{N}}, {\mathbb{N}}_{\leq n} },$
and following the notation of Definition~\ref{D_3406_SfUHF},
set $\rh_n = \ps^{p, S_n}.$
(The notation $\rh_n$ is used in Example~\ref{E_2Y14_LpUHF}.)
Let $B$ be the spatial $L^p$~UHF algebra whose supernatural number
is the same as that of~$A.$
Then \tfae:
\begin{enumerate}
\item\label{T_3708_Third_Iso}
$A \cong B.$
\item\label{T_3708_Third_IsoSp}
There exists an isomorphism $\ph \colon A \to B$
such that the algebraic tensor product of two copies of~$\ph$
extends to an isomorphism
$\ph \otimes_p \ph \colon A \otimes_p A \to B \otimes_p B.$
\item\label{T_3708_Third_SymAmen}
$A$ is symmetrically amenable.
\item\label{T_3708_Third_Amen}
$A$ is amenable.
\item\label{T_3708_Third_ClRangeP}
Whenever $\YCN$ is a \sfm,
$C \subset L (L^p (Y, \nu))$ is a closed unital subalgebra,
and $\ph \colon A \to C$
is a unital \ct{} \hm{}
such that
$\ph \otimes_p \id_A \colon A \otimes_p A \to C \otimes_p A$
is bounded,
then $\ph$ is bounded below.
\item\label{T_3708_Third_7a}
$A \otimes_p A$ has approximately inner $L^p$-tensor half flip.
\item\label{T_3708_Third_AIF}
$A$ has approximately inner $L^p$-tensor flip.
\item\label{T_3708_Third_Half}
$A$ has approximately inner $L^p$-tensor half flip.
\item\label{T_3708_Third_SmBddOne}
There is a uniform bound on the norms of the maps
$\sm_n \colon \MP{r_d (n)}{p} \to A.$
\item\label{T_3708_Third_SmBdd}
There is a uniform bound on the norms of the maps
\[
\sm_n \otimes_p \sm_n \colon \MP{r_d (n)^2}{p} \to A \otimes_p A.
\]
\item\label{T_3708_Third_RhBdd}
$\sum_{n = 1}^{\infty} ( \| \rh_n \| - 1 ) < \infty.$
\item\label{T_3708_Third_10A}
With $R_{p, S_n}$ as in Definition~\ref{D_3406_SySiRpn}
for $n \in \N,$
we have
$\sum_{n = 1}^{\infty} ( R_{p, S_n} - 1 ) < \infty.$
\item\label{T_3708_Third_SumFin}
The \hm{s}
$\rh_n \otimes_p \rh_n \colon \MP{d (n)^2}{p} \to A \otimes_p A$
satisfy
\[
\sum_{n = 1}^{\infty} ( \| \rh_n \otimes_p \rh_n \| - 1 ) < \infty.
\]
\item\label{T_3708_Third_Sim}
There exists
a spatial $L^p$~UHF algebra
$C \subset L \big( L^p ( X_{I, s, f}, \mu_{I, s, f} ) \big)$
which is isometrically isomorphic to~$B,$
and
$v \in
 \inv \big( L \big( L^p ( X_{I, s, f}, \mu_{I, s, f} ) \big) \big),$
such that $v A v^{-1} = C.$
\end{enumerate}
\end{thm}

Some of the conditions in Theorem~\ref{T_3708_Third}
do not involve the triple $(I, s, f).$
It follows that these
all hold, or all fail to hold,
for any choice of $(I, s, f)$
giving an isomorphic algebra.

The algebra
$B \otimes_p B$ in~(\ref{T_3708_Third_IsoSp})
is a spatial $L^p$~UHF algebra.
In particular, using Theorem 3.10(5) of~\cite{PhLp2a},
one can check that, up to isometric isomorphism, it does not depend
on how $B$ is represented on an $L^p$~space.

\begin{proof}[Proof of Theorem~\ref{T_3708_Third}]
We simplify the notation,
following Example~\ref{E_2Y14_LpUHF},
by setting $X_n = X_{S_n}$ and $\mu_n = \mu_{S_n}$
for all $n \in \N.$
We freely use other notation from Example~\ref{E_2Y14_LpUHF}.

The equivalence of
(\ref{T_3708_Third_IsoSp}),
(\ref{T_3708_Third_7a}),
(\ref{T_3708_Third_AIF}),
(\ref{T_3708_Third_Half}),
(\ref{T_3708_Third_SmBdd}),
and (\ref{T_3708_Third_SumFin})
follows from Theorem~\ref{T_2Z29_First}.

That (\ref{T_3708_Third_IsoSp}) implies~(\ref{T_3708_Third_Iso})
is trivial.
Assume~(\ref{T_3708_Third_Iso});
we prove~(\ref{T_3708_Third_SymAmen}).
Since symmetric amenability only depends on the isomorphism class
of a Banach algebra,
it is enough to prove symmetric amenability when $A = B.$ 
The maps $\sm_n \colon \MP{r_d (n)}{p} \to B$ are then isometric,
so that the maps
$\sm_n {\widehat{\otimes}} \sm_n \colon
   \MP{r_d (n)}{p} {\widehat{\otimes}} \MP{r_d (n)}{p} \to
    B {\widehat{\otimes}} B$
are contractive.
Now apply Proposition~\ref{P_2Y26_GetAIF}(\ref{P_2Y26_GetAIF_Amen})
and Lemma~\ref{L_2Y25_DiagMn}(\ref{L_2Y25_DiagMn_Norm}).
The implication from (\ref{T_3708_Third_SymAmen})
to~(\ref{T_3708_Third_Amen})
is trivial.

The implication from~(\ref{T_3708_Third_Amen})
to~(\ref{T_3708_Third_10A})
is immediate from Theorem~\ref{T_3330_ConseqAmen}
and Lemma~\ref{L_2Y26_SumProd}.
Now assume~(\ref{T_3708_Third_10A});
we prove~(\ref{T_3708_Third_RhBdd}).
Let $T_n$ be the basic system of $d (n)$-similarities
(called $S_0$ in Definition~\ref{D_3406_SysSim}).
When one identifies $M_{d (n)}$ with $\MP{d (n)}{p},$
the map $\rh_n = \ps^{p, S_n} \colon M_{d (n)} \to \MP{d (n)}{p, S_n}$
becomes the map
$\kp_{\C}^{p, S_n, T_n} \colon \MP{d (n)}{p} \to \MP{d (n)}{p, S_n}$
of Notation~\ref{N_3406_Kp}.
Then Lemma \ref{L_3405_Basics}(\ref{L_3405_Basics_FrmBase})
implies that $\| \rh_n \| = R_{p, S_n}$ for $n \in \N.$
The desired implication is now clear.

Assume~(\ref{T_3708_Third_RhBdd});
we prove~(\ref{T_3708_Third_Sim}).
Apply Lemma~\ref{L_3708_Sim}
to $S_n$ for all $n \in \N,$
obtaining spatial \rpn{s}
$\ta_n \colon M_{d (n)} \to
   L ( L^p ( X_{n}, \mu_{n} ) )$
and invertible elements
$w_n \in L ( L^p ( X_{n}, \mu_{n} ) ),$
satisfying the estimates given there.
We have $R_{p, S_n} = \| \rh_n \|$
as in the proof of the implication from (\ref{T_3708_Third_10A})
to~(\ref{T_3708_Third_RhBdd}).
Therefore the estimates become
\[
\| w_n \| = \| \rh_n \|,
\,\,\,\,\,\,
\| w_n - 1 \| = \| \rh_n \| - 1,
\,\,\,\,\,\,
\| w_n^{-1} \| = 1,
\andeqn
\| w_n^{-1} - 1 \| = 1 - \frac{1}{\| \rh_n \|}.
\]
Applying the construction of Example~\ref{E_2Y14_LpUHF}
to $d$ and $(\ta_1, \ta_2, \ldots),$
we obtain a spatial $L^p$~UHF algebra
$C = A (d, \ta) \subset  L \big( L^p ( X_{I, s, f}, \mu_{I, s, f} ) \big).$
Obviously its supernatural number is the same as that of~$A.$
Lemma~\ref{L_3805_SimTP} provides an invertible element
$y \in L \big( L^p ( X_{I, s, f}, \mu_{I, s, f} ) \big)$
such that $y A y^{-1} = C.$
Theorem~\ref{T_3729_EU} implies that
$C$ is isometrically isomorphic to~$B.$

We now show that~(\ref{T_3708_Third_Sim})
implies~(\ref{T_3708_Third_IsoSp}).
By hypothesis,
the formula $\ph (a) = v a v^{-1}$ for $a \in A$
defines an isomorphism $\ph \colon A \to B.$
Then $a \mapsto (v \otimes v) a (v \otimes v)^{-1}$
is an isomorphism from $A \otimes_p A$ to $B \otimes_p B$
which sends $a_1 \otimes a_2$ to $\ph (a_1) \otimes \ph (a_2)$
for all $a_1, a_2 \in A.$

It remains only to prove that
(\ref{T_3708_Third_ClRangeP})
and~(\ref{T_3708_Third_SmBddOne})
are equivalent to the other conditions.

We prove that (\ref{T_3708_Third_SmBddOne})
is equivalent to~(\ref{T_3708_Third_SmBdd}).
It suffices to prove that $\| \sm_n \otimes \sm_n \| = \| \sm_n \|^2$
for all $n \in \N.$
Repeated application of Lemma~\ref{L_3406_TPSys}
shows that
there is a diagonal system $T_0 = (J_0, t_0, g_0)$
of $r_d (n)$-similarities
such that $\sm_n = \ps_{\C}^{p, T_0}.$
The same lemma further shows
that setting $J = J_0 \times J_0$
and setting
$t (j, k) = t_0 (j) \otimes t_0 (k)$
and $g (j, k) = g_0 (j) g_0 (k)$
gives a diagonal system $T = (J, t, g)$
of $r_d (n)^2$-similarities
such that $\sm_n \otimes \sm_n = \ps_{\C}^{p, T}.$
The relation $\| \sm_n \otimes \sm_n \| = \| \sm_n \|^2$
now follows from
Lemma \ref{L_3405_Basics}(\ref{L_3405_Basics_NForm}),
Lemma~\ref{L_3317_NConj},
and the relation $\| v \otimes w \| = \| v \| \cdot \| w \|$
for $v, w \in \MP{r_d (n)}{p}.$

That (\ref{T_3708_Third_Half}) implies~(\ref{T_3708_Third_ClRangeP})
follows from Theorem~\ref{T_2Y25_AIPHF}.

We show that~(\ref{T_3708_Third_ClRangeP})
implies~(\ref{T_3708_Third_Iso}).
Let $\kp \colon A \to B$ be as in Corollary~\ref{L_2Z17New}.
By Corollary \ref{L_2Z17New}(\ref{2Z17New_TP}),
there is a \ct{} \hm{}
$\gm \colon A \otimes_p A \to B \otimes_p A$
such that $\gm (a_1 \otimes a_2) = \kp (a_1) \otimes a_2$
for all $a_1, a_2 \in A.$
Condition~(\ref{T_3708_Third_ClRangeP})
implies that $\kp$ has closed range.
Since $\kp$ has dense range by Corollary \ref{L_2Z17New}(\ref{2Z17New_DR}),
we conclude that $\kp$ is an isomorphism.
\end{proof}

For $p = 2,$
we can do a little better:
we don't need to require the systems of similarities to be diagonal.

\begin{thm}\label{T_3711_CStAmen}
Let $d = (d (1), \, d (2), \, \ldots )$
be a sequence in $\{ 2, 3, \ldots \},$
and let
\[
I = (I_1, I_2, \ldots ),
\,\,\,\,\,\,
s = (s_1, s_2, \ldots ),
\andeqn
f = (f_1, f_2, \ldots )
\]
be sequences such that $S_n = (I_n, s_n, f_n)$
is a system of $d (n)$-similarities for all $n \in \N.$
Set
$A = D_{2, d, I, s, f}
   \subset L \big( L^2 ( X_{I, s, f}, \mu_{I, s, f} ) \big)$
as in Definition~\ref{D_3406_KUHF}
with the given choices of $I,$ $s,$ and~$f.$
If $A$ is amenable,
then there exists
$v \in \inv \big( L \big( L^2 ( X_{I, s, f}, \mu_{I, s, f} ) \big) \big)$
such that $v A v^{-1}$ is a \ca.
\end{thm}

Problem~30 in the ``Open Problems'' chapter of~\cite{Rnd}
asks whether an amenable closed subalgebra of the bounded operators
on a Hilbert space is similar to a \ca.
This problem has been open for some time,
and little seems to be known.
(See the discussion in the introduction.)
Theorem~\ref{T_3711_CStAmen}
shows that the answer is yes for a class of algebras which,
as far as we know,
is quite different from any other class considered in this context.

\begin{proof}[Proof of Theorem~\ref{T_3711_CStAmen}]
Let $R_{p, S}$
be as in Definition~\ref{D_3406_SySiRpn}.
By Proposition~\ref{P_3906_NewIso}(\ref{P_3906_NewIso_Scalar}),
for all $n \in \N$ and $i \in I_n,$
we can replace $s_n (i)$ by $\| s_n (i)^{-1} \| s_n (i).$
Therefore we may assume that $\| s_n (i)^{-1} \| = 1.$
(This change preserves the requirement $1 \in \ran (s_n).$)

We claim that
$\sum_{n = 1}^{\infty} (R_{2, S_n} - 1) < \infty.$
This is the main part of the proof.

To prove the claim,
for all $n \in \N$
choose $i_n, j_n \in I_n$
such that $s_n (i_n) = 1$ and
$\| s_n (j_n) \| \cdot \| s_n (j_n)^{-1} \| > R_{2, S_n} - 2^{-n}.$
Set $J_n = \{ i_n, j_n \}.$
We define systems $T_n^{(0)} = \big( J_n, t_n^{(0)}, g_n \big)$
of $d (n)$-similarities
as follows.
If $i_n = j_n,$
so that $J_n$ has only one element,
we set $t_n^{(0)} (i_n) = 1$ and $g_n (i_n) = 1.$
If $i_n \neq j_n,$
we set
\[
t_n^{(0)} = s_n |_{J_n},
\,\,\,\,\,\,
g_n (i_n) = \frac{f_n (i_n)}{f_n (i_n) + f_n (j_n)},
\andeqn
g_n (j_n) = \frac{f_n (j_n)}{f_n (i_n) + f_n (j_n)}.
\]
By construction,
we have $R_{2, T_n^{(0)}} > R_{2, S_n} - 2^{-n}.$
Set
\[
J = (J_1, J_2, \ldots),
\,\,\,\,\,\,
t^{(0)} = \big( t_1^{(0)}, t_2^{(0)}, \ldots \big),
\andeqn
g = (g_1, g_2, \ldots).
\]

Define a subsystem of $(d, \rh)$
as in Definition~\ref{D_3907_SubSys}
by taking
\[
Z_n = \{ 1, 2, \ldots, d (n) \} \times J_n = X_{d (n), T_n^{(0)}}
\]
for $n \in \N.$
Let $\ld_n$ and $\gm_n$ be as in Definition~\ref{D_3907_SubSys}.
Then $\ld_n = \mu_{d (n), T_n^{(0)}}.$
Let $B^{(0)}$ be
the $L^p$~UHF algebra of the subsystem $(Z_n)_{n \in \N}$
as in Definition~\ref{D_3907_SubSys},
which is equal to $D_{2, d, J, t^{(0)}, g}.$
Let $\kp \colon A \to B^{(0)}$ be the \hm{}
of Lemma~\ref{L_3907_MapToSub}.
Since $A$ is amenable and $\kp$ has dense range,
it follows from Proposition 2.3.1 of~\cite{Rnd} that $B^{(0)}$ is amenable.

For $n \in \N,$
define a function $z_n \colon J_n \to M_{d (n)}$
by $z_n (i_n) = 1$ and,
if $j_n \neq i_n,$
using polar decomposition to choose
a unitary $z_n (j_n)$ such that
\[
z_n (j_n) s_n (j_n) = [s_n (j_n)^* s_n (j_n)]^{1/2}.
\]
Define $t_n^{(1)} \colon J_n \to \inv (M_{d (n)})$
by $t_n^{(1)} (i) = z_n (i) s_n (i)$ for $i \in J_n.$
Then set $T_n^{(1)} = \big( J_n, t_n^{(1)}, g_n \big),$
which is a system of $d (n)$-similarities.
Define a sequence $t^{(1)}$ by
$t^{(1)} = \big( t_1^{(1)}, t_2^{(1)}, \ldots \big),$
and set $B^{(1)} = D_{2, d, J, t^{(1)}, g}.$
Proposition~\ref{P_3906_NewIso}(\ref{P_3906_NewIso_Isometry})
implies that $B^{(1)}$ is isometrically isomorphic to $B^{(0)},$
so that $B^{(1)}$ is amenable,
and that $R_{2, T_n^{(1)}} = R_{2, T_n^{(0)}}$ for all $n \in \N.$

For $n \in \N,$
further choose a unitary $v_n \in M_{d (n)}$
such that $v_n t_n^{(1)} (j_n) v_n^*$ is diagonal.
Define $t_n \colon J_n \to \inv (M_{d (n)})$
by $t_n (i) = v_n (i) t_n^{(1)} (i) v_n (i)^*$ for $i \in I_n.$
Then set $T_n = \big( J_n, t_n, g_n \big),$
which is a system of diagonal $d (n)$-similarities.
Define a sequence $t$ by $t = \big( t_1, t_2, \ldots \big),$
and set $B = D_{2, d, J, t, g}.$
Proposition~\ref{P_3906_NewIso}(\ref{P_3906_NewIso_Conj})
implies that $B$ is isometrically isomorphic to $B^{(1)},$
so that $B$ is amenable,
and that $R_{2, T_n} = R_{2, T_n^{(0)}}$ for all $n \in \N.$
From Theorem~\ref{T_3708_Third} we get
$\sum_{n = 1}^{\infty} (R_{2, T_n} - 1) < \infty.$
Therefore
\[
\sum_{n = 1}^{\infty} (R_{2, S_n} - 1)
 < \sum_{n = 1}^{\infty} \left( R_{2, T_n} + \frac{1}{2^n} - 1 \right)
 < \infty.
\]
The claim is proved.

For $n \in \N$ and $i \in I_n,$
use the polar decomposition of $s_n (i)^*$
to find a selfadjoint element $c_n (i) \in M_{d (n)}$
and a unitary $u_n (i) \in M_{d (n)}$
such that $s_n (i) = c_n (i) u_n (i).$
Then $(S_n, u_n, f_n)$ is a system of $d (n)$-similarities.
We have
\[
\| c_n (i) \| = \| s_n (i) \|
\andeqn
\| c_n (i)^{-1} \| = \| s_n (i)^{-1} \| = 1,
\]
so $1 \leq c_n (i) \leq \| s_n (i) \| \cdot 1.$
Therefore
$\| c_n (i) - 1 \| \leq \| s_n (i) \| - 1 \leq R_{2, S_n} - 1.$
Also
\[
\| c_n (i)^{-1} - 1 \|
 \leq \| c_n (i)^{-1} \| \cdot \| 1 - c_n (i) \|
 = \| 1 - c_n (i) \|
 \leq R_{2, S_n} - 1.
\]
Recalling that
$X_{T_n} = \{ 1, 2, \ldots, d (n) \} \times I_n,$
and applying Lemma~\ref{L_3903_ChM}
and Remark~\ref{R_3913_DSElt},
let $c_n \in \inv ( L ( L^2 (X_{T_n}, \mu_{T_n} ))$
be the direct sum over $i \in I_n$
of the elements $c_n (i).$
Thus $\| c_n - 1 \| \leq R_{2, S_n} - 1$
and $\| c_n^{-1} - 1 \| \leq R_{2, S_n} - 1.$

Set $u = (u_1, u_2, \ldots).$
Since $\sum_{n = 1}^{\infty} (R_{2, S_n} - 1) < \infty,$
we have $\sum_{n = 1}^{\infty} \| c_n - 1 \| < \infty$
and $\sum_{n = 1}^{\infty} \| c_n^{-1} - 1 \| < \infty.$
So Lemma~\ref{L_3805_SimTP}
provides an invertible element
$v \in L (L^2 (X_{I, s, f}, \mu_{I, s, f})$
such that $v A v^{-1} = D_{2, d, I, u, f}.$
Since $u_{n} (i)$ is unitary for all $n \in \N$ and $i \in I_n,$
it is immediate that $D_{2, d, I, u, f}$
is a C*-subalgebra of $L (L^2 (X_{I, s, f}, \mu_{I, s, f}).$
\end{proof}

\begin{qst}\label{Q_3121_AmenSp}
Let $p \in [1, \infty).$
Let $A$ be an amenable $L^p$~UHF algebra,
but not of the type to which Theorem~\ref{T_3708_Third} applies.
Does it follow that $A$ is isomorphic to a spatial $L^p$~UHF algebra?
Does it follow that $A$ is similar to a spatial $L^p$~UHF algebra?
What if we assume, say, that $A$ is an $L^p$~UHF algebra
of tensor product type?
What if we assume that $A$ is in fact symmetrically amenable?
\end{qst}

Amenability of spatial $L^p$~UHF algebras
is used in the proof of Corollary~5.18 of~\cite{PhLp2a}
to prove that spatial representations
of the Leavitt algebra~$L_d$
generate amenable Banach algebras.
(That is, $\OP{d}{p}$ is amenable for $p \in [1, \infty)$
and $d \in \{ 2, 3, \ldots \}.$)
It is clear from Theorem~\ref{T_3708_Third}
(and will be much more obvious from Theorem~\ref{T_3326_Uctbl} below)
that there are $L^p$~UHF algebras of tensor product type
which are not amenable.
However,
we do not know how to use known examples to construct
nonamenable versions of~$\OP{d}{p}.$

\begin{qst}\label{Q_3121_CtNAmen}
Let $p \in [1, \infty)$
and let $d \in \{ 2, 3, \ldots \}.$
Does there exist a \sfm{} $\XBM$
and a \rpn{} $\rh$ of the Leavitt algebra~$L_d$
on $L^p (X, \mu)$
such that $\| \rh (s_j) \| = 1$ and $\| \rh (t_j) \| = 1$
for $j = 1, 2, \ldots, d,$
but such that ${\overline{\rh (L_d)}}$ is not amenable?
\end{qst}

The situation is a bit different from what we have considered here.
If $\| \rh (s_j) \| = 1$ and $\| \rh (t_j) \| = 1$
for $j = 1, 2, \ldots, d,$
then the standard matrix units
in the analog for ${\overline{\rh (L_d)}}$
of the UHF core of ${\mathcal{O}}_d$
all have norm~$1.$
We do not know whether there are nonamenable
$L^p$~UHF algebras
in which all the standard matrix units have norm~$1.$

\section{Many nonisomorphic $L^p$~UHF algebras}\label{Sec_ManyNI}

\indent
In this section,
for fixed $p \in (1, \infty)$
and a fixed supernatural number~$N,$
we prove that there are uncountably many
mutually nonisomorphic $L^p$~UHF algebras
of infinite tensor product type
with the same supernatural number~$N.$
We rule out not just isometric isomorphism but
isomorphisms which convert the norm to an equivalent norm.
Our algebras are all obtained using diagonal similarities,
so are covered by Theorem~\ref{T_3708_Third}.
In particular, all except possibly one of them is not amenable.

We do not use any explicit invariant.
Rather, we give lower bounds on the norms
of nonzero \hm{s} from matrix algebras into the
algebras we consider,
which increase with the size of the matrix algebra.
We use these to prove the nonexistence of \ct{} \hm{s}
between certain pairs of our algebras.

Our construction uses a class of diagonal similarities
which is easy to deal with,
and which we now introduce.

\begin{ntn}\label{N_3126_Qld}
Let $p \in [1, \infty),$
let $d \in \N,$
and let $\gm \in [1, \infty).$
We define $K_{d, \gm} \subset \inv (M_d)$ to be the
compact set consisting of
all diagonal matrices in $M_d$ whose diagonal entries are
all in $[1, \gm].$
When we need a system $S_{d, \gm} = (I_{d, \gm}, s_{d, \gm}, f_{d, \gm})$
of $d$-similarities
such that ${\overline{\ran (s_{d, \gm}) }} = K_{d, \gm},$
we take $I_{d, \gm}$ to consist
of all diagonal matrices in $M_d$ whose diagonal entries are
all in $[1, \gm] \cap \Q$
and $s_{d, \gm}$ to be the identity map.
For each $d$ and~$\gm,$
we choose once and for all an arbitrary function
$f_{d, \gm} \colon I_{d, \gm} \to (0, 1]$
satisfying $\sum_{i \in I_{d, \gm}} f (i) = 1.$

Let $\YCN$ be a \sfm,
and let $A \subset L (L^p (Y, \nu))$ be a closed subalgebra.
In Definition~\ref{D_3406_MP},
we make the following abbreviations.
\begin{enumerate}
\item\label{N_3126_Qld_MP}
$\MP{d}{p, \gm} = \MP{d}{p, K_{d, \gm}}.$
\item\label{N_3126_Qld_Norm}
$\| \cdot \|_{p, \gm} = \| \cdot \|_{p, K_{d, \gm}}.$
\item\label{N_3126_Qld_kp}
$\kp_{d}^{p, \gm_2, \gm_1} = \kp^{p, K_{d, \gm_2}, K_{d, \gm_1}}$
and
$\kp_{d, A}^{p, \gm_2, \gm_1} = \kp_A^{p, K_{d, \gm_2}, K_{d, \gm_1}}$
for $\gm_1, \gm_2 \in [1, \infty).$
\end{enumerate}
In particular,
$\MP{d}{p, 1} = \MP{d}{p}.$
\end{ntn}

We adapt Lemma~\ref{L_3405_Basics}
and Lemma~\ref{L_3406_NormDg} to our current situation.

\begin{lem}\label{L_3126_QBasic}
Adopt Notation~\ref{N_3126_Qld}.
Also let $\YCN$ be a \sfm,
and let $A \subset L (L^p (Y, \nu))$ be a closed unital subalgebra.
Then:
\begin{enumerate}
\item\label{L_3126_QBasic_Norm}
$\| \kp_{d, A}^{p, \gm, 1} \| = \gm.$
\item\label{L_3126_QBasic_Inv}
For $\gm_2 \geq \gm_1 \geq 1,$
we have
$\big\| \kp_{d, A}^{p, \gm_1, \gm_2} \big\| = 1.$
\item\label{L_3126_QBasic_Nejk}
For $j, k = 1, 2, \ldots, d$ and $a \in A,$
in $\MP{d}{p, \gm} \otimes_p A$ we have
\[
\| e_{j, k} \otimes a \|_{p, \gm}
 = \begin{cases}
   \| a \|     & j = k
        \\
   \gm \| a \| & j \neq k.
\end{cases}
\]
\item\label{L_3126_QBasic_Sim}
For any $x \in M_d \otimes A,$
we have
\[
\| x \|_{p, \gm}
 = \sup_{v \in K_{d, \gm}} \| (v \otimes 1) x (v^{-1} \otimes 1) \|_p.
\]
\item\label{L_3126_QBasic_Diag}
Let $a_1, a_2, \ldots, a_d \in A.$
Then in $\MP{d}{p, \gm} \otimes_p A,$
we have
\[
\big\| e_{1, 1} \otimes a_1 + e_{2, 2} \otimes a_2 + \cdots
    + e_{d, d} \otimes a_d \big\|_{p, \gm}
 = \max \big( \| a_1 \|, \, \| a_2 \|, \, \ldots, \, \| a_d \| \big).
\]
\end{enumerate}
\end{lem}

The estimates in (\ref{L_3126_QBasic_Norm})
and~(\ref{L_3126_QBasic_Inv})
mean that $\| x \| \leq \| \kp_{d}^{p, \gm, 1} (x) \| \leq \gm \| x \|$
for all $x \in \MP{d}{p}.$

The reason for our definition of $K_{d, \gm}$
is to ensure that
$\bt \geq \gm$ implies $K_{d, \gm} \subset K_{d, \bt}.$
If we used diagonal matrices with diagonal entries
in $\{ 1, \gm \},$
our (easy) proof that $\| x \|_{p, \bt} \geq \| x \|_{p, \gm}$
would break down,
and we do not know whether
this inequality would still be true.

\begin{proof}[Proof of Lemma~\ref{L_3126_QBasic}]
For $j, k \in \{ 1, 2, \ldots, d \},$
let $r_{j, k}$ be as in
Lemma~\ref{L_3406_NormDg}
with $K = K_{d, \gm}.$
Then $r_{j, j} = 1$ for $j = 1,2, \ldots, d$
by
Lemma~\ref{L_3406_NormDg}(\ref{L_3406_NormDg_Omjj}),
and it is easy to check that
$r_{j, k} = \gm$ for $j, k \in \{ 1, 2, \ldots, d \}$ with $j \neq k.$

Part~(\ref{L_3126_QBasic_Norm})
is now immediate from
Lemma~\ref{L_3406_NormDg}(\ref{L_3406_NormDg_Nrho}),
and part~(\ref{L_3126_QBasic_Inv})
is immediate from Lemma \ref{L_3405_Basics}(\ref{L_3405_Basics_Sub}).
Using the computation of $r_{j, k}$ above,
part~(\ref{L_3126_QBasic_Nejk})
follows from Lemma~\ref{L_3406_NormDg}(\ref{L_3406_NormDg_Normejk})
and~(\ref{L_3406_NormDg_Diag}).

Part~(\ref{L_3126_QBasic_Sim})
is Lemma \ref{L_3405_Basics}(\ref{L_3405_Basics_NForm}),
and
part~(\ref{L_3126_QBasic_Diag})
follows from Lemma \ref{L_3406_NormDg}(\ref{L_3406_NormDg_Sum}).
\end{proof}

\begin{ntn}\label{N_3325_InfTPN}
Let $d = (d (1), \, d (2), \, \ldots)$
be a sequence in $\{ 2, 3, \ldots \},$
and let $\gm = (\gm (1), \, \gm (2), \, \ldots)$
be a sequence in $[1, \infty).$
Using the choices of Notation~\ref{N_3126_Qld},
define
\[
I_{d, \gm} = (I_{d (1), \gm (1)}, I_{d (2), \gm (2)}, \ldots ),
\,\,\,\,\,\,
s_{d, \gm} = (s_{d (1), \gm (1)}, s_{d (2), \gm (2)}, \ldots ),
\]
and
\[
f_{d, \gm} = (f_{d (1), \gm (1)}, f_{d (2), \gm (2)}, \ldots ).
\]
Following Definition~\ref{D_3406_SfUHF},
we now set $B_d^{p, \gm} = D_{p, d, I_{d, \gm}, s_{d, \gm}, f_{d, \gm}}.$
When $d,$ $p,$ and $\gm$ are understood,
we just call this algebra~$B.$
For $n \in \Nz,$
we then further set
$B_n = D_{p, d, I_{d, \gm}, s_{d, \gm}, f_{d, \gm}}^{(0, n)}.$
Thus
\[
B_n = \MP{d (1)}{p, \gm (1)} \otimes_p \MP{d (2)}{p, \gm (2)} \otimes_p
       \cdots \otimes_p \MP{d (n)}{p, \gm (n)}.
\]
We can isometrically identify $B_0, B_1, \ldots$
with subalgebras of~$B$
in such a way that $B_0 \subset B_1 \subset \cdots$
and $B = {\overline{\bigcup_{n = 0}^{\infty} B_n}}.$
\end{ntn}

Thus,
$B_d^{p, \gm}$ is the $L^p$~spatial infinite tensor product
of the algebras $\MP{d (n)}{p, \gm (n)}.$

\begin{lem}\label{L_3326_DomRel}
Let $p \in [1, \infty),$
let $d = (d (1), \, d (2), \, \ldots)$
be a sequence in $\{ 2, 3, \ldots \},$
and let $\bt = (\bt (1), \, \bt (2), \, \ldots)$
and $\gm = (\gm (1), \, \gm (2), \, \ldots)$
be sequences
in $[1, \infty).$
Let $l \in \N,$
and suppose that
$\gm (j) \geq \bt (j)$ for $j = l + 1, \, l + 2, \, \ldots.$
Then the obvious map on the algebraic direct limits
extends to a \ct{} unital \hm{}
from $B_d^{p, \gm}$ to $B_d^{p, \bt}$
which has dense range.
\end{lem}

\begin{proof}
We follow Notation~\ref{N_3325_InfTPN}
and the notation of Definition~\ref{D_3406_SfUHF}.
We apply part~(\ref{3907_MapToSub_TP})
of the conclusion of Lemma~\ref{L_3907_MapToSub},
with $D = D_{p, d, I_{d, \gm}, s_{d, \gm}, f_{d, \gm}}^{(0, l)},$
to get a contractive unital \hm{}
\[
\ta \colon B_d^{p, \gm} \to
 D_{p, d, I_{d, \gm}, s_{d, \gm}, f_{d, \gm}}^{(0, l)}
   \otimes_p D_{p, d, I_{d, \bt}, s_{d, \bt}, f_{d, \bt}}^{(l, \infty)}
\]
which has dense range.
As algebras, we have
\[
D_{p, d, I_{d, \gm}, s_{d, \gm}, f_{d, \gm}}^{(0, l)}
 = D_{p, d, I_{d, \bt}, s_{d, \bt}, f_{d, \bt}}^{(0, l)}
 = M_{r_d (l)}.
\]
Since $M_{r_d (l)}$ is finite dimensional, the identity map
on
$M_{r_d (l)} \otimes_{\text{alg}}
  D_{p, d, I_{d, \bt}, s_{d, \bt}, f_{d, \bt}}^{(l, \infty)}$
extends to a \ct{} bijection
\[
D_{p, d, I_{d, \gm}, s_{d, \gm}, f_{d, \gm}}^{(0, l)}
   \otimes_p D_{p, d, I_{d, \bt}, s_{d, \bt}, f_{d, \bt}}^{(l, \infty)}
 \to D_{p, d, I_{d, \bt}, s_{d, \bt}, f_{d, \bt}}^{(0, l)}
   \otimes_p D_{p, d, I_{d, \bt}, s_{d, \bt}, f_{d, \bt}}^{(l, \infty)}
 = B_d^{p, \bt}.
\]
Compose this map with $\ta$ to complete the proof of the lemma.
\end{proof}

We now give some results on perturbations
of \hm{s} from direct sums of matrix algebras.
We adopt some notation from the beginning of Section~1 of~\cite{Jh0}.

\begin{dfn}\label{D_3326_JhNtn}
Let $A$ and $D$ be Banach algebras,
and let $T \colon A \to D$ be a bounded linear map.
We define a bounded bilinear map
$T^{\vee} \colon A \times A \to D$
by
\[
T^{\vee} (x, y) = T (x y) - T (x) T (y)
\]
for $x, y \in A.$
We further define $d (T)$ to be the distance from $T$ to the set
of bounded \hm{s} from $A$ to~$D,$
that is,
\[
d (T)
 = \inf \big( \big\{ \| T - \ph \| \colon
    {\mbox{$\ph \colon A \to D$ is a continuous homomorphism}} \big\}
    \big).
\]
\end{dfn}

Following~\cite{Jh0},
in the notation of Definition~\ref{D_3326_JhNtn}
it is easy to check that $d (T) = 0$ \ifo{}  $T$ is a \hm.

We recall that if $E,$ $F,$ and~$G$
are Banach spaces,
and $b \colon E \times F \to G$ is a bilinear map,
then $\| b \|$ is the least constant~$M$
such that $\| b (\xi, \et) \| \leq M \| \xi \| \cdot \| \et \|$
for all $\xi \in E$ and $\et \in F.$
We next recall the following estimate.

\begin{lem}[Proposition~1.1 of~\cite{Jh0}]\label{L_3326_JhProp}
Adopt the notation of Definition~\ref{D_3326_JhNtn}.
Then $\| T^{\vee} \| \leq \big( 1 + d (T) + 2 \| T \| \big) d (T).$
\end{lem}

The following result,
but with $\dt$ depending on~$D,$
is Corollary 3.2 of~\cite{Jh0}.

\begin{lem}\label{L_3325_fdsj}
Let $A$ be a Banach algebra which is isomorphic,
as a complex algebra,
to a finite direct sum of full matrix algebras.
The for every $\ep > 0$ and $M \in [0, \infty)$
there is $\dt > 0$ such that,
whenever $D$ is a Banach algebra and $T \colon A \to D$
is a linear map with $\| T \| \leq M$
and $\| T^{\vee} \| < \dt,$
then there is a \hm{} $\ph \colon A \to D$
such that $\| \ph - T \| < \ep.$
\end{lem}

Example~1.5 of~\cite{Jh0}
shows that it is not possible to take $\dt$ to be independent
of~$M,$
even for $A = \C$ (with nonunital \hm{s}).

\begin{proof}[Proof of Lemma~\ref{L_3325_fdsj}]
Let $A$ be as in the hypotheses,
let $\ep > 0,$
and let $M \in [0, \infty).$
Suppose that the conclusion fails.
Then for every $n \in \N$ there is
a Banach algebra $D_n$ and a linear map $T_n \colon A \to D_n$
such that $\| T_n \| \leq M$
and
$\| T^{\vee} \| < \frac{1}{n},$
but no \hm{} $\ph \colon A \to D_n$
such that $\| \ph - T \| < \ep.$

Let $D$ be the Banach algebra of all
bounded sequences $b = (b_n)_{n \in \N} \in \prod_{n = 1}^{\infty} D_n,$
with the pointwise operations
and the norm $\| b \| = \sup_{n \in \N} \| b_n \|.$
For $n \in \N,$ let $\pi_n \colon D \to D_n$
be given by $\pi_n (b) = b_n.$
Corollary 3.2 of~\cite{Jh0}
provides $\dt > 0$ such that whenever $T \colon A \to D$
is a linear map with $\| T \| \leq M$
and $\| T^{\vee} \| < \dt,$
then there is a \hm{} $\ph \colon A \to D$
such that $\| \ph - T \| < \ep.$
Choose $n \in \N$ such that $\frac{1}{n} < \dt.$
Define $T \colon A \to D$
by
\[
T (x)
 = (0, \, 0, \, \ldots, \, 0, \, T_n (x), \, T_{n + 1} (x), \, \ldots )
\]
for $x \in A.$
Then $\| T \| \leq M$
and $\| T^{\vee} \| < \frac{1}{n} < \dt.$
So there is a \hm{} $\ph \colon A \to D$
such that $\| \ph - T \| < \ep.$
The map $\pi_n \circ \ph \colon A \to D_n$
is a \hm{} such that $\| \pi_n \circ \ph - T_n \| < \ep,$
contradicting the assumption that no such \hm{} exists.
\end{proof}

The following simple lemma will be used several times.

\begin{lem}\label{L_3326_CloseUnit}
Let $A$ be a unital Banach algebra,
let $D$ be a Banach algebra,
and let $\ph, \ps \colon A \to D$
be \hm{s} such that $\| \ph - \ps \| < 1.$
If $\ph$ is nonzero then so is~$\ps,$
and if $D$ and $\ph$ are unital then so is~$\ps.$
\end{lem}

\begin{proof}
For the first statement,
$\ph (1)$ is a nonzero idempotent in~$D,$
whence $\| \ph (1) \| \geq 1.$
Therefore $\ps (1) \neq 0.$
For the second,
$\ps (1)$ is an idempotent in $D$
with $\| \ps (1) - 1 \| < 1.$
Therefore $\ps (1)$ is an invertible idempotent,
so $\ps (1) = 1.$
\end{proof}

\begin{lem}\label{L_3325_HomEst}
Let $A$ be a Banach algebra which is isomorphic,
as a complex algebra,
to a finite direct sum of full matrix algebras.
Let $\ep > 0,$
and let $M \in [1, \infty).$
Then there is $\dt_{A, \ep, M} > 0$
such that whenever $D$ is a Banach algebra,
$C \subset D$ is a subalgebra,
$\ph \colon A \to D$
is a \hm{} such that $\| \ph \| \leq M,$
and $S \colon A \to C$ is a linear map
such that $\| S \| \leq M$ and $\| S - \ph \| < \dt_{A, \ep, M},$
then there is a \hm{} $\ps \colon A \to C$
such that $\| \ps - \ph \| < \ep.$
If $D,$ $C,$ and $\ph$ are unital,
then we may require that $\ps$ be unital.
\end{lem}

\begin{proof}
\Wolog, $\ep < 1.$
Apply Lemma~\ref{L_3325_fdsj}
with $\frac{\ep}{2}$ in place of~$\ep$
and with $A$ and $M$ as given,
obtaining $\dt_0 > 0.$
Set
\[
\dt = \min \left( 1, \, \frac{\ep}{2}, \, \frac{\dt_0}{2 (1 + M)} \right).
\]
Let $\ph$ and $S$ be as in the hypotheses, with $\dt_{A, \ep, M} = \dt.$
Lemma~\ref{L_3326_JhProp} implies that
\[
\| S^{\vee} \|
 < \dt (1 + \dt + 2 M)
 \leq \dt (2 + 2 M)
 \leq \dt_0.
\]
Therefore there exists a \hm{} $\ps \colon A \to C$
such that $\| \ps - S \| < \frac{\ep}{2}.$
Since $\| S - \ph \| < \dt \leq \frac{\ep}{2},$
it follows that $\| \ph - \ps \| < \ep.$

It remains to prove,
under the conditions in the last sentence, that $\ps$ is unital.
Since $\ep < 1,$
this follows from Lemma~\ref{L_3326_CloseUnit}.
\end{proof}

\begin{lem}\label{L_3325_DLimPert}
Let $A$ be a Banach algebra which is isomorphic,
as a complex algebra,
to a finite direct sum of full matrix algebras.
Let $D$ be a Banach algebra,
and let $D_0 \subset D_1 \subset \cdots \subset D$
be an increasing sequence of subalgebras such that
$D = {\overline{\bigcup_{n = 0}^{\infty} D_n}}.$
Then for every \hm{} $\ph \colon A \to D$
and every $\ep > 0,$
there is $n \in \Nz$ and a \hm{} $\ps \colon A \to D$
such that $\| \ps - \ph \| < \ep$
and $\ps (A) \subset D_n.$
Moreover, if $D$ and $\ph$ are unital,
and $D_n$ is unital for $n \in \Nz,$
then $\ps$ can be chosen to be unital.
\end{lem}

\begin{proof}
Apply Lemma~\ref{L_3325_HomEst}
with $A$ as given,
with $\| \ph \| + 1$ in place of~$M,$
and with $\ep$ as given,
obtaining $\dt = \dt_{A, \ep, 1 + \| \ph \|} > 0.$
Set $N = \dim (A).$
Choose a basis $(x_k)_{k = 1, 2, \ldots, N}$ for $A$ consisting
of elements $x_k \in A$
such that $\| x_k \| = 1$ for $k = 1, 2, \ldots, N.$
Define a bijection $S \colon l^1 ( \{ 1, 2, \ldots, N \} ) \to A$
by identifying $l^1 ( \{ 1, 2, \ldots, N \} )$ with $\C^N$
and setting
\[
S (\af_1, \af_2, \ldots, \af_N) = \sum_{k = 1}^N \af_k x_k
\]
for $\af_1, \af_2, \ldots, \af_N \in \C.$
Set
$\dt_0 = \frac{1}{2} \| S^{-1} \|^{- 1} \min ( 1, \dt ).$

Choose $n \in \Nz$ such that there are
$b_1, b_2, \ldots, b_N \in D_n$
with $\| b_k - \ph (x_k) \| < \dt_0$
for $k = 1, 2, \ldots, N.$
Let $T \colon A \to D_n$
be the unique linear map such that $T (x_k) = b_k$
for $k = 1, 2, \ldots, N.$
We claim that
$\| T - \ph \| < 2 \| S^{-1} \| \dt_0.$
%
% \begin{equation}\label{Eq_3325_EstTph}
% \| T - \ph \| < \| S^{-1} \| \dt_0.
% \end{equation}
%
To see this,
let $x \in A.$
Choose $\af_1, \af_2, \ldots, \af_N \in \C$ such that
$x = \sum_{k = 1}^N \af_k x_k.$
Then $S^{-1} (x) = (\af_1, \af_2, \ldots, \af_N),$
so
$\| (\af_1, \af_2, \ldots, \af_N) \|_1 \leq \| S^{-1} \| \cdot \| x \|.$
Now
\[
\| T (x) - \ph (x) \|
  \leq \sum_{k = 1}^N | \af_k | \cdot  \| T (x_k) - \ph (x_k) \|
  \leq \sum_{k = 1}^N | \af_k | \dt_0
  \leq \| S^{-1} \| \cdot \dt_0 \cdot \| x \|.
\]
So $\| T - \ph \| \leq \| S^{-1} \| \dt_0 < 2 \| S^{-1} \| \dt_0.$

Since $2 \| S^{-1} \| \dt_0 \leq 1,$
the claim implies that $\| T \| < \| \ph \| + 1.$
Since $2 \| S^{-1} \| \dt_0 \leq \dt,$
we get $\| T - \ph \| < \dt,$
so the choice of~$\dt,$
with $D_n$ in place of~$C$ in Lemma~\ref{L_3325_HomEst},
provides a \hm{} $\ps \colon A \to D$
such that $\| \ps - \ph \| < \ep$
and $\ps (A) \subset D_n,$
which is unital under the conditions in the last sentence.
\end{proof}

The following lemma is the key step of our argument.
For any $M,$ $\gm_0,$ $d_0,$ $d,$ and~$p,$
if $\gm$ is sufficiently large
and there is a nonzero \hm{}
$\ph \colon \MP{d_0}{p, \gm_0} \to \MP{d}{p, \gm} \otimes_p A,$
then there is a nonzero \hm{}
from $\MP{d_0}{p, \gm_0}$
whose range is in $\C \cdot 1 \otimes_p A$
and whose norm is nearly the same as that of~$\ph.$

\begin{lem}\label{L_3325_BigHom}
Let $p \in [1, \infty),$
let $d_0, d \in \N,$
and let $\gm_0 \in [1, \infty).$
Let $M \in [1, \infty)$ and let $\ep > 0.$
Then there is $R \in [0, \infty)$
such that whenever $\gm \in [R, \infty),$
the following holds.
Let $\YCN$ be a \sfm,
and let $A \subset L (L^p (Y, \nu))$ be a closed subalgebra.
Let
$\ph \colon \MP{d_0}{p, \gm_0} \to
      \MP{d}{p, \gm} \otimes_p A$
be a nonzero \hm{} such that $\| \ph \| \leq M.$
Then there is a nonzero \hm{}
$\ps \colon \MP{d_0}{p, \gm_0} \to A$
such that $\| \ps \| \leq M + \ep.$
\end{lem}

\begin{proof}
We use Notation~\ref{N_3126_Qld} throughout.
\Wolog, $\ep < 1.$
Let $\dt > 0$
be the constant of Lemma~\ref{L_3325_HomEst}
obtained using $\ep$ and $M$ as given and with $\MP{d_0}{p}$
in place of~$A.$
Set $R = 2 d^2 M \gm_0 / \dt.$
Now let $\gm,$ $A,$ and $\ph$ be as in the statement of the lemma.
Define
\[
C = \MP{d}{p, \gm} \otimes_p A
\andeqn
C_0 = \MP{d}{p} \otimes_p A = \MP{d}{p, 1} \otimes_p A,
\]
with norms $\| \cdot \|_{p, \gm}$
and $\| \cdot \|_{p, 1}.$
We will also need the maps,
defined as in Notation~\ref{N_3126_Qld},
\[
\kp_{d, A}^{p, 1, \gm} \colon C \to C_0,
\,\,\,\,\,\,
\kp_{d_0}^{p, 1, \gm_0} \colon \MP{d_0}{p, \gm_0} \to \MP{d_0}{p},
\andeqn
\kp_{d_0}^{p, \gm_0, 1}
%  = \big(\kp_{d_0}^{p, 1, \gm_0})^{-1}
   \colon
    \MP{d_0}{p} \to \MP{d_0}{p, \gm_0}.
\]
% $\kp_{d, A}^{p, 1, \gm} \colon C \to C_0,$
% $\kp_{d_0}^{p, 1, \gm_0} \colon \MP{d_0}{p, \gm_0} \to \MP{d_0}{p},$
% and
% $\kp_{d_0}^{p, \gm_0, 1}
%  = \big(\kp_{d_0}^{p, 1, \gm_0})^{-1} \colon
%     \MP{d_0}{p} \to \MP{d_0}{p, \gm_0}.$

Define $T \colon \MP{d_0}{p, \gm_0} \to
      \MP{d}{p, \gm} \otimes_p A$
by
\[
T (x) = \sum_{l = 1}^{d}
       (e_{l, l} \otimes 1) \ph (x) (e_{l, l} \otimes 1)
\]
for $x \in \MP{d_0}{p, \gm_0}.$
For $l = 1, 2, \ldots, d,$
there is a linear map
$T_l \colon \MP{d_0}{p, \gm_0} \to A$
such that
\[
(e_{l, l} \otimes 1) \ph (x) (e_{l, l} \otimes 1)
  = e_{l, l} \otimes T_l (x)
\]
for all $x \in \MP{d_0}.$
Since $\| e_{l, l} \|_{p, \gm} = 1,$
we have $\| T_l \| \leq \| \ph \|$ for $l = 1, 2, \ldots, d,$
so Lemma~\ref{L_3126_QBasic}(\ref{L_3126_QBasic_Diag})
implies that $\| T \| \leq \| \ph \|.$

We now claim that
$\big\| \kp_{d, A}^{p, 1, \gm} \circ T \circ \kp_{d_0}^{p, \gm_0, 1}
    - \kp_{d, A}^{p, 1, \gm} \circ \ph \circ \kp_{d_0}^{p, \gm_0, 1}
             \big\| < \dt.$
To prove the claim, let $x \in \MP{d_0}{p}$ satisfy $\| x \|_p \leq 1.$
There are elements $a_{l, m} \in A$
for $l, m = 1, 2, \ldots, d$
such that
\[
\big( \ph \circ \kp_{d_0}^{p, \gm_0, 1} \big) (x)
 = \sum_{l, m = 1}^{d} e_{l, m} \otimes a_{l, m}.
\]
For $l \neq m,$
using
Lemma~\ref{L_3126_QBasic}(\ref{L_3126_QBasic_Nejk})
at the first step,
$\| e_{l, l} \otimes 1 \|_{p, \gm}
  = \| e_{m, m} \otimes 1 \|_{p, \gm} = 1$
at the third step,
and Lemma~\ref{L_3126_QBasic}(\ref{L_3126_QBasic_Norm})
at the fifth step,
we get
\begin{align*}
\gm \| a_{l, m} \|
 & = \| e_{l, m} \otimes a_{l, m} \|_{p, \gm}
   = \| (e_{l, l} \otimes 1) \ph (x) (e_{m, m} \otimes 1) \|_{p, \gm}
     \\
 & \leq \| \ph (x) \|_{p, \gm}
   \leq \| \ph \| \cdot \| x \|_{p, \gm_0}
   \leq \| \ph \| \gm_0 \| x \|_{p}
   \leq M \gm_0.
\end{align*}
So
\[
\| a_{l, m} \| \leq M \gm_0 \gm^{-1}
\andeqn
\| e_{l, m} \otimes a_{l, m} \|_{p, 1}
  = \| a_{l, m} \|
  \leq M \gm_0 \gm^{-1}.
\]
Now
\begin{align*}
\big\| \big( \kp_{d, A}^{p, 1, \gm} \circ T \circ
       \kp_{d_0}^{p, \gm_0, 1} \big) (x)
  - \big( \kp_{d, A}^{p, 1, \gm} \circ \ph \circ
           \kp_{d_0}^{p, \gm_0, 1} \big) (x) \big\|_{p, 1}
& \leq 
       \sum_{l = 1}^{d} \sum_{m \neq l}
          \| e_{l, m} \otimes a_{l, m} \|_{p, 1}
        \\
& \leq d^2 M \gm_0 \gm^{-1}.
\end{align*}
Thus
\[
\big\| \kp_{d, A}^{p, 1, \gm} \circ T \circ \kp_{d_0}^{p, \gm_0, 1}
    - \kp_{d, A}^{p, 1, \gm} \circ \ph \circ \kp_{d_0}^{p, \gm_0, 1}
                \big\|
  \leq \frac{d^2 M \gm_0}{\gm}
  \leq \frac{d^2 M \gm_0}{R}
  = \frac{\dt}{2}
  < \dt.
\]
The claim is proved.

Let $D \subset C_0 = \MP{d}{p} \otimes_p A$ be the subalgebra
consisting of all diagonal matrices in $M_d (A).$
We algebraically identify $D$ with $\bigoplus_{l = 1}^{d} A$
via the map
$(a_1, a_2, \ldots, a_{d})
 \mapsto \sum_{l = 1}^{d} e_{l, l} \otimes a_l$
from $\bigoplus_{l = 1}^{d} A$ to~$C_0.$
Equip $\bigoplus_{l = 1}^{d} A$ with the norm
\[
\| (a_1, a_2, \ldots, a_{d}) \|
 = \max \big( \| a_1 \|, \, \| a_2 \|, \, \ldots, \, \| a_{d} \| \big)
\]
for $a_1, a_2, \ldots, a_{d} \in A.$
By Lemma~\ref{L_3126_QBasic}(\ref{L_3126_QBasic_Diag}),
the identification of $\bigoplus_{l = 1}^{d} A$ with $D$
is then isometric.
By construction,
the range of $\kp_{d, A}^{p, 1, \gm} \circ T$ is contained in~$D.$
The claim above and the choice of~$\dt$
provide a \hm{} $\ps_0 \colon \MP{d_0}{p} \to D$
such that
$\big\| \ps_0 - \kp_{d, A}^{p, 1, \gm} \circ \ph
          \circ \kp_{d_0}^{p, \gm_0, 1} \big\|
  < \ep.$
Since $\kp_{d, A}^{p, 1, \gm} \circ \ph \circ \kp_{d_0}^{p, \gm_0, 1}$
is a
nonzero \hm{} and $\ep < 1,$
it follows from Lemma~\ref{L_3326_CloseUnit} that $\ps_0$ is nonzero.

Set $\ps_1 = \ps_0 \circ \kp_{d_0}^{p, 1, \gm_0}.$
Using
$\kp_{d_0}^{p, \gm_0, 1} \circ \kp_{d_0}^{p, 1, \gm_0}
  = \id_{\MP{d_0}{p, \gm_0}}$
at the first step,
and
$\big\| \kp_{d_0}^{p, 1, \gm_0} \big\| = 1$
(from Lemma~\ref{L_3126_QBasic}(\ref{L_3126_QBasic_Inv}))
at the second step,
we get
\[
\big\| \ps_1 - \kp_{d, A}^{p, 1, \gm} \circ \ph \big\|
 \leq \big\| \ps_0
     - \kp_{d, A}^{p, 1, \gm} \circ \ph \kp_{d_0}^{p, \gm_0, 1} \big\|
       \cdot \big\| \kp_{d_0}^{p, 1, \gm_0} \big\|
  < \ep.
\]
So, using $\big\| \kp_{d, A}^{p, 1, \gm} \big\| = 1$
(from Lemma~\ref{L_3126_QBasic}(\ref{L_3126_QBasic_Inv})),
we get
\[
\| \ps_1 \|
 < \big\| \kp_{d, A}^{p, 1, \gm} \circ \ph \big\| + \ep
 \leq \| \ph \| + \ep.
\]
For $l = 1, 2, \ldots, d$
and $a = \sum_{k = 1}^{d} e_{k, k} \otimes a_k \in D,$
define $\pi_l (a) = a_l.$
The formula defines a contractive \hm{} $\pi_l \colon D \to A.$
Choose $l$ such that $(\pi_l \circ \ps_1) (1) \neq 0.$
Then $\ps = \pi_l \circ \ps_1 \colon \MP{d_0}{p} \to A$
is a nonzero \hm{} such that $\| \ps \| < \| \ph \| + \ep.$
This completes the proof.
\end{proof}

Recall that for a sequence $d = (d (1), \, d (2), \, \ldots)$
in $\{ 2, 3, \ldots \}$ and $n \in \Nz,$
we defined $r_d (n) = d (1) d (2) \cdots d (n).$

\begin{lem}\label{L_3325_OneStep}
Let $p \in [1, \infty),$
let $d = (d (1), \, d (2), \, \ldots)$
be a sequence in $\{ 2, 3, \ldots \},$
and let $\af \in [1, \infty).$
Then for every $M \in [1, \infty)$ and $l \in \N,$
there is a nondecreasing sequence $\bt = (\bt (1), \, \bt (2), \, \ldots)$
in $[1, \infty)$
such that, whenever $\gm = (\gm (1), \, \gm (2), \, \ldots)$
is a nondecreasing sequence in $[1, \infty)$
such that $\gm (j) \geq \bt (j)$ for $j = l, \, l + 1, \, \ldots,$
whenever $B_d^{p, \gm}$ is as in Notation~\ref{N_3325_InfTPN},
and whenever $\ph \colon \MP{r_d (l)}{p, \af} \to B_d^{p, \gm}$
is a nonzero \hm,
then $\| \ph \| > M.$
\end{lem}

\begin{proof}
For $m = l, l + 1, \ldots,$
apply Lemma~\ref{L_3325_BigHom}
with $r_d (l)$ in place of~$d_0,$
with $d (m)$ in place of~$d,$
with $\af$ in place of~$\gm_0,$
with $2^{- (m - l + 1)}$ in place of~$\ep,$
and with $M + 1 + 2^{- (m - l + 1)}$ in place of~$M.$
Let $\bt_0 (m)$ be the resulting value of~$R.$
For $m = 1, 2, \ldots, l - 1$
set $\bt (m) = 1,$
and for $m = l, l + 1, \ldots$
set
\[
\bt (m)
 = \max \big( \bt_0 (l), \, \bt_0 (l + 1), \, \ldots, \, \bt_0 (m) \big).
\]

Now let $\gm$ be a nondecreasing sequence in $[1, \infty)$
such that $\gm (j) \geq \bt (j)$ for $j = l, \, l + 1, \ldots.$
Suppose that there is a nonzero \hm{}
$\ph \colon \MP{r_d (l)}{p, \af} \to B_d^{p, \gm}$
such that $\| \ph \| \leq M.$
Let $B = B_d^{p, \gm}, B_0, B_1, \ldots$
be as in Notation~\ref{N_3325_InfTPN}.
Lemma~\ref{L_3325_DLimPert}
provides $n \in \Nz$ and a \hm{}
$\ps \colon \MP{r_d (l)}{p, \af} \to B_n$
such that $\| \ps - \ph \| < 1.$
Then $\ps \neq 0$ by Lemma~\ref{L_3326_CloseUnit}.
Therefore $n \geq l.$
Also, $\| \ps \| \leq M + 1 \leq M + 1 + 2^{- (n - l + 1)}.$

Set $\ps_n = \ps.$
Using choice of $\bt_0 (n),$
the inequality
$\gm (n) \geq \bt (n) \geq \bt_0 (n),$
and the tensor product decomposition
$B_n = B_{n - 1} \otimes_p \MP{d (n)}{p, \gm (n)},$
we get a nonzero \hm{}
$\ps_{n - 1} \colon \MP{r_d (l)}{p, \af} \to B_{n - 1}$
such that
\[
\| \ps_{n - 1} \|
 \leq \big( M + 1 + 2^{- (n - l + 1)} \big) + 2^{- (n - l + 1)}
 = M + 1 + 2^{- (n - l)}.
\]
Similarly,
there is now a nonzero \hm{}
$\ps_{n - 2} \colon \MP{r_d (l)}{p, \af} \to B_{n - 2}$
such that
\[
\| \ps_{n - 2} \|
 \leq M + 1 + 2^{- (n - l - 1)}.
\]
Proceed inductively.
We eventually find a nonzero \hm{}
$\ps_{l - 1} \colon \MP{r_d (l)}{p, \af} \to B_{l - 1}$
such that
\[
\| \ps_{l - 1} \|
 \leq M + 1 + 1.
\]
Since $B_{l - 1} \cong M_{r_d (l - 1)}$
and $r_d (l - 1) < r_d (l),$
this is a contradiction.
\end{proof}

\begin{lem}\label{L_3325_NoHom}
Let $p \in [1, \infty),$
let $d = (d (1), \, d (2), \, \ldots)$
be a sequence in $\{ 2, 3, \ldots \},$
and let $\af = (\af (1), \, \af (2), \, \ldots)$
be a nondecreasing sequence in $[1, \infty).$
Then there is a nondecreasing sequence
$\bt = (\bt (1), \, \bt (2), \, \ldots)$ in $[1, \infty)$
such that,
whenever $l \in \N$ and $\gm = (\gm (1), \, \gm (2), \, \ldots)$
is a nondecreasing sequence in $[1, \infty)$
with $\gm (j) \geq \bt (j)$ for $j = l, \, l + 1, \ldots,$
and $B_d^{p, \af}$ and $B_d^{p, \gm}$
are as in Notation~\ref{N_3325_InfTPN},
then there is no nonzero \ct{} \hm{} from
$B_d^{p, \af}$ to $B_d^{p, \gm}.$
\end{lem}

\begin{proof}
For each $m \in \N,$
apply Lemma~\ref{L_3325_OneStep}
with $d$ as given,
with $m$ in place of~$l,$
with $\prod_{k = 1}^{m} \af (k)$ in place of~$\af,$
and with $M = m.$
Call the resulting sequence
$\bt_m = \big( \bt_m (1), \, \bt_m (2), \, \ldots \big).$
Define
\[
\bt (m) = \max \big( \bt_m (1), \bt_m (2), \ldots, \bt_m (m) \big)
\]
for $m \in \N.$
Clearly $\bt = (\bt (1), \, \bt (2), \, \ldots)$
is a nondecreasing sequence in $[1, \infty).$

Now let $l \in \N$
and let $\gm = (\gm (1), \, \gm (2), \, \ldots)$
be a nondecreasing sequence in $[1, \infty)$
such that $\gm (j) \geq \bt (j)$ for $j = l, \, l + 1, \, \ldots.$
Suppose that $\ph \colon B_d^{p, \af} \to B_d^{p, \gm}$
is a nonzero \ct{} \hm.
Choose $m \in \N$ such that $m > \max (l, \| \ph \|).$
Set $\et = \prod_{k = 1}^{m} \af (k).$
Let
\[
\ps \colon M_{r_d (m)}
  \to M_{d (1)} \otimes M_{d (2)}
           \otimes \cdots \otimes M_{d (m)}
\]
be an isomorphism which sends standard matrix units
to tensor products of standard matrix units.
In particular,
the image of the diagonal subalgebra of $M_{r_d (m)}$
is the tensor product of the diagonal subalgebras
of the algebras $M_{d (k)}$
for $k = 1, 2, \ldots, m.$

We claim that $\ps$ is a contractive \hm{}
\[
\ps \colon \MP{r_d (m)}{p, \et}
  \to \MP{d (1)}{p, \af (1)} \otimes_p \MP{d (2)}{p, \af (2)}
           \otimes_p \cdots \otimes_p \MP{d (m)}{p, \af (m)}.
\]
To see this,
apply Lemma~\ref{L_3406_TPSys} repeatedly to form the tensor product
$S = (I, s, f)$ of the systems $S_{ d (k), \af (k) }$
of Notation~\ref{N_3126_Qld}
for $k = 1, 2, \ldots, m.$
Also let $S_{r_d (m), \et}$ be as in Notation~\ref{N_3126_Qld}.
For $i \in I = \prod_{k = 1}^m I_{d (k), \af (k)},$
the matrix $s (i)$ is diagonal and satisfies
\begin{align*}
\| s (i) \|
 & = \big\| s_{d (1), \af (1)} (i_1) \otimes s_{d (2), \af (2)} (i_2)
      \otimes \cdots \otimes s_{d (m), \af (m)} (i_m) \big\|
    \\
 & = \prod_{k = 1}^{m} \| s_{d (k), \af (k)} (i_k) \|
   \leq \prod_{k = 1}^{m} \af (k)
   = \et.
\end{align*}
Moreover, all its diagonal entries are real and at least~$1.$
Therefore
\[
\ran (s) \subset K_{r_d (m), \et}
         = {\overline{\ran (s_{r_d (m), \et})}},
\]
so $\ps$ is contractive
by Lemma \ref{L_3405_Basics}(\ref{L_3405_Basics_Sub}).

As in Notation~\ref{N_3325_InfTPN},
we have an isometric inclusion
\[
\MP{d (1)}{p, \af (1)} \otimes_p \MP{d (2)}{p, \af (2)}
           \otimes_p \cdots \otimes_p \MP{d (m)}{p, \af (m)}
       \to B_d^{p, \af}.
\]
Call it~$\io.$
Since $\| \ps \| \leq 1,$
the map
\[
\sm = \ph \circ \io \circ \ps \colon
 \MP{r_d (m)}{p, \et} \to B_d^{p, \gm}
\]
is a nonzero \hm{} such that
$\| \sm \| \leq \| \ph \|.$
Since $m \geq l,$
for $k = m, m + 1, \ldots$
we have $\gm (k) \geq \bt (k) \geq \bt_m (k).$
The choice of $\bt_m$
therefore implies that $\| \sm \| \geq m.$
Since $m > \| \ph \|,$
this is a contradiction.
\end{proof}

\begin{thm}\label{T_3326_Uctbl}
Let $p \in [1, \infty),$
and let $d = (d (1), \, d (2), \, \ldots)$
be a sequence in $\{ 2, 3, \ldots \}.$
Let $\Om$ be the set of countable ordinals.
Then there exists a family $(\gm_{\kp})_{\kp \in \Om}$
of nondecreasing sequences in $[1, \infty)$
such that whenever $\kp, \ld \in \Om$
satisfy $\kp < \ld,$
then there is no nonzero \ct{} \hm{} from
$B_d^{p, \gm_{\kp}}$ to $B_d^{p, \gm_{\ld}},$
but there is a unital \ct{} \hm{}
from $B_d^{p, \gm_{\ld}}$ to $B_d^{p, \gm_{\kp}}$
with dense range.
\end{thm}

\begin{proof}
We construct $(\gm_{\kp})_{\kp \in \Om}$
by transfinite induction on~$\kp.$
We start by taking $\gm_0 = (1, 1, \ldots).$

Suppose now $\kp \in \Om$
and we have the sequences $\gm_{\mu}$
for $\mu < \kp.$
First suppose that $\kp$ is a successor ordinal,
that is, there is $\ld \in \Om$ such that $\kp = \ld + 1.$
Apply Lemma~\ref{L_3325_NoHom}
with $p$ and $d$ as given
and with $\gm_{\ld}$ in place of~$\af,$
obtaining a nondecreasing sequence
$\bt = (\bt (1), \, \bt (2), \, \ldots)$ in $[1, \infty).$
Set $\gm_{\kp} (n) = \max (\bt (n), \, \gm_{\ld} (n) )$
for $n \in \N.$
Then there is no nonzero \ct{} \hm{} from
$B_d^{p, \gm_{\ld}}$ to $B_d^{p, \gm_{\kp}}.$
However,
Lemma~\ref{L_3326_DomRel} provides a \ct{} unital \hm{}
$\ps \colon B_d^{p, \gm_{\kp}} \to B_d^{p, \gm_{\ld}}$
with dense range.
For any ordinal $\mu < \ld,$
the induction hypothesis provides a \ct{} unital \hm{}
from $B_d^{p, \gm_{\ld}}$ to $B_d^{p, \gm_{\mu}}$
with dense range,
giving a \ct{} unital \hm{}
from $B_d^{p, \gm_{\kp}}$ to $B_d^{p, \gm_{\mu}},$
again with dense range.
Suppose now $\mu \in \Om$
satisfies $\mu < \ld$ and that there is
a nonzero \ct{} \hm{}
$\ph \colon B_d^{p, \gm_{\mu}} \to B_d^{p, \gm_{\kp}}.$
Then $\ps \circ \ph$ is a nonzero \ct{} \hm{}
from $B_d^{p, \gm_{\mu}}$ to $B_d^{p, \gm_{\ld}}.$
We have contradicted the induction hypothesis,
and the successor ordinal case of the induction step is complete.

Now suppose that $\kp$ is a limit ordinal.
Let $(\ld_n)_{n \in \N}$ be an enumeration
of $\{ \ld \in \Om \colon \ld < \kp \}$
(in arbitrary order).
For $n \in \N,$
apply Lemma~\ref{L_3325_NoHom}
with $p$ and $d$ as given
and with $\gm_{\ld_n}$ in place of~$\af,$
obtaining a nondecreasing sequence
$\bt_n = (\bt_n (1), \, \bt_n (2), \, \ldots)$ in $[1, \infty).$
Now recursively define
$\gm_{\kp} (1) = \max ( \bt_1 (1), \, \gm_{\ld_1} (1) )$ and
\[
\gm_{\kp} (n)
 = \max \big( \gm_{\kp} (n - 1), \,
              \bt_n (1), \, \bt_n (2), \, \ldots, \, \bt_n (n), \,
              \gm_{\ld_1} (n), \, \gm_{\ld_2} (n), \, \ldots, \,
              \gm_{\ld_n} (n) \big)
\]
for $n \geq 2.$
We verify that $\gm_{\kp}$ satisfies the required conditions.
So let $\ld \in \Om$ satisfy $\ld < \kp.$
Choose $n \in \N$ such that $\ld_n = \ld.$
Then $\gm_{\kp} (k) \geq \gm_{\ld} (k)$
for $k = n, \, n + 1, \, \ldots,$
so Lemma~\ref{L_3326_DomRel} provides a \ct{} unital \hm{}
$B_d^{p, \gm_{\kp}} \to B_d^{p, \gm_{\ld}}$
with dense range.
Also,
$\gm_{\kp} (k) \geq \bt_n (k)$
for $k = n, \, n + 1, \, \ldots,$
so the choice of $\bt_n$ using Lemma~\ref{L_3325_NoHom}
ensures that there is no nonzero \ct{} \hm{} from
$B_d^{p, \gm_{\kp}}$ to $B_d^{p, \gm_{\ld}}.$
This completes the proof.
\end{proof}

\begin{pbm}\label{Pb_3121_Dist}
Let $p \in [1, \infty).$
For a given supernatural number~$N,$
give invariants which classify up to isomorphism
some reasonable class of
nonspatial $L^p$~UHF algebras of tensor product type,
such as those constructed using diagonal similarities.
\end{pbm}

\end{document}